\documentclass[12pt]{amsart}

\usepackage{amsmath,amsthm,amssymb,mathdots,fullpage,enumerate,color,ulem,marginnote}

\usepackage[all,cmtip]{xy}
\usepackage{tikz-cd}
\usepackage{xcolor}
\usepackage{bm}
\usepackage{enumitem}
\usepackage{upgreek}
\usepackage{mathtools}
\usepackage[active]{srcltx} 
\usepackage{mathrsfs}
\usepackage{hyperref}

\definecolor{couleur_cite}{rgb}{0.05,.4,0.05}
\definecolor{couleur_link}{rgb}{0.05,0.05,0.4}
\hypersetup{
bookmarksopen,bookmarksnumbered,
colorlinks,
linkcolor=couleur_link,citecolor=couleur_cite,
}

\newtheorem{theorem}{Theorem}[section]
\newtheorem{lemma}[theorem]{Lemma}
\newtheorem{prop}[theorem]{Proposition}

\newtheorem{cor}[theorem]{Corollary}
\newtheorem{definition}[theorem]{Definition}

\theoremstyle{remark}
\newtheorem{remark}{Remark}

\newcommand{\R}{\mathbb R}
\newcommand{\C}{\mathbb C}
\newcommand{\N}{\mathbb N}
\newcommand{\Z}{\mathbb Z}
\newcommand{\Q}{\mathbb Q}

\newcommand{\A}{\mathbb A}

\newcommand{\p}{\mathfrak p}
\newcommand{\g}{\mathfrak g}

\newcommand{\gk}{\mathfrak k}

\newcommand{\ga}{\mathfrak a}

\newcommand{\gn}{\mathfrak n}
\newcommand{\gh}{\mathfrak h}

\newcommand{\cO}{\mathcal O}

\newcommand{\cA}{\mathcal A}

\newcommand{\cM}{\mathcal M}
\newcommand{\cP}{\mathcal P}

\newcommand{\cC}{\mathcal C}

\newcommand{\cU}{\mathcal U}
\newcommand{\cZ}{\mathcal Z}

\newcommand{\bG}{{\bf G}}
\newcommand{\bP}{{\bf P}}
\newcommand{\bH}{{\bf H}}
\newcommand{\bM}{{\bf M}}
\newcommand{\bN}{{\bf N}}
\newcommand{\bR}{{\bf R}}
\newcommand{\bU}{{\bf U}}
\newcommand{\bL}{{\bf L}}
\newcommand{\bA}{{\bf A}}
\newcommand{\bV}{{\bf V}}

\newcommand{\tpl}{\mu_{\text{Pl}, \tau}}

\newcommand{\SL}{\text{SL} }
\newcommand{\SO}{\text{SO} }
\newcommand{\Oo}{\text{O} }
\newcommand{\Sp}{\text{Sp} }

\newcommand{\GL}{\text{GL} }

\newcommand{\Ad}{\textup{Ad}}

\newcommand{\tr}{\text{tr}}

\newcommand{\Hom}{\text{Hom}}

\newcommand{\reg}{T}

\newcommand{\be}{\begin{equation}}
\newcommand{\ee}{\end{equation}}
\newcommand{\bes}{\begin{equation*}}
\newcommand{\ees}{\end{equation*}}

\newcommand{\ba}{\begin{eqnarray}}
\newcommand{\ea}{\end{eqnarray}}
\newcommand{\bas}{\begin{eqnarray*}}
\newcommand{\eas}{\end{eqnarray*}}

\title{Concentration properties of theta lifts on orthogonal groups}

\author{Farrell Brumley}
\address{LAGA - Institut Galil\'ee\\
99 avenue Jean Baptiste Cl\'ement\\
93430 Villetaneuse\\
France}
\email{brumley@math.univ-paris13.fr}

\author{Simon Marshall}
\address{Department of Mathematics\\
University of Wisconsin -- Madison\\
480 Lincoln Drive\\
Madison\\
WI 53706, USA}
\email{marshall@math.wisc.edu}
\thanks{F.B. benefited from a CNRS IEA grant MELA-PERA during the final editing. S.M. is supported by NSF grants DMS-1902173 and DMS-1954479.}

\begin{document}

\begin{abstract}

Let $n>m\geqslant 1$ be integers with $n+m\geqslant 4$ even.  We prove the existence of Maass forms with large sup norms on anisotropic $\Oo(n,m)$, by combining a counting argument with a new period relation showing that a certain orthogonal period on $\Oo(n,m)$ distinguishes theta lifts from $\Sp_{2m}$.  This generalizes a method of Rudnick and Sarnak in the rank one case, when $m = 1$.  Our lower bound is naturally expressed as a ratio of the Plancherel measures for the groups $\Oo(n,m)$ and $\Sp_{2m}(\R)$, up to logarithmic factors, and strengthens the lower bounds of our previous paper \cite{BM} for such groups. In the case of odd-dimensional hyperbolic spaces, the growth exponent we obtain improves on a result of Donnelly, and is optimal under the purity conjecture of Sarnak.

\end{abstract}

\maketitle
\setcounter{tocdepth}{1}
\tableofcontents

\section{Introduction}
\label{sec:intro}

A well-known principle in quantum chaos suggests that, on a closed compact negatively curved Riemannian manifold, eigenfunctions should not manifest extreme localisation properties. Nevertheless, it is an emerging theme in arithmetic analysis that {\it intermediate} localisation properties can be realized in the setting of congruence manifolds, often for reasons related to functoriality. Indeed, one can access concentration behavior of eigenfunctions through their periods, and the non-vanishing of the latter, according to the relative Langlands program, distinguish functorial lifts from various source manifolds. A quantitative analysis can sometimes bootstrap non-vanishing to lower bounds.

In this paper we investigate atypical concentration behavior for sparse subsequences of arithmetic eigenfunctions on compact congruence manifolds associated with indefinite orthogonal groups. These spaces fall under the purview of our previous paper \cite{BM}, where they were shown to admit exceptional sequences of eigenfunctions having large sup norm, in the sense that $\|f_\lambda\|_\infty/\|f_\lambda\|_2=\Omega(\lambda^\delta)$ for some ineffective (in the logical sense) $\delta>0$, where $\lambda^2$ is the Laplacian eigenvalue. By contrast with the amplification method of \textit{loc. cit}, our methods in this article make direct use of the theta correspondence and automorphic period relations -- techniques which allow for refined (and effective) estimates.

More precisely, for integers $n\geqslant m\geqslant 1$, let $\mathbb{H}^{n,m}$ be the hyperbolic Grassmannian of signature $(n,m)$, parametrizing negative definite $m$-dimensional subspaces in $\R^{n,m}$. This is a simply connected symmetric space of non-compact type, of rank $m$ and dimension $d=nm$.  It is irreducible unless $n = m = 2$, and carries an action of $\Oo(n,m)$ by isometries. Let $\Gamma<\Oo(n,m)$ be a uniform lattice and consider the quotient $Y=\Gamma\backslash\mathbb{H}^{n,m}$. Then $Y$ is a compact locally symmetric space with non-positive sectional curvature. We shall always assume $\Gamma$ to be congruence, so that $Y$ enjoys additional arithmetic symmetries. These symmetries may be exploited through the algebra of Hecke correspondences, the joint eigenfunctions of which are called \textit{Hecke--Maass forms}. 

Our main result, Theorem \ref{sup-thm}, states that when $n>m\geqslant 1$ and $n+m\geqslant 4$ is even, one can find in every $O(1)$ spectral window a Hecke--Maass form on $Y$ whose $L^\infty$ norm grows by an explicit comparison of spectral density functions. Specifically, the minorant has the following form: it is the square-root of the ratio of the Plancherel measure on $\mathbb{H}^{n,m}$ with that of the Siegel upper half-space $\mathcal{H}^m$. The novelty here is the exact form of the lower bound in higher rank, which sheds light on the Sarnak purity conjecture \cite{Sa} when the spectral parameters approach root hyperplanes. Moreover, in rank 1, where the minorant is converted without loss in terms of the Laplacian eigenvalue $\lambda^2$, our lower bound yields $\|f_\lambda\|_\infty/\|f_\lambda\|_2=\Omega(\lambda^{(n-2)/2}/(\log\lambda)^{r/2})$, for an explicit constant $r\geqslant 1$. This improves upon the bound of Donnelly \cite{Donnelly} when $n$ is odd, and is best possible (up to logarithmic factors) under the purity conjecture.


The proof of Theorem \ref{sup-thm} is based on the foundational article \cite{RS}. The main idea can be expressed very simply. Namely, if $\mathbf{H}\subset\mathbf{G}$ are reductive groups defined over a number field, one can show the existence of Hecke--Maass forms on $\mathbf{G}$ whose $\mathbf{H}$-periods are larger than average, provided
\begin{enumerate}[label=(\alph*)]
\item\label{1sketch} one knows the size of the average $\mathbf{H}$-period;
\smallskip
\item\label{2sketch} one can show that $\mathbf{H}$-periods \textit{distinguish} a functorial lift from some other group $\mathbf{G}'$;
\smallskip
\item\label{3sketch} one can bound the number of contributing forms on the group $\mathbf{G}'$.
\end{enumerate}
Here, by `distinguish' we mean that the non-vanishing of an $\mathbf{H}$-period characterizes the image of a functorial lift from $\mathbf{G}'$. In favorable situations, the bound in \ref{3sketch} will show that this distinction is a rare occurrence, so \ref{2sketch} is in effect saying that most terms in \ref{1sketch} vanish. It follows that those forms which survive distinction must compensate for this by having unusually large $\mathbf{H}$ periods. When $\mathbf{H}$ is compact at infinity, this gives large point evaluations, whose average size is 1.

For Theorem \ref{sup-thm}, we apply this to $\bG=\Oo(V)$, for a non-degenerate anisotropic quadratic space $V$ over a totally real number field, with signature $(n,m)$ at a fixed real place and definite at the others. We take $\bH=\Oo(U^\perp)\times\Oo(U)$, where $U$ is an $m$-dimensional quadratic subspace that is negative definite at the distinguished real place, and $\mathbf{G}'=\Sp_{2m}$.  The functorial lift between $\bG'$ and $\bG$ will be the theta lift.

In a follow-up paper, we shall again exploit this machinery to study the exceptional behavior of higher dimensional periods, showing the existence of Hecke--Maass forms on hyperbolic congruence manifolds with large geodesic Fourier coefficients in the resonance range. 


\section{Main results}

Let $Y=\Gamma\backslash\mathbb{H}^{n,m}$ be a compact congruence quotient. Under certain conditions on the signature $(n,m)$, we would like to establish the existence of eigenfunctions on $Y$ whose point evaluation is quantifiably larger than average.

\subsection{Statement of main theorem}\label{sec:exceptional}

We begin by recalling the spectral parametrization of Maass forms on a general compact locally symmetric space of non-positive curvature, and the average size of values at fixed points.  

Let $G$ be a real connected\footnote{We remark that, while our main theorem concerns Maass forms on the disconnected group $\Oo(n,m)$, we show in Section \ref{sec:points} that the notion of spectral parameter on this group is the same as for the connected Lie group ${\rm SO}(n,m)^0$ when $n > m$.} semi-simple Lie group and $K$ a maximal compact subgroup. Let $\g=\mathfrak{p}\oplus\mathfrak{k}$ be the corresponding Cartan decomposition of the Lie algebra $\g$ of $G$. We make a choice of a maximal abelian subspace $\mathfrak{a}$ of $\mathfrak{p}$. Let $W=N_K(\mathfrak{a})/Z_K(\mathfrak{a})$ denote the Weyl group.

The quotient $S=G/K$ is a Riemannian globally symmetric space. Let $\Gamma\subset G$ be a uniform lattice and consider the locally symmetric space $\Gamma\backslash S$. Recall that a \textit{Maass form} $f\in L^2(\Gamma\backslash S)$ is a joint eigenfunction of $\mathbb{D}(S)$, the commutative algebra of left $G$-invariant differential operators on $S$. Given a Maass form $f$, let $\chi\in {\rm Hom}_{\C-\textrm{alg}}(\mathbb{D}(S),\C)$ be the associated eigencharacter, so that $Df =\chi(D)f$ for all $D\in\mathbb{D}(S)$. Via the Harish-Chandra isomorphism $\gamma_{\rm HC}:\mathbb{D}(S)\overset{\sim}{\longrightarrow} S(\mathfrak{a})^W$ onto the Weyl group invariants of the symmetric algebra of $\mathfrak{a}$ with complex coefficients, we have $\chi(D)=\nu(\gamma_{\rm HC}(D))$ for some $\nu\in\mathfrak{a}_\C^*$, unique up to $W$-conjugacy, called the \textit{spectral parameter} of $f$. 

Let $\{f_i\}_{i\geqslant 0}$ be an orthonormal basis of $L^2(\Gamma\backslash S)$ consisting of Maass forms, and denote by $\mu_i$ the spectral parameter of $f_i$. In Proposition \ref{local-Weyl} we show that there is $Q \geqslant 1$ such that for any finite collection of points $x_1,\ldots ,x_h\in \Gamma\backslash S$, and any $\nu\in i\mathfrak{a}^*$ of large enough norm depending on $x_i$, we have
\begin{equation}\label{eq:intro:local-Weyl}
\sum_{\|{\rm Im}\,\mu_i-\nu\|\leqslant Q}\bigg|\sum_{j=1}^h f_i(x_j)\bigg|^2\asymp \beta_S(\nu),
\end{equation}
where $\beta_S(\nu)$ is the spectral density function of $L^2(S)$. This is a form of the local Weyl law for the $\mathbb{D}(S)$-spectrum of compact locally symmetric spaces.

We deduce from \eqref{eq:intro:local-Weyl} that the average size of point evaluation (or, more generally, discrete periods) of eigenfunctions, in any given spectral ball of radius $Q$ with sufficiently regular center, is $\asymp 1$. In view of this average behavior, we will call \textit{exceptional} any Maass form whose sup norm grows by at least a power of its eigenvalue. Note that by dropping all but one term, we derive from \eqref{eq:intro:local-Weyl} an upper bound for the sup norm of a Maass form $f_\nu$ of spectral parameter $\nu$, of the form
\begin{equation}\label{Linfty-local}
\|f_\nu\|_\infty\ll \beta_S(\nu)^{1/2}\|f_\nu\|_2.
\end{equation}
We call this the \textit{local bound}, as the proof takes into consideration only local information about a given point. For the spaces we consider, it is not expected to be optimal, reflecting a far reaching delocalisation principle.

Having established average and extremal benchmarks in generality, our aim is to prove explicit intermediate growth estimates on the sup norms for exceptional Maass forms in the setting of \textit{congruence quotients} of hyperbolic Grassmannians $\Gamma\backslash\mathbb{H}^{n,m}$, or certain finite disjoint unions of such. (All of the above definitions and bounds apply equally well in this setting.) These disconnected manifolds arise as adelic double quotient spaces associated with the isometry group of a rational quadratic form, and may be expressed more classically as $Y=\bigcup_i\Gamma_i\backslash\mathbb{H}^{n,m}$, where $\Gamma_i$ runs over the genus class of $\Gamma$; see Section \ref{adelic2classical}. 

As was mentioned in the introduction, the underlying mechanism for such results is a functorial correspondence from a tertiary source manifold, which in our case will be a congruence covering of the Siegel variety ${\rm Sp}_{2m}(\Z)\backslash\mathcal{H}^m$, where $\mathcal{H}^m=\Sp_{2m}(\R)/{\rm U}(m)$ is the Siegel upper half-space. Note that the real rank of $\Sp_{2m}(\R)$ is the same as that of $\Oo(n,m)$, when $n\geqslant m$. This will allows us (see Section \ref{2examples} below for more details) to view a spectral parameter $\nu$ simultaneously for $\mathbb{H}^{n,m}$ and $\mathcal{H}^m$.

\begin{theorem}\label{sup-thm}
Fix integers $n>m\geqslant 1$ with $n+m\geqslant 4$ even. Let $F$ be a totally real number field, with fixed archimedean place $v_0$. Let $V$ be a non-degenerate anisotropic quadratic space over $F$, of signature $(n,m)$ at $v_0$ and positive definite at the other real places. Let $Y=\bigcup_i\Gamma_i\backslash\mathbb{H}^{n,m}$ be a congruence arithmetic locally symmetric space associated with $\mathrm{O}(V)$.

For sufficiently regular $\nu\in i\mathfrak{a}^*$ there exists an $L^2$-normalized Hecke--Maass form $f$ on $Y$ with spectral parameter $\nu+O(1)$ verifying
\[
\|f\|_\infty\gg \left(\log (3+\|\nu\|)\right)^{-r/2}\left(\beta_{\mathbb{H}^{n,m}}(\nu)/\beta_{\mathcal{H}^m}(\nu)\right)^{1/2},
\]
where $r=[F:\Q]m$. \end{theorem}

\begin{remark}
The condition that $n+m$ is even implies that $n \geqslant m+2$.  When $n > m+2$, the forms produced in Theorem \ref{sup-thm} are non-tempered at almost all places; see Section \ref{sec:unramified}. This fact is not, however, used in a direct way in the proof of Theorem \ref{sup-thm}; contrast this to the approach outlined in \cite[\S A.4]{BM}.
\end{remark}

\begin{remark}\label{remark2-intro} The logarithmic loss in the spectral parameter is due to a methodological short-cut we have adopted in the execution of Step \ref{3sketch} from the introduction. More precisely, in proving Proposition \ref{Weyl-upper-bd}, we use a partial trace formula, which allows for a simple treatment of the geometric side. A more sophisticated analysis would eliminate this loss. \end{remark}

\subsection{Explicating the lower bound}\label{2examples}

We now explicate the lower bound in Theorem \ref{sup-thm} by writing out in coordinates the two spectral density functions $\beta_S(\nu)$ appearing there, namely for $S=\mathbb{H}^{n,m}$ the hyperbolic Grassmannian and $S=\mathcal{H}^m$ the Siegel upper half-space. In fact, it will suffice to describe an approximation to $\beta_S(\nu)$, denoted here by $\tilde\beta_S(\nu)$, which is tight under the regularity assumption on $\nu$ in Theorem \ref{sup-thm}.

We begin with the case of $S=\mathbb{H}^{n,m}= \SO(n,m)^0/(\SO(n)\times\SO(m))$. For $n\geqslant m\geqslant 0$, we fix the signature $(n,m)$ quadratic form on $\R^m\oplus\R^m\oplus \R^{n-m}$ given by
\[
\sum_{i=1}^m (x_iy_i'+y_ix_i')+\sum_{i=1}^{n-m}z_iz_i'.
\]
The Lie algebra $\mathfrak{so}(n,m)$ of $\SO(n,m)^0$ consists of block matrices in $\mathfrak{sl}_{n+m}(\R)$ of the form
\begin{equation}\label{blocks}
\begin{pmatrix} A & B& C\\ D & -^t\!A &E \\ -^tE& -^tC&F \end{pmatrix},
\end{equation}
where $A\in \mathfrak{gl}_m(\R)$, $B,D\in \mathfrak{so}(m)$, $C,E\in M_{m,n-m}(\R)$, $F\in\mathfrak{so}(n-m)$. The $-1$ eigenspace $\mathfrak{p}$ of the Cartan involution ``minus transpose"  consists of symmetric matrices of the form \eqref{blocks}. A maximal abelian subspace of $\mathfrak{p}$ can be taken to be $\mathfrak{a}=\{{\rm diag}(A,-\!A,0)\}$, with $A$ diagonal.  We let $\{ \epsilon_i : i=1,\ldots ,m\}$ be the standard basis of $\mathfrak{a}^*$ given by evaluation of the diagonal entries of $A$.  Then, if $n>m$, 
\begin{equation*}
\Sigma=\{\pm\epsilon_i\pm\epsilon_j\}\cup\{\pm \epsilon_i\}\quad\text{and}\quad \Sigma^+=\{\epsilon_i\pm\epsilon_j\}\cup\{\epsilon_i\}
\end{equation*}
are the sets of restricted roots and the subset of positive roots, respectively. The dimension of the root spaces for the $\pm\epsilon_i\pm\epsilon_j$ is $1$ and that for the $\pm\epsilon_i$ is $n-m$ (see \cite[VI.4 Example 3]{Knapp}). From this and the formula $\beta_S(\nu)=|{\bf c}(\nu)|^{-2}$, where ${\bf c}(\nu)$ is the Harish-Chandra ${\bf c}$-function, recalled in Section \ref{sec:inv-formula}, one can deduce, using the Gindikin--Karpelevich product formula along with Stirling's formula, that if $\nu\in i\mathfrak{a}^*$ is sufficiently regular we have
\begin{equation}\label{first-beta-eq}
\beta_{\mathbb{H}^{n,m}}(\nu)\asymp \tilde\beta_{\mathbb{H}^{n,m}}(\nu),\quad\textrm{where}\quad\tilde\beta_{\mathbb{H}^{n,m}}(\nu)=\prod_{i\neq j}(1+|\nu_i\pm \nu_j|)\prod_i (1+|\nu_i|)^{n-m},
\end{equation}
where $\nu_i=\langle\nu,\epsilon_i\rangle$.

We now do the same with the Siegel upper half-space $\mathcal{H}^m$ of (real) dimension $m(m+1)$. We may realize $\mathcal{H}^m$ as $\Sp_{2m}(\R)/{\rm U}(m)$, where
\[
\Sp_{2m}(\R)=\left\{g\in \SL_{2m}(\R): {}^tgJg=J\right\}, \quad J=\begin{pmatrix} & I_m\\ -I_m & \end{pmatrix},
\]
and ${\rm U}(m)=\Oo(2m)\cap \Sp_{2m}(\R)$. 

We recall that $\mathfrak{sp}_{2m}=\{X\in\mathfrak{gl}_{2m}(\R): {}^tXJ+JX=0\}$ is the space of block matrices in $\mathfrak{sl}_{2m}(\R)$ of the form $\left(\begin{smallmatrix} A & B\\ C & -{}^t\!A \end{smallmatrix}\right)$, where $A,B,C\in \mathfrak{gl}_m(\R)$ and both $B$ and $C$ are symmetric. The $-1$ eigenspace $\mathfrak{p}'$ of the Cartan involution ``minus transpose" is the subspace corresponding to $B=C$, and we can take as a maximal abelian subspace of $\mathfrak{p}'$ the space $\mathfrak{a}'=\{{\rm diag}(A,-A)\}$, with $A$ diagonal. If we again let $\{ \epsilon_i : i=1,\ldots ,m \}$ be the standard basis of $\mathfrak{a}'^*$, then the sets of restricted roots $\Sigma$ and positive roots $\Sigma^+$ are
\begin{equation}
\label{symplectic-roots}
\Sigma=\{\pm\epsilon_i\pm\epsilon_j\}\cup\{\pm 2\epsilon_i\}\quad\text{and}\quad \Sigma^+=\{\epsilon_i\pm\epsilon_j\}\cup\{2\epsilon_i\},
\end{equation}
and all root space dimensions are $1$. We deduce, similarly to before, the estimate
\begin{equation}\label{second-beta-eq}
\beta_{\mathcal{H}^m}(\nu)\asymp\tilde\beta_{\mathcal{H}^m}(\nu),\quad\textrm{where}\quad\tilde\beta_{\mathcal{H}^m}(\nu)=\prod_{i\neq j} (1+|\nu_i\pm\nu_j|)\prod_i(1+|\nu_i|),
\end{equation}
valid for sufficiently regular $\nu\in i\mathfrak{a}^*$.

\begin{remark}
Consider the \textit{split} Lie algebra $\mathfrak{so}(m,m)$, viewed as a subalgebra of $\mathfrak{so}(n,m)$, by taking $C=E=F=0$ in \eqref{blocks}. Then $\mathfrak{a}\,\cap\,\mathfrak{so}(m,m)$ is isomorphic to $\mathfrak{a}$ and coincides in $\mathfrak{sl}_{2m}$ with $\mathfrak{a}'$. It is this sense in which we can view a spectral parameter $\nu$ simultaneously for $\mathbb{H}^{m,n}$ and $\mathcal{H}^m$.
\end{remark}

We deduce from \eqref{first-beta-eq} and \eqref{second-beta-eq} that the lower bound of Theorem \ref{sup-thm} can be expressed as
\begin{equation}\label{lower-bd-nu-i}
\|f\|_\infty\gg \left(\log (3+\|\nu\|)\right)^{-m[F:\Q]/2 }\prod_i (1+|\nu_i|)^{(n-m-1)/2}.
\end{equation}
Note that $n-m-1> 0$, due to the additional assumption on the parity of $n+m$.  We may therefore view the power growth as arising from the \textit{excess multiplicity} $n-m-1 > 0$ for the roots $\pm \epsilon_i$ in the non-split group $\SO(n,m)^0$.

\subsection{Relation to literature}

Although our interest here lies primarily in the higher rank setting of $m\geqslant 2$, we point out the relation between Theorem \ref{sup-thm} and the existing literature when the rank $m=1$. Let $\lambda^2$ be the Laplacian eigenvalue of $f$. In this case, Theorem \ref{sup-thm} yields a lower bound of $\lambda^{(n-2)/2}/(\log\lambda)^{[F:\Q]/2 }$, which is a factor of $\lambda^{-1/2}/(\log\lambda)^{[F:\Q]/2 }$ off of the upper bound $\lambda^{(n-1)/2}$ in \eqref{Linfty-local}. For $n=3$ this recovers (up to a logarithmic loss) a famous result of Rudnick-Sarnak \cite{RS}. Moreover, for $n\geqslant 5$ odd this improves upon a result of Donnelly \cite{Donnelly}, who had shown a weaker lower bound of $\lambda^{(n-4)/2}$. The exponent $(n-2)/2$ is the largest possible allowed by the purity conjecture (see Section \ref{sec:purity}).

We now compare the setting of Theorem \ref{sup-thm} to the general context of our earlier paper \cite{BM}. Let $V$ be a non-degenerate anisotropic quadratic space over a totally real number field, with signature $(n,m)$ at a fixed real place $v_0$ and definite for real places $v\neq v_0$, and put $\mathbf{G}=\SO(V)$. If $U$ is any $m$-dimensional quadratic subspace of $V$ that is defined over $F$ and negative definite at $v_0$, let $\theta$ be the involution $\theta(g)=sgs$, where $s$ is the orthogonal reflection with respect to $U$. Put $\bH=\bG^\theta$; then $\mathbf{H}={\rm S}(\Oo(U^\perp)\times\Oo(U))$. Note that $\theta$ induces a Cartan involution on $\bG(F_{v_0})$, since $\mathbf{H}_{v_0}={\rm S}(\Oo(n)\times\Oo(m))$ is a maximal compact subgroup of $\mathbf{G}(F_{v_0})=\SO(n,m)$. Moreover, the conditions on the signature at $v_0$ imply that $\mathbf{G}(F_{v_0})$ is nonsplit. Applying the main result from \cite{BM} we obtain the existence of compact quotients $\Gamma \backslash \mathbb{H}^{n,m}$ that support $L^2$-normalized $f$ with spectral parameter $\nu+O(1)$ such that $\| f \|_\infty\gg \lambda^\delta$ for some ineffective $\delta>0$.  

The interest in Theorem \ref{sup-thm} is therefore in the quality (in rank 1) and combinatorial shape (in higher rank) of the lower bound. It should be remarked, however, that while \cite{BM} produces a weaker exponent than Theorem \ref{sup-thm}, it has the advantage of not requiring the assumption that $\nu$ is sufficiently regular.


\subsection{Relation to purity conjecture}\label{sec:purity}

We now make some remarks comparing our results to the Sarnak purity conjecture \cite[Appendix]{Sa}. 

Let $Y$ be a compact negatively curved locally symmetric space, of dimension $d$ and rank $r$, and of congruence type. The Sarnak purity conjecture roughly states that the sup norms of \textit{well-balanced} Hecke--Maass forms on $Y$, of high frequency $\lambda$, should behave as $\lambda^{k/2}$, where $k$ lies in $\Z\cap [0,d-r)$. Thus, the $L^\infty$ exponents on arithmetic manifolds are quantized. Here, ``well-balanced'' means that the spectral parameter $\nu$ should verify $1+|\langle \alpha, \nu \rangle|\asymp 1+\|\nu\|$ for all roots $\alpha$.

If one takes $\nu$ well-balanced in the above sense, then the quotient $\left(\beta_{\mathbb{H}^{n,m}}(\nu)/\beta_{\mathcal{H}^m}(\nu)\right)^{1/2}$ appearing in Theorem \ref{sup-thm} becomes $\lambda^{(n-m-1)m/2}=\lambda^{(d-r)/2-r^2/2}$, where $d=mn$ and $r=m$. (Note that the relation $1+\|\nu\|\asymp 1+\lambda$ always holds.) Thus, for well-balanced $\nu$, Theorem \ref{sup-thm} is consistent with Sarnak's purity conjecture, in that the exponent of $\lambda$ is indeed a half-integer. Note that the well-balanced condition is automatic in rank 1.

That the well-balanced condition is necessary for any purity conjecture to hold was observed in \cite[Appendix]{Sa}, in light of the examples of \cite{LO}. Indeed, as a consequence of their period formula, Lapid and Offen exhibited a lower bound of the form $\beta_{\mathfrak{H}_\C^n}(\nu)^{1/4}$ for the sup norm of base change lifts on compact quotients of $\mathfrak{H}_\C^n=\GL_n(\C)/{\rm U}(n)$ coming from anisotropic unitary groups. In this higher rank setting, the expression $\beta_{\mathfrak{H}_\C^n}(\nu)^{1/4}$ can vary continuously as a power of $\lambda$ according to the position of $\nu$ relative to the root hyperplanes. This in turn lead to the suggestion (see \cite[Appendix]{Sa}) that the correct measure of complexity was rather $\beta_S(\nu)$.

Observe that the lower bound in our Theorem \ref{sup-thm} is not, in light of \eqref{first-beta-eq} and \eqref{lower-bd-nu-i}, a fractional power of $\beta_{\mathbb{H}^{n,m}}(\nu)$. So while $\beta_{\mathbb{H}^{n,m}}(\nu)$ is indeed of relevance, it does not \textit{alone} determine the precise shape of $\|f\|_\infty$. In fact, the method of proof of Theorem \ref{sup-thm} would lead one to venture that the correct quantization of sup norm growth is through a \textit{comparison} of spectral densities. For example, putting $\mathfrak{H}_\R^n=\GL_n(\R)/{\rm O}(n)$, the Lapid-Offen bound can be written 
\[
\beta_{\mathfrak{H}_\C^n}(\nu)^{1/4}\asymp (\beta_{\mathfrak{H}_\C^n}(\nu)/\beta_{\mathfrak{H}_\R^n}(\nu))^{1/2}.
\]
The power growth is once again seen to come from the excess multiplicity (here equal to $2$) of the roots for $\GL_n(\C)$. Note, finally, that for the rank one signature $(n,1)$ we may recognize the eligible purity exponents as
\[
\lambda^{(n-j)/2}=(\lambda^{n-1}/\lambda^{j-1})^{1/2}\asymp (\beta_{\mathbb{H}^n}(\nu)/\beta_{\mathbb{H}^j}(\nu))^{1/2} \qquad (j=2,\ldots , n),
\]
since $1+\lambda^{j-1}\asymp\beta_{\mathbb{H}^j}(\nu)$.

\subsection{Outline of the paper}

We shall prove Theorem \ref{sup-thm} by following the steps \ref{1sketch}, \ref{2sketch}, and \ref{3sketch} given in Section \ref{sec:intro}.  We now expand on these steps, and explain how they relate to the structure of the paper.  As in Section \ref{sec:intro}, let $\bG=\Oo(V)$, for a non-degenerate anisotropic quadratic space $V$ over a totally real number field, with signature $(n,m)$ at a fixed real place and definite at the others.  Let $\bH=\Oo(U^\perp)\times\Oo(U)$, where $U$ is an $m$-dimensional quadratic subspace that is negative definite at the distinguished real place, and $\mathbf{G}'=\Sp_{2m}$.

\subsubsection{Sections \ref{sec:notation} to \ref{sec:points}}

These sections correspond to Steps \ref{1sketch} and \ref{3sketch}. Step \ref{1sketch} asks for the average pointwise value of Maass forms on $\bG$ in the eigenvalue aspect, and Step \ref{3sketch} asks for an upper bound for the number of cusp forms on $\bG'$ with fixed ${\rm U}(m)$-type $\tau$, where $\tau$ is fixed and one-dimensional, in the eigenvalue aspect.  We solve these problems using appropriate variants of the trace formula.  Note that the problem of counting $\tau$-spherical forms on $\bG'$ arises because these correspond to spherical forms on $\bG$ under the theta lift, as explained below.

After introducing notation for real groups in Section \ref{sec:notation}, we construct the archimedean test functions that we shall use in the trace formula in Sections \ref{sec:inv-formula} and \ref{sec:test-function}.  Because of the need to count $\tau$-spherical forms on $\bG'$, our test functions will likewise need to be $\tau$-spherical in Step \ref{3sketch}, and we construct these using the $\tau$-spherical transform studied by Shimeno \cite{Shimeno}.

Section \ref{sec:upper-bd-cusp-spec} contains our bound for the number of $\tau$-spherical cusp forms on $\bG'$.  We use a relatively elementary approach, by showing that cusp forms with spectral parameter $\lambda$ are concentrated below height $\| \lambda \|^{1+\epsilon}$ in the cusp, and then bounding the integral of the geometric kernel function over this region.  As discussed in Remark \ref{remark2-intro}, this leads to the logarithmic loss in our main theorem.

We establish the average value result for Maass forms on $\bG$ in Section \ref{sec:points}.  The result we prove here is purely analytical in nature, requiring no arithmetic hypotheses on the locally symmetric spaces nor the eigenfunctions themselves.

\subsubsection{Sections \ref{sec:theta-review} and \ref{sec:Period-relation}} These sections are devoted to the distinction principle of Step \ref{2sketch}.  This distinction principle will follow from a period relation between forms on $\bG$ and $\bG'$, stated as Proposition \ref{period-relation}.

To describe this relation, let $\phi$ be a vector in an automorphic representation $\pi$ of $\bG$, and assume that both $\pi$ and $\phi$ are spherical at infinity.  Let $\theta(\phi)$ be the theta lift of $\phi$ to a form on $\bG'$.  In Section \ref{sec:Bessel-period}, we define a Bessel subgroup $\bR = \Oo(U) \rtimes \bN$ of the Siegel parabolic $\bP_\text{Siegel} = \bM \bN$, and a character $\psi_\bN$ of $\bR$.  Our period relation then connects the $\bH$-period of $\phi$,
\[
\mathscr{P}_\bH (\phi)=\int_{[\bH]} \phi(h) dh,
\]
with the Bessel period of $\theta(\phi)$,
\[
\mathscr{B}_\bR^{\, \psi_N}(\theta(\phi))=\int_{[\bR]} \theta(\phi)(r)\psi_N^{-1}(r)dr.
\]
This will imply that if $\mathscr{P}_\bH (\phi)$ is nonzero, then so is the lift $\theta(\phi)$, and hence that $\mathscr{P}_\bH$ distinguishes theta lifts from $\bG'$.  This approach is the same one used by Rudnick and Sarnak to prove distinction in the case $m = 1$.  In higher rank, this period relation was also considered by Gan in \cite{Gan}; note that he refers to $\mathscr{B}_\bR^{\, \psi_N}(\theta(\phi))$ as a Shalika period, rather than a Bessel period.

Section \ref{sec:theta-review} introduces notation for the theta lift, and describes the global lift $\Theta(\pi)$ of the spherical representation $\pi$ to $\bG'$.  In particular, we show that $\Theta(\pi)$ is $\tau$-spherical, where $\tau = \det^{(n-m)/2}$. We also show that the lift of a tempered spherical representation of $\bG(\R)$ to ${\rm Sp}_{2m'}$ vanishes for $m' < m$.  It follows from this that if $\pi_\infty$ is tempered, then $\Theta(\pi)$ is either zero or cuspidal; the temperedness of $\pi_\infty$ will follow from the sufficient regularity condition on $\nu$ in Theorem \ref{sup-thm}.  These two results imply that the forms on $\bG'$ participating in our distinction principle are $\tau$-spherical cusp forms, which are counted in Section \ref{sec:upper-bd-cusp-spec} as described above.

\subsubsection{Section \ref{ProofsOfThms}}
In this section, we deduce Theorem \ref{sup-thm} from the above ingredients via a counting argument.

\subsection{Acknowledgements}

We would like to thank Stephen Kudla and Colette Moeglin for helpful conversations.

\section{Notation}\label{sec:notation}

In this section we introduce notation related to harmonic analysis on Lie groups.  Let $G$ be a non-compact connected\footnote{In \S\S\ref{sec:notation}-\ref{sec:test-function}, we will assume $G$ to be connected. The results proved in these sections will be applied in \S\S\ref{sec:upper-bd-cusp-spec}-\ref{sec:points} to the neutral component of the $F_{v_0}$-points of a semisimple group $\bG$ defined over a totally real number field $F$, with distinguished real place $v_0$. This neutral component is denoted by $G^0$ in Section \ref{sec:upper-bd-cusp-spec} and by $G^{00}$ in Section \ref{sec:points}.} semisimple Lie group with finite center, and $K$ a maximal compact subgroup of $G$. We write $S=G/K$ for the associated irreducible Riemannian globally symmetric space. 

\subsection{Group decompositions}\label{sec:gp-decomp}

Let $\g$ be the Lie algebra of $G$ and $\mathfrak{k}\subset\g$ the Lie algebra of $K$. Then $K$ is the group of fixed points of a Cartan involution $\Theta$ of $G$. Let $\theta$ be the differential of $\Theta$ and write $\g=\p\oplus\mathfrak{k}$ for the corresponding Cartan decomposition of $\g$.

Let $B$ be the Killing form on $\mathfrak{g}$; identifying $\p$ with the tangent space at the identity $K$ of $S=G/K$, we endow $S$ with the Riemannian metric induced by the restriction of $B$ to $\p$. Let $d$ be the associated bi-$K$-invariant distance function on $G$. For $R>0$, we put $G_R=\{g\in G: d(g,K)\leqslant R\}$. 

Fix a maximal abelian subspace $\mathfrak{a}\subset\p$. Denote by $\langle \cdot ,\cdot\rangle$ the restriction of the Killing form to $\mathfrak{a}$. Write $\Sigma=\Sigma(\mathfrak{a},\g)\subset\mathfrak{a}^*$ for the set of restricted roots and let $\Sigma^+$ denote a choice of positive roots. Then we have a restricted root space decomposition $\g=\mathfrak{m}\oplus\mathfrak{a}\oplus\sum_{\alpha\in\Sigma}\g_\alpha$, where $\mathfrak{m}=Z_\mathfrak{k}(\mathfrak{a})$ is the centralizer of $\mathfrak{a}$ in $\mathfrak{k}$. Let $m_\alpha=\dim\g_\alpha$ be the multiplicity of the root $\alpha\in\Sigma$. As usual, we let $\rho=\frac12\sum_{\alpha\in\Sigma^+}m_\alpha\alpha$ be the half-sum of positive roots. We denote by $\Psi\subset\Sigma^+$ the set of simple roots. Let $\{H_\alpha:\alpha\in\Psi\}$ be the dual basis of $\Psi$ in $\mathfrak{a}$. For a root $\alpha\in\Sigma$ let $\alpha^\vee=2\alpha/\langle\alpha,\alpha\rangle$ denote the corresponding coroot. Let $W=N_K(\mathfrak{a})/Z_K(\mathfrak{a})$ be the associated Weyl group.

Let $\mathfrak{a}_\C=\mathfrak{a}\otimes\C$. We call $\lambda\in\mathfrak{a}_\C^*$ \textit{regular} if $\langle\lambda,\alpha\rangle\neq 0$ for all $\alpha\in\Sigma$. For $\reg > 0$ we say that $\lambda$ is $\reg$-\textit{regular} if $|\langle\lambda,\alpha\rangle|\geqslant \reg$ for all $\alpha\in\Sigma^+$. Finally, we say that $\lambda\in\mathfrak{a}^*_\C$ is \textit{sufficiently regular} if $\lambda$ is $\reg$-regular for a sufficiently large $\reg>0$, depending only on the group $G$. In general, when working with sufficiently regular parameters, all implied constants will depend on the underlying $\reg$; we will not, however, explicate this dependence in the notation. The notion of sufficiently regular should not be confused with the much stronger ``well-balanced'' condition introduced in Section \ref{sec:purity}.

Let $A$ be the analytic subgroup corresponding to $\mathfrak{a}$, and $N$ be the analytic subgroup corresponding to $\mathfrak{n}=\sum_{\alpha\in\Sigma^+}\g_\alpha$. We furthermore write $\overline{N}$ for the analytic subgroup of $\theta(\mathfrak{n})=\sum_{-\alpha\in\Sigma^+}\g_\alpha$. Let $M= Z_K(A)$ be the centralizer of $A$ in $K$. Then $P= NAM$ is a minimal parabolic subgroup of $G$ containing $A$.

We have an Iwasawa decomposition $G=NAK$. Denote by $H:G\rightarrow\mathfrak{a}$ and $\kappa: G\rightarrow K$ the respective Iwasawa projections, given by $g=n e^{H(g)}\kappa (g)$. We put $a(g)=e^{H(g)}$; more generally, when $\lambda\in\mathfrak{a}_\C^*$ we write $a(g)^\lambda=e^{\lambda(H(g))}$. 

We now fix measure normalizations. We equip $K$ with the probability Haar measure. The Killing form induces a $W$-invariant Haar measure on $\mathfrak{a}$ and hence $\mathfrak{a}^*$; this yields a measure $da$ on $A$ via the exponential map. Next we let $dn$ be the left-invariant Haar measure on $N$ normalized as in \cite[p. 37]{DKV}. We now define a Haar measure $dg$ on $G$ through the Iwasawa decomposition $G=NAK$. Namely, we let
\begin{equation}\label{Iwasawa-measure1}
dg=a^{-2\rho} dnda dk \qquad (g=nak).
\end{equation}
We recall that $A$ normalizes $N$, and that
\begin{equation}\label{A-normalizes-N}
\det \Ad (a)|_{\mathfrak{n}}=a^{2\rho}\qquad (a\in A).
\end{equation}
We may therefore write $dg$ in the coordinates $G=ANK$ as 
\begin{equation}\label{Iwasawa-measure2}
dg=dadn dk \qquad (g=ank).
\end{equation}

Finally, we recall the Cartan decomposition $G=KAK$. Let $B_\mathfrak{a}(0,R)$ denote the ball in $\mathfrak{a}$ centered at the origin and of radius $R>0$, relative to the norm induced by $\langle\cdot,\cdot\rangle$. We put $A_R=\exp B_\mathfrak{a}(0,R)$. Then we have $G_R=KA_RK$.

\subsection{Spherical functions of abelian $K$-type}\label{sec:tau-spherical}

We now recall several formulae in the theory of $\tau$-spherical functions, where $\tau$ is an abelian $K$-type. 

Fix a character $\tau: K\rightarrow\C^\times$ and let $C^\infty(S,\tau)$ denote the space of smooth functions $f:G\rightarrow\C$ satisfying $f(gk)=\tau(k)^{-1}f(g)$ for all $k\in K$. We view $C^\infty(S,\tau)$ as the space of smooth sections $H^0(S,L(\tau))$ of the associated homogeneous line bundle $L(\tau)=G\times_K \tau$ over $S$. Let $\mathbb{D}(\tau)$ denote the algebra of all left $G$-invariant differential operators on $H^0(S,L(\tau))$. Let $\gamma_{\rm HC}:\mathbb{D}(\tau)\rightarrow I(\mathfrak{a})$ be the Harish-Chandra isomorphism of type $\tau$ (see \cite[Theorem 2.4]{Shimeno1990}), where $I(\mathfrak{a})=S(\mathfrak{a})^W$ and $S(\mathfrak{a})$ is the symmetric algebra of $\mathfrak{a}$ with complex coefficients. For fixed $\lambda\in\mathfrak{a}_\C^*$ we let $\chi_\lambda: \mathbb{D}(\tau)\rightarrow\C$ be the algebra homomorphism given by $\chi_\lambda(D)=\gamma_{\rm HC}(D)(\lambda)$; clearly, $\chi_{w\lambda}(D)=\chi_\lambda(D)$ for all $w\in W$. The space $\mathcal{E}_\lambda(S,\tau)$ of joint $\mathbb{D}(\tau)$-eigensections in $H^0(S,L(\tau))$ having fixed system of eigenvalues $\chi_\lambda$ forms a representation of $G$. Endowed with its natural Fr\'echet space topology, it is a smooth representation of $G$.

Let $C^\infty(S,\tau,\tau)$ denote the subspace of $C^\infty(S,\tau)$ consisting of smooth functions $f:G\rightarrow\C$ such that $f(k_1gk_2)=\tau(k_1k_2)^{-1}f(g)$ for all $k_1,k_2\in K$. This is a commutative algebra under convolution \cite[Lemma 3.1]{Shimeno}. By \cite[Proposition 6.1]{Shimeno1990}, for $\lambda\in\mathfrak{a}_\C^*$ there is a unique $\varphi_{\lambda,\tau}\in C^\infty(S,\tau,\tau)\cap \mathcal{E}_\lambda(S,\tau)$ such that $\varphi_{\lambda,\tau}(e)=1$. Then the line spanned by $\varphi_{\lambda,\tau}$ is the $\tau$-isotypic subspace for $\mathcal{E}_\lambda(S,\tau)$ under the action of $K$. Moreorer, if $\lambda,\mu\in\mathfrak{a}^*_\C$ then
\begin{equation}\label{eq:equality-sph-fn}
\varphi_{\lambda,\tau}=\varphi_{\mu,\tau} \quad\text{if, and only if, there is } w\in W \textrm{ such that } \lambda=w\mu.
\end{equation}
We call $\varphi_{\lambda,\tau}$ the \textit{$\tau$-spherical function} of spectral parameter $\lambda$. 

Let $V(\lambda,\tau)$ denote the smallest closed invariant subspace of $\mathcal{E}_\lambda(S,\tau)$ containing $\varphi_{\lambda,\tau}$. We give it the relative topology. Then $V(\lambda,\tau)$ is irreducible (see \cite[Corollary 5.2.5]{Heckman}) and is generically of infinite dimension (see \cite[Corollary 5.2.8]{Heckman}).

\subsection{Properties of the $\tau$-spherical function}\label{sec:induced-reps}

In this section we recall the Poisson integral formula for $\varphi_{\lambda,\tau}$ and study its transformation properties.

The Poisson integral formula is obtained by writing $\varphi_{\lambda,\tau}$ as a matrix coefficient of a generalized principal series representation. To define these representations, let $P=NAM$ be the standard minimal parabolic subgroup from Section \ref{sec:gp-decomp}. Let $(\xi,V_\xi)$ be a finite dimensional unitary representation of $M$ and let $\lambda\in\mathfrak{a}^*_\C$. We let $\pi_{\lambda,\xi}={\rm Ind}_P^G(1\otimes e^\lambda\otimes\xi)$ be the smooth representation unitarily induced from the representation $1\otimes e^\lambda\otimes\xi$ on $P=NAM$. Thus $\pi_{\lambda,\xi}$ consists of smooth functions $f:G\rightarrow V_\xi$ that satisfy
\[
f(namx)=a^{\rho+\lambda}\xi(m)f(x),\qquad n\in N,\; a\in A, \; m\in M, \; x\in G,
\]
the $G$-action being via the right-regular representation.  We have a nondegenerate $G$-invariant pairing between $\pi_{\lambda, \xi}$ and $\pi_{-\lambda, \xi^\vee}$, where $\xi^\vee$ denotes the contragredient, given by
\[
\langle f_1, f_2\rangle=\int_K  \langle f_1(k),f_2(k) \rangle dk.
\]


To specialize this to the present setting, we let $\tau$ be an abelian $K$-type, and $\xi$ the restriction of $\tau$ to $M$. Then by Frobenius reciprocity $\pi_{\lambda, \xi}$ admits $\tau$ as a non-trivial $K$-type of dimension one, the line generated by $e_{\lambda, \tau}(g)=\tau(\kappa(g))a(g)^{\rho+\lambda}$. Note that $e_{\lambda, \tau}$ is the unique function in $\pi_{\lambda, \xi}$ which extends the function $\tau$ on $K$.

The matrix coefficient assignment $v\mapsto m_v$, where $m_v(g)=\langle v,\pi_{\lambda, \xi^{-1}}(g)e_{\lambda, \tau^{-1}} \rangle$, yields an intertwining map $\pi_{-\lambda,\xi}\rightarrow \mathcal{E}_\lambda(S,\tau)$. In particular, when $v=e_{-\lambda, \tau}$, this yields the integral expression
\begin{equation}\label{defn-sph-fn}
\varphi_{\lambda,\tau}(g)=\langle e_{-\lambda, \tau}, \pi_{\lambda, \xi^{-1}}(g) e_{\lambda, \tau^{-1}} \rangle=\int_K a(kg)^{\rho+\lambda}\tau(\kappa(kg)^{-1} k)dk,
\end{equation}
which recovers that of Harish-Chandra when $\tau$ is trivial.

We shall use the following transformational property of $\varphi_{\lambda,\tau}$.

\begin{prop}\label{prop:sph-fn-2-vars}
The $\tau$-spherical function satisfies $\varphi_{\lambda,\tau}(g^{-1})=\varphi_{-\lambda,\tau^{-1}}(g)$.
\end{prop}
\begin{proof}

This follows from the expression \eqref{defn-sph-fn} for $\varphi_{\lambda,\tau}$ as a matrix coefficient.  We have
\begin{align*}
\varphi_{\lambda,\tau}(g^{-1}) & = \langle e_{-\lambda, \tau}, \pi_{\lambda, \xi^{-1}}(g^{-1}) e_{\lambda, \tau^{-1}} \rangle \\
& = \langle \pi_{-\lambda, \xi}(g) e_{-\lambda, \tau}, e_{\lambda, \tau^{-1}} \rangle \\
& = \langle e_{\lambda, \tau^{-1}}, \pi_{-\lambda, \xi}(g) e_{-\lambda, \tau} \rangle = \varphi_{-\lambda,\tau^{-1}}(g),
\end{align*}
as required.
\end{proof}

\subsection{Harish-Chandra transform of type $\tau$}

We continue to let $\tau$ be a one-dimensional $K$-type. Let $C_c^\infty(S,\tau,\tau)$ denote the $\tau$-spherical Hecke algebra, i.e., the space $C_c^\infty(S,\tau,\tau)$ equipped with the convolution product. Using Gelfand's trick, one can show that $C_c^\infty(S,\tau,\tau)$ is commutative (see \cite[Lemma 3.1]{Shimeno}).

For $R>0$, let $\mathcal{PW}(\mathfrak{a}_\C^*)_R$ denote the space of Paley--Wiener functions of exponential type $R$ on $\mathfrak{a}_\C^*$. Put $\mathcal{PW}(\mathfrak{a}_\C^*)=\bigcup_{R>0} \mathcal{PW}(\mathfrak{a}_\C^*)_R$. The Harish-Chandra transform of type $\tau$ is the algebra homomorphism $\mathscr{H}_\tau:C_c^\infty(S,\tau,\tau)\rightarrow \mathcal{PW}(\mathfrak{a}_\C^*)$ given by
\begin{equation}\label{HC-transform}
\mathscr{H}_\tau f(\lambda)=\int_G f(g)\varphi_{-\lambda,\tau^{-1}}(g)dg.
\end{equation}
In particular $\mathscr{H}_\tau (f\ast g)=\mathscr{H}_\tau f.\mathscr{H}_\tau g$. The map $\mathscr{H}_\tau$ is injective with image $\mathcal{PW}(\mathfrak{a}_\C^*)^W$, the subspace of Weyl group invariants \cite[Lemma 6.2 and Theorem 6.5]{Shimeno}. We shall sometimes write $\hat{f}(\lambda;\tau)$ for $\mathscr{H}_\tau f(\lambda)$; moreover, when $\tau=1$, we shall often drop $\tau$ from the notation and simply write  $\mathscr{H}f$ or $\hat{f}$. 

On the other hand, the space $C_c^\infty(S,\tau,\tau)$ maps onto $C^\infty_c(A)^W$ via the Abel--Satake transform
\[
\mathscr{A}_\tau f(a)=a^\rho \int_N f(an) dn.
\]
These two maps fit into the following commutative diagram \cite[Theorem 5.4.2]{Heckman}
\begin{equation}\label{eq:commutative-diagram}
\xymatrix{
&C_c^\infty(S,\tau,\tau)\ar[rd]^{\mathscr{A}_\tau}  \ar[ld]_{\mathscr{H}_\tau}\\
\mathcal{PW}(\mathfrak{a}_\C^*)^W  && C^\infty_c(A)^W,\ar[ll]_{\mathscr{F}}
}
\end{equation}
where $\mathscr{F}f(\lambda)=\int_A f(a)a^{-\lambda} da$ is the Fourier transform. These maps respect supports in the following way: for $R>0$, let $C^\infty_R(S,\tau,\tau)$ denote the space of functions in $C^\infty_c(S,\tau,\tau)$ with ${\rm supp}(f)\subset G_R$; similarly, let $C^\infty_R(A)^W$ denote the space of functions in $C^\infty_c(A)^W$ with ${\rm supp}(\mathscr{A}_\tau(f))\subset A_R$. Then, for $f\in C^\infty(S,\tau,\tau)$: 
\[
f\in C^\infty_R(S,\tau,\tau)\;\text{ if, and only if, }\; \mathscr{A}_\tau(f)\in C^\infty_R(A)^W\;\text{ if, and only if, }\; \mathscr{H}_\tau (f)\in \mathcal{PW}(\mathfrak{a}_\C^*)^W_R;
\]
see \cite[Theorem 6.5]{Shimeno}.

\section{$\tau$-spherical inversion formula}\label{sec:inv-formula}

In this section $G$ shall denote a non-compact connected \textit{simple} Lie group, $K$ a maximal compact subgroup of $G$, and $S=G/K$. The center of $\gk$ is either trivial or one dimensional, and in the latter case $G$ is said to be of Hermitian type. We let $\tau$ be a character of $K$. We shall make the following simplifying assumption on $G$:
\begin{equation}\label{eq:assumption}
\textit{whenever $\tau$ is non-trivial, the Hermitian group $G$ has reduced root system.}
\end{equation}
This will allow us to keep the notational complexity to a minimum, while allowing for sufficient generality for our applications. For the proof of Theorem \ref{sup-thm} the only time we shall consider a non-trivial character $\tau$ will be for $G=\Sp_{2m}(\R)$.

In this section we recall the explicit formulae for the $\tau$-spherical Plancherel measure $\tpl$ on $\mathfrak{a}^*_\C$. This is the unique $\sigma$-finite measure for which the $\tau$-spherical Harish-Chandra transform $\mathscr{H}_\tau$ extends to an isometry
\[
\mathscr{H}_\tau: L^2(S,\tau,\tau)\rightarrow L^2(\mathfrak{a}^*_\C,|W|^{-1} \tpl).
\]
Importantly for us, the $\tau$-spherical Plancherel measure satisfies an inversion formula of the form
\begin{equation}\label{eq:general-inversion}
k(g)=\frac{1}{|W|}\int_{\mathfrak{a}^*_\C}\mathscr{H}_\tau k(\lambda)\varphi_{\lambda,\tau}(g)d\tpl(\lambda),
\end{equation}
for every $k\in C^\infty_c(S,\tau,\tau)$. 

\subsection{Parabolic subgroups}\label{sec:parabolic-notation}
Recall from \S\ref{sec:gp-decomp} that we have fixed a minimal parabolic subgroup $P$ of $G$. Any parabolic subgroup of $G$ containing $P$ is called \textit{standard}; they are parametrized as follows. Let $\Theta$ be a subset of the positive simple roots $\Psi$, and write $\langle\Theta\rangle=\Sigma^+\cap\big(\sum_{\alpha\in\Theta}\N\alpha\big)$ for the system of positive roots generated by $\Theta$. We write $W_\Theta$ for the subgroup of $W$ generated by the hyperplane reflections $s_\alpha$ for $\alpha\in\Theta$. Then $P_\Theta=PW_\Theta P$ is a standard parabolic subgroup, with unipotent radical $N_\Theta$ given by the product over $\alpha\in\Sigma^+\setminus\langle\Theta\rangle$ of the one-parameter subgroups corresponding to $\g_\alpha$. In particular $P_\emptyset=P$ and (by the Bruhat decomposition) $P_{\Psi}=G$. Let $P_\Theta= N_\Theta A_\Theta M_\Theta$ be its Langlands decomposition.

The Lie algebra of $A_\Theta$ is given by
\begin{equation}\label{defn:Lie-A-theta}
\mathfrak{a}_\Theta=\{H\in\mathfrak{a}: \alpha(H)=0 \;\forall\alpha\in\Theta\}=\sum_{\alpha\in\Psi\setminus\Theta}\R H_\alpha.
\end{equation}
Note that $\mathfrak{a}_\emptyset=\mathfrak{a}$ and $\mathfrak{a}_{\Psi}=\{0\}$. We write $\mathfrak{a}(\Theta)$ for the orthocomplement of $\mathfrak{a}_\Theta$ in $\mathfrak{a}$ relative to the Killing form, and $N(\Theta)$ for the unipotent subgroup of $M_\Theta$ given by the product over $\alpha \in \langle \Theta \rangle$ of the one-parameter subgroups corresponding to $\g_\alpha$.

Finally, let $\overline{N_\Theta}$ and $\overline{N(\Theta)}$ denote the analytic subgroups corresponding to the Lie subalgebras $\sum_{-\alpha\in\Sigma^+\setminus\langle\Theta\rangle}\g_\alpha$ and $\sum_{-\alpha\in\langle\Theta\rangle}\g_\alpha$, respectively.

\subsection{More notation on roots}\label{sec:more-roots}
In preparation for the paragraphs to come, we shall need to introduce some notation concerning various types of roots in $\Sigma$.

We let $\Sigma_{\rm red}=\{\alpha\in\Sigma: \alpha/2\notin\Sigma\}$ denote the subset of  reduced roots. We say that $\Sigma$ is \textit{reduced} if $\Sigma_{\rm red}=\Sigma$. As $\alpha$ varies over $\Sigma_{\rm red}$, there are at most 2 possible lengths $\langle \alpha,\alpha\rangle^{1/2}$. (Recall that the root systems $\Sigma$ for which all roots have the same length are called \textit{simply laced}; a simply laced root system is automatically reduced.) We call the reduced roots of minimal length \textit{short} and those of maximal length \textit{long} (in the simply laced case, all roots will be deemed long). Roots of the same length have the same multiplicities, as they lie in the same Weyl group orbit.

If $G$ is Hermitian with reduced root system, as in assumption \eqref{eq:assumption}, the classification result of C.C. Moore \cite[Theorem 2]{Moore} shows that $\Sigma$ is of the form \eqref{symplectic-roots}, i.e., of type $C$.  However, we shall work with a different basis of $\ga^*$ to the one in \eqref{symplectic-roots} (differing only by a factor or 2), in order to agree with Shimeno's paper.  We take a basis $\{ \beta_1, \ldots, \beta_r \}$ of $\ga^*$ so that the sets of short and long roots are
\begin{equation}
\label{Shimeno-roots}
\Sigma_s=\{\tfrac12 (\pm\beta_i\pm\beta_j): 1\leqslant i,j\leqslant r, i\neq j\} \quad \text{and} \quad \Sigma_\ell=\{\pm\beta_1,\ldots ,\pm\beta_r\}
\end{equation}
respectively.  We have $m_\alpha=1$ for all $\alpha\in\Sigma_\ell$. The positive short and long roots are $\Sigma_s^+=\{\frac12 (\beta_i\pm\beta_j) : i > j \}$ and $\Sigma_\ell^+=\{\beta_1,\ldots ,\beta_r\}$, respectively. Then $\Sigma_\ell^+$ forms a basis for $\mathfrak{a}^*$.

\subsection{$\tau$-spherical ${\bf c}$-function}\label{sec:c-function}

We continue to let $\tau$ be a one-dimensional representation of $K$. Let $\Theta$ be a subset of the simple roots $\Psi$. For $\lambda\in\mathfrak{a}_\C^*$ with ${\rm Re}\, \langle\lambda,\alpha\rangle>0$ for all $\alpha\in\Sigma^+\setminus\langle\Theta\rangle$ we consider the absolutely convergent integral
\[
{\bf c}^\Theta(\lambda;\tau)=\int_{\overline{N_\Theta}}a(\bar{n})^{\lambda+\rho} \tau(\kappa(\bar{n}))^{-1}d\bar{n}.
\]
Similarly, for $\lambda\in\mathfrak{a}_\C^*$ with ${\rm Re}\, \langle\lambda,\alpha\rangle>0$ for all $\alpha\in\langle\Theta\rangle$, the integral
\[
{\bf c}_\Theta(\lambda;\tau)=\int_{\overline{N(\Theta)}}a(\bar{n})^{\lambda+\rho} \tau(\kappa(\bar{n}))^{-1}d\bar{n}
\]
converges absolutely. Both of these integrals are taken with respect to the invariant measure $d\bar{n}$, normalized so that ${\bf c}_\Theta(\rho,\tau_{\rm triv})=1={\bf c}^\Theta(\rho,\tau_{\rm triv})$. When $\Theta=\emptyset$, in which case $\overline{N_\Theta}=\overline{N}$, or $\Theta=\Psi$, in which case $\overline{N(\Theta)}=\overline{N}$, we shall often write ${\bf c}(\lambda;\tau)$ for ${\bf c}^\emptyset (\lambda;\tau)={\bf c}_\Psi (\lambda;\tau)$. Furthermore, when $\tau=\tau_{\rm triv}$ we shorten ${\bf c}(\lambda;\tau_{\rm triv})$ further to ${\bf c}(\lambda)$.


The rank one reduction procedure of the Gindikin--Karpelevich formula allows us to evaluate these integrals using quotients of Gamma functions attached to each positive reduced root, as we now describe. 

We begin with the spherical case. Let $\lambda\in\mathfrak{a}^*_\C$. For $\alpha\in \Sigma_{\rm red}^+$ we put
\begin{equation}\label{c-tilde-spherical}
{\bf c}_\alpha(\lambda) = c_{\alpha, 0} \frac{ 2^{ - \tfrac{1}{2}\langle \lambda, \alpha^\vee \rangle } \Gamma( \tfrac{1}{2} \langle \lambda, \alpha^\vee \rangle ) }{ \Gamma( \tfrac{1}{4} m_\alpha + \tfrac{1}{2} + \tfrac{1}{4} \langle \lambda, \alpha^\vee \rangle ) \Gamma( \tfrac{1}{4} m_\alpha + \tfrac{1}{2} m_{2\alpha} + \tfrac{1}{4} \langle \lambda, \alpha^\vee \rangle ) },
\end{equation}
where the constant $c_{\alpha, 0}$ is given by
\[
c_{\alpha,0} = 2^{ \tfrac{1}{2} m_\alpha + m_{2\alpha}} \Gamma( \tfrac{1}{2}( m_\alpha + m_{2\alpha} + 1) ).
\]
This is the same as the formula from \cite[Ch. IV, Section 6]{Helgason} after replacing $\alpha_0$ with $\tfrac{1}{2} \alpha^\vee$.  Let $\langle\Theta\rangle_{\rm red}=\langle\Theta\rangle\cap\Sigma_{\rm red}$. Then
\[
{\bf c}_\Theta(\lambda)=\prod_{\alpha\in\langle\Theta\rangle_{\rm red}}{\bf c}_\alpha(\lambda),\qquad {\bf c}^\Theta(\lambda)=\prod_{\alpha\in \Sigma_{\rm red}^+\setminus\langle\Theta\rangle_{\rm red}}{\bf c}_\alpha(\lambda).
\]
In particular, when $\Theta=\Psi$ in the first of the above two products, or $\Theta=\emptyset$ in the second, we recover ${\bf c}(\lambda)=\prod_{\alpha\in \Sigma_{\rm red}^+}{\bf c}_\alpha(\lambda)$ from \cite[Ch. IV (43)]{Helgason}. 

Now let $\tau$ be an arbitrary character of $K$. In accordance with \eqref{eq:assumption}, and since we have already treated the spherical case for general groups, we may now assume that $G$ is of Hermitian type with reduced root system. In this case, our formulas for the $\tau$-spherical ${\bf c}$-function will involve a real parameter $\ell_\tau$ associated to $\tau$, whose definition we recall from \cite[p. 337]{Shimeno}. Let $\gk_s$ and $\mathfrak{z}_\gk$ be the derived subalgebra and center of $\gk$, respectively. Then $\gk_s\subsetneq \gk$ and $\dim \mathfrak{z}_\gk=1$. Let $G_\C$ be the simply connected Lie group with Lie algebra $\g_\C$, and let $G_\R$, $K_\R$, and $(K_\R)_s$ be the analytic subgroups of $G_\C$ corresponding to $\g$, $\gk$, and $\gk_s$.  Let $Z$ be the non-zero element of $\mathfrak{z}_\gk$ constructed in \cite[Sect. 3]{Schl}, so that $e^{tZ} \in (K_\R)_s$ if and only if $t \in 2 \pi \Z$.  Given a character $\tau$ of $K$, we define $\ell_\tau$ by requiring that $\tau( e^{tZ}) = e^{i \ell_\tau t}$.

As in Section \ref{sec:more-roots} we let $\Sigma_s$ and $\Sigma_\ell$ denote the short and long roots. Let $\lambda\in\mathfrak{a}^*_\C$. Following (4.18)--(4.19) of \cite{Shimeno}, for $\alpha\in \Sigma_\ell^+$ we put
\begin{equation}\label{c-tilde}
{\bf c}_\alpha(\lambda,\tau) = \frac{ 2^{1 - \langle \lambda, \alpha^\vee \rangle } \Gamma( \langle \lambda, \alpha^\vee \rangle ) }{ \Gamma( \tfrac{1}{2} + \tfrac{1}{2} \langle \lambda, \alpha^\vee \rangle + \tfrac{1}{2} \ell_\tau ) \Gamma( \tfrac{1}{2} + \tfrac{1}{2} \langle \lambda, \alpha^\vee \rangle - \tfrac{1}{2} \ell_\tau ) },
\end{equation}
whereas for $\alpha\in \Sigma_s^+$ we put ${\bf c}_\alpha(\lambda,\tau)={\bf c}_\alpha(\lambda)$. Then \cite[Lemma 4.8]{Shimeno} and \cite[Theorem 7.4]{Shimeno1990} show that
\begin{equation}\label{c-theta-tau-factorization}
{\bf c}_\Theta(\lambda,\tau)=\prod_{\alpha\in\langle\Theta\rangle}{\bf c}_\alpha(\lambda,\tau),\qquad {\bf c}^\Theta(\lambda,\tau)=\prod_{\alpha\in \Sigma^+\setminus\langle\Theta\rangle}{\bf c}_\alpha(\lambda,\tau).
\end{equation}
Since \eqref{c-tilde} reduces to \eqref{c-tilde-spherical} when $\tau$ is trivial, this definition specializes to the preceding case for type $C$ irreducible root systems.

\subsection{Comparison to literature}\label{sec:Gamma-comparison}

We make a few comments to illustrate that our formulas for ${\bf c}_\alpha(\lambda,\tau)$ agree with those in Shimeno \cite{Shimeno}, and also with \eqref{c-tilde-spherical} in the case of overlap, when $G$ is Hermitian with reduced root system and $\tau$ is trivial.

We begin with the agreement with Shimeno.  If we use the notation $\lambda_i = \langle \lambda, \beta_i^\vee \rangle$ as in \cite[(1.8)]{Shimeno}, our formula \eqref{c-tilde} in the case $\alpha \in \Sigma^+_\ell$ is the same as \cite[(4.19)]{Shimeno} after using $m = 0$ there.  For $\alpha \in \Sigma^+_s$, the formula \cite[(4.18)]{Shimeno} reads
\begin{equation}\label{eq:Shimeno-c-alpha}
{\bf c}_\alpha(\lambda, \tau) = \frac{ 2^{m_\alpha - 1} \Gamma( \tfrac12(m_\alpha + 1) ) \Gamma( \tfrac{1}{2} \langle \lambda, \alpha^\vee \rangle ) }{ \sqrt{\pi} \Gamma( \tfrac{1}{2} m_\alpha + \tfrac{1}{2} \langle \lambda, \alpha^\vee \rangle ) }
\end{equation}
after using $m'=m_\alpha$, while formula \eqref{c-tilde-spherical} for ${\bf c}_\alpha(\lambda, \tau) = {\bf c}_\alpha(\lambda)$ gives
\[
{\bf c}_\alpha(\lambda, \tau) = \frac{ 2^{\tfrac{1}{2} m_\alpha - \tfrac{1}{2} \langle \lambda, \alpha^\vee \rangle } \Gamma( \tfrac12(m_\alpha + 1) ) \Gamma( \tfrac{1}{2} \langle \lambda, \alpha^\vee \rangle ) }{ \Gamma( \tfrac{1}{4} m_\alpha + \tfrac12 + \tfrac{1}{4} \langle \lambda, \alpha^\vee \rangle ) \Gamma( \tfrac{1}{4} m_\alpha + \tfrac{1}{4} \langle \lambda, \alpha^\vee \rangle ) }.
\]
It may be checked that these agree after using the duplication formula $\Gamma(s)\Gamma(s+1/2)=2^{1-2s}\sqrt{\pi}\Gamma(2s)$ for the Gamma function.

We next check that our formulas for ${\bf c}_\alpha(\lambda, \tau)$ agree with \eqref{c-tilde-spherical} when $G$ is Hermitian with reduced root system and $\tau$ is trivial.  When $\alpha \in \Sigma^+_s$, there is nothing to check.  When $\alpha \in \Sigma^+_\ell$, \eqref{c-tilde-spherical} becomes
\[
{\bf c}_\alpha(\lambda) = \frac{ 2^{\tfrac{1}{2} - \tfrac{1}{2} \langle \lambda, \alpha^\vee \rangle } \Gamma( \tfrac{1}{2} \langle \lambda, \alpha^\vee \rangle ) }{ \Gamma( \tfrac{3}{4} + \tfrac{1}{4} \langle \lambda, \alpha^\vee \rangle ) \Gamma( \tfrac{1}{4} + \tfrac{1}{4} \langle \lambda, \alpha^\vee \rangle ) }.
\]
Just as in the case above, this may be converted to
\begin{equation}\label{c-alpha-simplified}
{\bf c}_\alpha(\lambda) = \frac{ \Gamma( \tfrac{1}{2} \langle \lambda, \alpha^\vee \rangle ) }{ \sqrt{\pi} \Gamma( \tfrac{1}{2} + \tfrac{1}{2} \langle \lambda, \alpha^\vee \rangle ) }
\end{equation}
by using the duplication formula.  On the other hand, \eqref{c-tilde} gives
\[
{\bf c}_\alpha(\lambda, \tau_\text{triv} ) = \frac{ 2^{1 - \langle \lambda, \alpha^\vee \rangle } \Gamma( \langle \lambda, \alpha^\vee \rangle ) }{ \Gamma( \tfrac{1}{2} + \tfrac{1}{2} \langle \lambda, \alpha^\vee \rangle )^2 }.
\]
By applying the duplication formula once more to the $\Gamma( \langle \lambda, \alpha^\vee \rangle )$ term in the numerator of this last expression, we see that the two formulas are equal.

\subsection{Spherical inversion formula}\label{sec:tau-triv-inversion}

In this paragraph, we take $\tau$ to be trivial and $G$ general, as spelled out at the beginning of Section \ref{sec:inv-formula}, and we review the classical spherical inversion formula.

When $\tau$ is trivial, we shall often abbreviate $\tpl$ to $\mu_\text{Pl}$. In this case, the Plancherel measure $\mu_\text{Pl}$ was determined by Gangolli \cite[Theorem 3.5]{Gangolli} and Helgason \cite[Ch. IV, \S 7, Theorem 7.1]{Helgason}, as part of their proof of the commutativity of \eqref{eq:commutative-diagram}, and is given by
\[
d \mu_\text{Pl} = |{\bf c}(\lambda)|^{-2}d\lambda,
\]
where $d\lambda$ is the Lebesgue measure on $i \ga^*$. Thus, $d\mu_\text{Pl} = \beta_S(\lambda) d\lambda$, where $\beta_S(\lambda)=|{\bf c}(\lambda)|^{-2}$ is the density function appearing throughout the introduction.

For any $F\in \mathcal{PW}(\mathfrak{a}_\C^*)^W$ we let
\begin{equation}\label{eq:def-J}
\mathscr{J}F(g)=\int_{i\mathfrak{a}^*}F(\lambda)\varphi_\lambda (g)|{\bf c}(\lambda)|^{-2}d\lambda.
 \end{equation}
We have $\mathscr{J}F\in C^\infty_c(S,\tau_{\rm triv},\tau_{\rm triv})$. Recall from \eqref{HC-transform} the definition of the Harish-Chandra transform $\mathscr{H}=\mathscr{H}_{\tau_{\rm triv}}$, as well as the shorthand notation $k\mapsto\hat{k}=\mathscr{H}k$ we often use for it. For $k\in C^\infty_c(S,\tau_{\rm triv},\tau_{\rm triv})$, one has the following inversion formula
\begin{equation}\label{eq:sph-planch-inv}
k(g)=\frac{1}{|W|}\mathscr{J}(\mathscr{H}k)(g)=\frac{1}{|W|}\int_{i\mathfrak{a}^*}\hat{k}(\lambda)\varphi_{\lambda}(g)|\mathbf{c}(\lambda)|^{-2}d\lambda \qquad (g\in G),
\end{equation}
which recovers \eqref{eq:general-inversion} for the trivial $K$-type.

\subsection{$\tau$-spherical Plancherel measure and inversion formula}\label{sec:tau-sph-inversion}

We now let $G$ be Hermitian of reduced root system. The $\tau$-spherical Plancherel measure on $\mathfrak{a}_\C^*$ (where $\tau$ is a character of $K$), was determined in \cite{Shimeno}, using the various ${\bf c}$-functions from Section \ref{sec:c-function}. To describe it, we begin by introducing some notation. 

We maintain the assumption \eqref{eq:assumption} that the root system $\Sigma$ of $G$ is reduced. Recall from Section \ref{sec:more-roots} that $\Sigma$ is of the form \eqref{Shimeno-roots}. In that case the simple roots $\Psi=\{\alpha_1,\ldots ,\alpha_r\}$ are of the form
\begin{equation}\label{eq:alpha-to-beta}
\alpha_1=(\beta_r-\beta_{r-1})/2,\quad \alpha_2=(\beta_{r-1}-\beta_{r-2})/2,\quad \ldots ,\quad \alpha_{r-1}=(\beta_2-\beta_1)/2,\quad \alpha_r=\beta_1.
\end{equation}
Recall from Section \ref{sec:parabolic-notation} the notation associated with subsets $\Theta$ of $\Psi$. We shall be particularly interested in the subsets
\[
\Theta_j=\{\alpha_{r-j+1},\ldots ,\alpha_r\}\qquad (j=1,\ldots ,r).
\]
We write $\Theta_0=\emptyset$. Thus $\dim \mathfrak{a}_{\Theta_j}=r-j$, $\dim \mathfrak{a}(\Theta_j)=j$, and
\begin{align*}
\{0\}=\mathfrak{a}_{\Theta_r}\subset\cdots&\subset\mathfrak{a}_{\Theta_1}\subset \mathfrak{a}_{\Theta_0}=\mathfrak{a},\\
\{0\}=\mathfrak{a}(\Theta_0)\subset\cdots&\subset\mathfrak{a}(\Theta_{r-1})\subset \mathfrak{a}(\Theta_r)=\mathfrak{a}.
\end{align*}
Note that $W_{\Theta_r}=\{1\}$ and $W_{\Theta_0}=W$.

We shall associate with each $j=0,1,\ldots ,r$ a measure $\tpl^{(j)}$, as well as a transform
\[
\mathscr{J}_{\Theta_j,\tau}:  \mathcal{PW}(\mathfrak{a}_\C^*)^W\rightarrow C^\infty(S,\tau,\tau),
\]
given by
\[
\mathscr{J}_{\Theta_j,\tau} F(g) = \frac{|W|}{|W_{\Theta_{r-j}}|} \int_{ \ga^*_\C} F(\lambda) \varphi_{\lambda, \tau}(g) d\tpl^{(j)}(\lambda).
\]
We begin with the extreme cases $j=0$ and $j=r$.

\medskip

--{\sc The case $j=0$} (the ``most continuous part''): We define
\[
d\tpl^{(0)} = |{\bf c}(\lambda,\tau)|^{-2}d\lambda,
\]
where $d\lambda$ is the Lebesgue measure on $i \ga^*$. For any $F\in \mathcal{PW}(\mathfrak{a}_\C^*)^W$, the corresponding transform is given by
\[
\mathscr{J}_{\Theta_0,\tau}F(g)=\int_{i\mathfrak{a}^*}F(\lambda)\varphi_{\lambda,\tau}(g)|{\bf c}(\lambda,\tau)|^{-2}d\lambda,
 \]
which directly generalizes the transform $\mathscr{J}$ in \eqref{eq:def-J} in the spherical case, but with arbitrary abelian $K$-type $\tau$.

\medskip

--{\sc The case $j=r$} (the ``discrete part''): Let $L^2(S,\tau)$ denote the Hilbert space of square-integrable functions on $G$ which translate under $K$ on the right by $\tau^{-1}$. We define the $W$-invariant subset $D_\tau$ of $\mathfrak{a}^*_\C$ as the set of $\lambda\in\mathfrak{a}_\C^*$ for which $\varphi_{\lambda,\tau}\in L^2(S,\tau)$. In this case, the representation $V(\lambda,\tau)$ generated by $\varphi_{\lambda,\tau}$, as described in \S\ref{sec:tau-spherical}, is a holomorphic discrete series representation \cite[Theorem 5.10]{Shimeno}. The set $D_\tau$ was determined in \cite[Theorem 5.1]{Shimeno}; in particular, $D_\tau$ is finite, and lies in $\mathfrak{a}^*$. 

\begin{remark}
For example, for $G=\Sp_{2m}(\R)$, it is shown in \cite[Theorem 5.1]{Shimeno} that $D_\tau$ is empty precisely when $|\ell_\tau|\leqslant m$ \cite[(5.4)]{Shimeno}. When $m=1$ this corresponds to the fact that holomorphic discrete series of $\Sp_2(\R)=\SL_2(\R)$ have weight at least $2$.
\end{remark}

For $\lambda\in D_\tau$ let $d(\lambda,\tau)$ denote the formal degree of the representation $V(\lambda,\tau)$ generated by $\varphi_{\lambda,\tau}$. We define
\[
\tpl^{(r)} = \sum_{\varrho\in D_{\tau}} d(\varrho,\tau) \delta_\varrho,
\]
where $\delta_\varrho$ is the delta measure at $\varrho$.  For $F\in \mathcal{PW}(\mathfrak{a}_\C^*)^W$, the corresponding transform is
\[
\mathscr{J}_{\Theta_r,\tau}F(g)=\sum_{\varrho\in D_{\tau}}d(\varrho,\tau)F(\varrho)\varphi_{\varrho,\tau}(g).
\]

\medskip

-- {\sc The intermediate cases $j=1,\ldots ,r-1$}: For each $j=1,\ldots ,r-1$, we define a discrete subset of $\mathfrak{a}(\Theta_j)^*$ as follows. Recall from Section \ref{sec:more-roots} that the positive long roots $\Sigma_\ell^+=\{\beta_1,\ldots ,\beta_r\}$ form a basis of $\mathfrak{a}^*$. For $\lambda\in\mathfrak{a}^*_\C$ recall furthermore the notation $\lambda_i= \langle\lambda,\beta_i^\vee\rangle$, where $\beta_i^\vee=2\beta_i/\langle\beta_i,\beta_i\rangle$. Following \cite[(6.13)]{Shimeno} we put\footnote{There is a misprint in \cite[(6.13)]{Shimeno}: the condition $\lambda_r<0$ should be $\lambda_j<0$.}
\begin{equation}\label{defn-D-tau-j}
D_{\tau,j}=\left\{\varrho\in \mathfrak{a}(\Theta_j)^*: \varrho_1+|\ell_\tau|-1\in 2\N,\, \varrho_{i+1}-\varrho_i-m'\in 2\N \, (1\leqslant i\leqslant j-1),\, \varrho_j<0\right\},
\end{equation}
where $m'$ is the common value of the multiplicity $m_\alpha$ for any $\alpha\in\Sigma_s$. This is a finite set (possibly empty, if $|\ell_\tau|$ is small enough), as any $\varrho\in D_{\tau,j}$ must in particular satisfy $1-|\ell_\tau|\leqslant \varrho_1<\varrho_2<\cdots <\varrho_j<0$ with $\varrho_i + |\ell_\tau|$ integral.

For $\lambda\in\mathfrak{a}(\Theta_j)^*$ we let
\[
d_j(\lambda,\tau)=(-2\pi i)^j{\rm Res}_{\mu_j=\lambda_j}\cdots {\rm Res}_{\mu_1=\lambda_1} ({\bf c}_{\Theta_j}(\mu,\tau)^{-1}{\bf c}_{\Theta_j}(-\mu,\tau)^{-1}).
\]
We define $\tpl^{(j)}$ by
\[
\int_{\ga^*_\C} F(\lambda) d \tpl^{(j)} = \sum_{\varrho\in D_{\tau,j}}d_j(\varrho,\tau)\int_{i\mathfrak{a}^*_{\Theta_j}}F(\varrho+\lambda_{\Theta_j}) |{\bf c}^{\Theta_j}(\varrho+\lambda_{\Theta_j},\tau)|^{-2} d\lambda_{\Theta_j},
\]
with corresponding transform
 \[
\mathscr{J}_{\Theta_j,\tau}F(g)=\frac{|W|}{|W_{\Theta_{r-j}}|}\sum_{\varrho\in D_{\tau,j}}d_j(\varrho,\tau)\int_{i\mathfrak{a}^*_{\Theta_j}}F(\varrho+\lambda_{\Theta_j})\varphi_{\varrho+\lambda_{\Theta_j},\tau}(g)|{\bf c}^{\Theta_j}(\varrho+\lambda_{\Theta_j},\tau)|^{-2}d\lambda_{\Theta_j}.
 \]

\begin{remark}
The formulae from the intermediate cases $j=1,\ldots ,r-1$ continue to make sense for $j=r$ and indeed coincide with the separate definition we gave. We have only separated out the $j=r$ case for expository clarity. Indeed, from \cite[(5.14), (6.13)]{Shimeno} one has $D_{\tau,r}=D_\tau$, and for $\varrho\in D_\tau$ one has $d(\varrho,\tau)^{-1}=\|\varphi_{\varrho,\tau}\|_2^2$, from \cite[Remark 6.9]{Shimeno}.
\end{remark}

With each of these transforms defined, we now put $\tpl = \sum_{j = 0}^r \tpl^{(j)}$, and
\begin{equation}\label{eq:J-tau-transform}
\mathscr{J}_\tau=\sum_{j=0}^r \mathscr{J}_{\Theta_j,\tau}.
\end{equation}
For $k\in C^\infty_c(S,\tau,\tau)$, one has the following inversion formula
\begin{equation}\label{spherical-inversion}
k=\frac{1}{|W|}\mathscr{J}_\tau(\mathscr{H}_\tau k),
\end{equation}
due to Shimeno \cite[Theorem 6.7]{Shimeno}. In this way, the inversion formula \eqref{spherical-inversion} is expressed in the form announced in \eqref{eq:general-inversion} (and it of course coincides with the case of trivial $K$-type in Section \ref{sec:tau-triv-inversion}).

\section{$\tau$-spherical test functions}\label{sec:test-function}

We continue to assume that $G$ is simple, while invoking assumption \eqref{eq:assumption} when dealing with non-trivial (one-dimensional) $K$-types.

This section provides for certain $\tau$-spherical test functions which will be of use in the next two sections. Indeed, for the proof of Proposition \ref{Weyl-upper-bd}, an important ingredient in Theorem \ref{sup-thm}, we shall need test functions whose Harish-Chandra transforms have pleasing positivity and concentration properties on the unitary dual. Moreover, we shall need similar test functions in the spherical case in the proofs of Lemma \ref{upper-local-Weyl} and Proposition \ref{local-Weyl}.

The outline of this section is as follows.  Our test functions are required to be positive on a subset $\ga^*_{\tau, \text{un}}$ of $\ga^*_\C$ (the \textit{$\tau$-unitary dual}), defined in Definition \ref{tau-unitary-def}, and we begin in Sections \ref{sec:taubounded} and \ref{ASP} by proving some bounds for $\ga^*_{\tau, \text{un}}$.  We construct our test functions in Section \ref{sec:test-fn}.  Section \ref{sec:test-fn-bds} contains bounds for our test functions, and as we prove these by inverting the Harish-Chandra transform, we first establish, in Section \ref{sec:c-fun-estimates}, some estimates for the various Plancherel measures appearing in the inversion formulas.

\subsection{Boundedness of the $\tau$-spherical function}
\label{sec:taubounded}

The following result will be useful in controlling the $\tau$-unitary dual $\ga^*_{\tau, \text{un}}$.

\begin{prop}\label{prop:criterion-4-bounded}
There are $a,b\in\R$, $b>0$, depending only on the root system of $G$, such that, if the $\tau$-spherical function $\varphi_{\lambda,\tau}$ is bounded then $\| \textup{Re}\lambda \| \leqslant a + b | \ell_\tau|$.
\end{prop}
\begin{proof}

We begin by recalling that if $\tau$ is trivial, then a theorem of Helgason \cite[Ch. IV, Thm 8.1]{Helgason} states that $\varphi_\lambda$ is bounded if and only if $\text{Re} \lambda$ lies in the convex hull of $W \rho$.  We may therefore assume that $\tau$ is nontrivial, and that $G$ is Hermitian with reduced root system.

We shall deduce the proposition from an asymptotic formula for $\varphi_{\lambda,\tau}$ due to Shimeno \cite[Prop 4.9]{Shimeno}, which we now recall.  This formula is stated in terms of the partial spherical function $\varphi_{\lambda, \tau}^\Theta$ associated to a subset $\Theta \subset \Psi$, defined in \cite[(4.22)]{Shimeno}.  Let $\Theta \subset \Psi$ be a subset of the simple roots and let $\lambda \in \ga_\C^*$ be such that $\text{Re} \langle \lambda, \alpha \rangle > 0$ for all $\alpha \in \Sigma^+ \setminus \langle \Theta \rangle$. Recall the subgroup $A_\Theta$ of $A$, whose Lie algebra is given in \eqref{defn:Lie-A-theta}. Then applying \cite[Prop 4.9]{Shimeno} to the function $f = 1_{\lambda, \tau}$, and using $\varphi_{\lambda, \tau}^\Theta(e) = 1$, which follows easily from the definition, we obtain
\begin{equation}
\label{phi-limit}
\lim_{a\, \underset{\Theta}{\rightarrow}\, \infty} a^{\rho - \lambda} \varphi_{\lambda, \tau}(a) = {\bf c}^\Theta(\lambda, \tau) \varphi^\Theta_{\lambda, \tau}(e) = {\bf c}^\Theta(\lambda, \tau).
\end{equation}
Here, the $\Theta$ subscript in the limit means that $a \in A_\Theta$ and $a^\alpha \to \infty$ for all $\alpha \in \Psi \setminus \Theta$.  

Note that $\lambda$ is not a pole of $c_\alpha(\lambda, \tau)$ for any $\alpha \in \Sigma^+ \setminus \langle \Theta \rangle$, by our assumption that $\text{Re} \langle \lambda, \alpha \rangle > 0$ for these $\alpha$; the right-hand side of this formula is therefore finite.

We now return to the proof of the proposition.  We let $\lambda \in \ga_\C^*$ be a spectral parameter for which $\varphi_{\lambda,\tau}$ is bounded, and assume without loss of generality that $\text{Re}\, \lambda \in \overline{\ga_+^*}$.  We shall prove the existence of $a,b\in\R$, $b>0$, such that, either $\text{Re} \langle \lambda, \alpha \rangle \leqslant a + b |\ell_\tau|$ for all simple roots $\alpha \in \Psi$, or that $\text{Re}\, \lambda(H_\alpha) \leqslant \rho(H_\alpha)$ for some $\alpha \in \Psi$.  Moreover, either of these imply the proposition; this is easy to see for the first condition, as $\Psi$ forms a basis for $\ga^*$.  For the second condition, recall the notation for the root system $\Sigma$ introduced in Section \ref{sec:more-roots}, in which $\Psi$ is given by \eqref{eq:alpha-to-beta}. If we write $\lambda = \sum \tfrac12 \lambda_i \beta_i$, where $\lambda_i=\langle \lambda, \beta_i^\vee\rangle$ as usual, then the condition $\text{Re}\, \lambda \in \overline{\ga_+^*}$ is equivalent to $0 \leqslant \text{Re}\, \lambda_1 \leqslant \cdots \leqslant \text{Re}\, \lambda_r$.  Moreover, if we let $e_i$ be the basis for $\ga$ dual to $\beta_i$, then we have
\[
H_{\alpha_i} = 2 \sum_{j = r +1 - i}^r e_i, \quad 1 \leqslant i < r, \quad H_{\alpha_r} = \sum_{i = 1}^r e_i.
\]
It follows that a bound for $\text{Re}\, \lambda(H_\alpha)$ for some $\alpha \in \Psi$ implies a bound on all $\text{Re}\, \lambda_i$ as required.

We may assume that $\text{Re}\, \lambda(H_\alpha) > \rho(H_\alpha)$ for all $\alpha \in \Psi$, and wish to find $a,b$ such that $\text{Re} \langle \lambda, \alpha \rangle \leqslant a + b |\ell_\tau|$ for all $\alpha \in \Psi$. Let $\alpha\in \Psi$ and $\Theta = \Psi \setminus \{ \alpha \}$. Then $\text{Re} \langle \lambda, \alpha \rangle \leqslant \text{Re} \langle \lambda, \beta\rangle$ for any $\beta\in \Sigma^+ \setminus \langle \Theta \rangle$, since $\beta - \alpha$ is a non-negative linear combination of simple roots. It therefore suffices to find $a, b$, as well as some $\beta\in \Sigma^+ \setminus \langle \Theta \rangle$, for which $\text{Re} \langle \lambda, \beta\rangle\leqslant a+b|\ell_\tau|$.

Recall the assumption that $\text{Re}\,\lambda\in\overline{\ga_+^*}$. If $\text{Re} \langle \lambda, \alpha \rangle = 0$, then we are done. Otherwise $\text{Re} \langle \lambda,\alpha\rangle > 0$, and hence $\text{Re} \langle \lambda, \beta\rangle > 0$ for all $\beta\in \Sigma^+ \setminus \langle \Theta \rangle$. We may therefore apply \eqref{phi-limit} to this choice of $\Theta$. Because $A_\Theta$ is one-dimensional and spanned by $H_{\alpha}$, our assumption that $\text{Re}\, \lambda( H_{\alpha}) > \rho(H_{\alpha})$ means that $a^{\rho - \lambda} \to 0$ as $a \to \infty$ in $A_\Theta$. Since $\varphi_{\lambda, \tau}$ is bounded, \eqref{phi-limit} implies that ${\bf c}^\Theta(\lambda, \tau) = 0$, and hence that ${\bf c}_{\beta}(\lambda, \tau) = 0$ for some $\beta\in \Sigma^+ \setminus \langle \Theta \rangle$.

If $\beta\in \Sigma^+_s$ is a short positive root, the formula for ${\bf c}_{\beta}(\lambda, \tau)$ from Section \ref{sec:Gamma-comparison} shows that ${\bf c}_{\beta}(\lambda, \tau)$ can only be zero when $\Gamma( \tfrac12 m_{\beta} + \tfrac12 \langle \lambda, \beta^\vee \rangle )$ has a pole, but this can't happen because $\text{Re} \langle \lambda, \beta\rangle > 0$. If $\beta\in \Sigma^+_\ell$ is a long positive root, we likewise have that ${\bf c}_{\beta}(\lambda, \tau)$ can only be zero when one of
\[
\Gamma( \tfrac{1}{2} + \tfrac{1}{2} \langle \lambda, \beta^\vee \rangle + \tfrac{1}{2} \ell_\tau ) \quad \text{and} \quad \Gamma( \tfrac{1}{2} + \tfrac{1}{2} \langle \lambda, \beta^\vee \rangle - \tfrac{1}{2} \ell_\tau ) 
\]
has a pole.  If we assume without loss of generality that $\ell_\tau \geqslant 0$, then the first Gamma factor again cannot have a pole by positivity.  For the second factor to have a pole, we must have
\[
\tfrac{1}{2} + \tfrac{1}{2} \text{Re} \langle \lambda, \beta^\vee \rangle - \tfrac{1}{2} \ell_\tau \leqslant 0.
\]
This implies that $\text{Re} \langle \lambda, \beta^\vee \rangle \leqslant \ell_\tau - 1$, as required.
\end{proof}

\subsection{The $\tau$-unitary dual}\label{ASP}

We now recall the definition of the $\tau$-unitary dual.

\begin{definition}
\label{tau-unitary-def}

Let $\tau$ be an abelian $K$-type. The set 
\[
\mathfrak{a}_{\tau, {\rm un}}^* =\{\lambda\in\mathfrak{a}^*_\C:  \varphi_{\lambda, \tau} \;\textrm{is positive definite} \}
\]
is the {\rm unitary spectrum} of $\mathbb{D}(\tau)$. For $\tau$ trivial we abbreviate this to $\mathfrak{a}_{\rm un}^*$.
\end{definition}

\begin{lemma}\label{lemma:unitary}
There are constants $a,b>0$, depending only on $G$, such that
\[
\mathfrak{a}_{\tau,\rm un}^*\subset \{\lambda\in\ga^*_\C: \| \textup{Re} \lambda \| \leqslant a + b | \ell_\tau|\}.
\]
\end{lemma}

\begin{proof}
This follows from Proposition \ref{prop:criterion-4-bounded} since positive definite functions are bounded.
\end{proof}

\begin{lemma}\label{lem:unitary-hermitian}
If $\lambda \in\mathfrak{a}_{\tau,\rm{un}}^*$ then $W\lambda = -W\bar\lambda$.
\end{lemma}

\begin{proof}
We note that for $\lambda\in\mathfrak{a}_{\tau, {\rm un}}^*$ we have
\begin{equation}\label{eq:sph-unitary-lambda}
\varphi_{\lambda,\tau}(g)=\varphi_{\bar\lambda,\tau^{-1}}(g^{-1}) = \varphi_{-\bar\lambda, \tau}(g).
\end{equation}
Indeed, we have $\varphi_{\lambda, \tau}(g) = \overline{\varphi_{\lambda, \tau}(g^{-1})}$ because $\varphi_{\lambda, \tau}$ is positive definite, and the formula \eqref{defn-sph-fn} shows that $\overline{\varphi_{\lambda, \tau}(g^{-1})} = \varphi_{\bar\lambda,\tau^{-1}}(g^{-1})$, which gives the first formula. The second formula is Proposition \ref{prop:sph-fn-2-vars}. Combining \eqref{eq:sph-unitary-lambda} with \eqref{eq:equality-sph-fn} yields the lemma.
\end{proof}

\subsection{Existence of $\tau$-spherical test functions}\label{sec:test-fn}
Using the preliminaries of the previous sections, we now arrive at the existence of $\tau$-spherical approximate spectral projectors.

\begin{prop}\label{test-fn}
There are constants $R_0,c>0$ such that the following holds. Let $\nu\in i\mathfrak{a}^*$ and $\tau\in\widehat{K}$ be one-dimensional, subject to the assumption \eqref{eq:assumption}. For all $0<R\leqslant R_0$ there is a function $k_\nu\in C^\infty_R(S,\tau,\tau)$ whose Harish-Chandra transform $\widehat{k}_\nu(\,\cdot\,;\tau)\in\mathcal{PW}(\mathfrak{a}_\C^*)^W_R$ satisfies
\begin{enumerate}
\item\label{4} $\widehat{k}_\nu(\lambda;\tau)\ll_{A,R} \exp(R\|{\rm Re}\,\lambda\|)\displaystyle\sum_{w\in W}\big(1+\|w\lambda-\nu\|\big)^{-A}$ for all $\lambda\in  \mathfrak{a}_\C^*$;
\item\label{1} $\widehat{k}_\nu(\lambda;\tau)\geqslant 0$ for all $\lambda\in\mathfrak{a}_{\tau, {\rm un}}^*$;
\item\label{3} $\widehat{k}_\nu(\lambda;\tau) \geqslant c > 0$ for all $\lambda\in\mathfrak{a}_{\tau, {\rm un}}^*$ with $\|{\rm Im}\,\lambda-\nu\|\leqslant 1$.
\end{enumerate}
Here, we are writing $\lambda={\rm Re}\,\lambda+{\rm Im}\,\lambda\in\mathfrak{a}^*+i\mathfrak{a}^*$.
\end{prop}

\begin{proof}
We begin by constructing the function $k_\nu$, using the commutative diagram from \eqref{eq:commutative-diagram}. We then verify each property in turn. We follow closely \cite[\S 4]{BM}. The proof uses the fact that the $\tau$-spherical Harish-Chandra transform $\mathscr{H}_\tau$ is an algebra isomorphism (preserving support conditions), but does not use the inversion formula \eqref{eq:general-inversion}.

Let $R>0$. Let $g_0\in C^\infty_c(\mathfrak{a})$ be non-negative, even, supported in $B_{\mathfrak{a}}(0,R/4)$, and satisfy $\int_{\mathfrak{a}} g_0 = 1$. Let $g = g_0 * g_0$. Then $g \in C^\infty_0(\ga)$ is non-negative, even, supported in $B_{\mathfrak{a}}(0,R/2)$, and satisfies $\int_\mathfrak{a} g = 1$. The Fourier transform $h=\mathscr{F}g\in\mathcal{P}(\mathfrak{a}_\C^*)_{R/2}$ is even and non-negative on $i\mathfrak{a}^*$, satisfies $h(\overline{\lambda})=\overline{h(\lambda)}$ on $\ga_\C^*$, and is normalized such that $h(0)=1$. We center $h$ at the fixed tempered parameter $\nu\in i\ga^*$, and force $W$-invariance, by putting
\begin{equation}\label{defn-h-nu-0}
h_\nu^0(\lambda)=\sum_{w\in W}h(w\lambda-\nu).
\end{equation}
Then $h_\nu^0\in\mathcal{P}(\mathfrak{a}_\C^*)_{R/2}^W$; since $\nu\in i\ga^*$, we have $h_\nu^0(-\overline{\lambda})=\overline{h_\nu^0(\lambda)}$. We let $k_\nu^0(\cdot, \tau)=\mathscr{J}_\tau (h_\nu^0)\in C^\infty_{R/2}(S,\tau,\tau)$ be its inverse $\tau$-spherical transform, defined in \eqref{eq:def-J} for $\tau$ trivial and $G$ arbitrary, and in \eqref{eq:J-tau-transform} for $\tau$ non-trivial and $G$ Hermitian with reduced root system. Finally we put
\[
k_\nu(\cdot, \tau)=k_\nu^0(\cdot, \tau)\ast k_\nu^0(\cdot, \tau)\in C^\infty_R(S,\tau,\tau).
\]
We deduce that $\widehat{k}_\nu(\lambda;\tau)=h_\nu^0(\lambda)^2\in \mathcal{P}(\mathfrak{a}_\C^*)_R^W$. 

Since $h\in \mathcal{P}(\mathfrak{a}_\C^*)_{R/2}$ it satisfies the Paley--Wiener estimate
\[
h(\lambda)\ll_{A,R} (1+\|\lambda\|)^{-A}\exp(R\|{\rm Re}\, \lambda\|/2).
\]
Thus
\[
h^0_\nu(\lambda)\ll_{A,R} \exp(R\|{\rm Re}\, \lambda\|/2)\sum_{w\in W}(1+\|w\lambda-\nu\|)^{-A}.
\]
Squaring this, and using $(1+\|w_1\lambda-\nu\|)(1+\|w_2\lambda-\nu\|)\geqslant (1+\|w\lambda-\nu\|)^2$ for the off-diagonal terms, where $w$ realizes $\min\{\|w_1\lambda-\nu\|,\|w_2\lambda-\nu\|\} $, proves \eqref{4}.

We deduce from the $W$-invariance of $h^0_\nu$ and Lemma \ref{lem:unitary-hermitian} that for $\lambda \in \ga^*_{\tau, {\rm un}}$, we have $h_\nu^0(\lambda)=h_\nu^0(-\overline{\lambda})=\overline{h_\nu^0(\lambda)}$. Hence $h_\nu^0$ is real-valued on $\ga^*_{\tau, {\rm un}}$. Since $\widehat{k}_\nu(\lambda;\tau)=h_\nu^0(\lambda)^2$, this establishes Property \eqref{1}.

For Property \eqref{3}, let $a, b>0$ be as in Proposition \ref{prop:criterion-4-bounded}. We shall show that for $\lambda\in\ga^*_\C$ with $\|{\rm Re}\,\lambda\|\leqslant a+b|\ell_\tau|$ and $\|{\rm Im}\,\lambda-\nu\|\leqslant 1$ we have ${\rm Re}\, h_\nu^0(\lambda)\geqslant 1/4$. Since unitary parameters satisfy the first bound by Lemma \ref{lemma:unitary}, and $h_\nu^0$ is real-valued on $\ga^*_{\tau, {\rm un}}$, we have $\widehat{k}_\nu(\lambda;\tau)=h_\nu^0(\lambda)^2 = ({\rm Re}\, h_\nu^0(\lambda))^2\geqslant 1/16$ for all $\lambda\in\ga^*_{\tau, {\rm un}}$ with $\|{\rm Im}\,\lambda-\nu\|\leqslant 1$, as desired.

We begin by remarking that, by the normalization $h(0)=1$, there is small enough $R_1=R_1(G,\tau)>0$ such that, if $0<R\leqslant R_1$, then ${\rm Re}\, h(\lambda) \geqslant 1/2$, say, for all $\lambda\in\mathfrak{a}_\C^*$ with $\| {\rm Im}\, \lambda \| \leqslant 1$ and $\| {\rm Re}\, \lambda \| \leqslant a+b|\ell_\tau|$. Taking real parts in the definition \eqref{defn-h-nu-0}, we deduce that for all $\lambda\in\mathfrak{a}_\C^*$ with $\|{\rm Im}\,\lambda-\nu\|\leqslant 1$ and $\|{\rm Re}\, \lambda\|\leqslant a+b|\ell_\tau|$, we have
\[
{\rm Re}\,h_\nu^0(\lambda)= \sum_{w \in W}{\rm Re}\, h(w\lambda - \nu)  \geqslant \frac12 + (|W|-1)\min_{w\in W\setminus\{1\}}{\rm Re}\, h(w\lambda - \nu).
\]
We claim that there is $C>1$ such that, for $R$ small enough, ${\rm Re}\, h(\lambda) \geqslant -CR$ for all $\lambda\in\mathfrak{a}_\C^*$ with $\|{\rm Re}\,\lambda \| \leqslant a+b|\ell_\tau|$. Indeed, for any $\lambda=\lambda_\Re+i\lambda_\Im\in\mathfrak{a}_\C^*$, we have
\begin{align*}
h(\lambda) & = \int_\ga g(H) e^{-\lambda(H)} dH \\
& = \int_\ga g(H) [ e^{- i \lambda_\Im(H)} + ( e^{-\lambda(H)} - e^{-i\lambda_\Im(H)} ) ] dH \\
& = h(i\lambda_\Im) + \int_\ga g(H)  e^{-i\lambda_\Im(H)} ( e^{-\lambda_\Re (H)} - 1 ) dH.
\end{align*}
Now assume $0<R\leqslant R_1<(a+b|\ell_\tau|)^{-1}$. From $\| \lambda_\Re\| \leqslant a+b|\ell_\tau|$ we deduce the existence of a constant $C>1$, depending only on $G$ and $\tau$, such that $| e^{-\lambda_\Re(H)} - 1 | \leqslant C R$ for all $H \in \text{supp}(g)\subset B_{\mathfrak{a}}(0,R/2)$. Thus, recalling that $\int_\ga g=1$, we get ${\rm Re}\,h(\lambda) \geqslant h(i\lambda_\Im) - CR$. Since $h$ is positive on $i\mathfrak{a}^*$ we deduce ${\rm Re}\,h(\lambda) \geqslant -CR$, as required.

Now, since $\nu\in i\ga^*$ and $\|\cdot \|$ is $W$-invariant, we have $\|{\rm Re}\, (w\lambda-\nu)\|=\|{\rm Re}\, w\lambda\|=\|{\rm Re}\,\lambda\|\leqslant a+b|\ell_\tau|$ for all $w\in W$. It therefore follows from the above claim that if $0<R\leqslant R_0$, where $R_0= \min\{R_1, \frac{1}{4(|W|-1)C}\}$, then ${\rm Re}\, h_\nu^0(\lambda)\geqslant 1/4$, finishing the proof.
\end{proof}

\subsection{Estimates on ${\bf c}$-functions}\label{sec:c-fun-estimates}

We now establish some useful estimates for the $\tau$-spherical Plancherel measure. 

\subsubsection{The spherical Plancherel majorant}
We begin by estimating the spherical density function $|{\bf c}(\lambda)|^{-2}$ on the support of $\mu_{\rm Pl}$, namely the tempered subspace $i\mathfrak{a}^*$ of $\mathfrak{a}_\C^*$. The results in this paragraph are contained in \cite{DKV}, but we review them for completeness.

For $\alpha\in\Sigma$ let $d_\alpha=m_\alpha+m_{2\alpha}$. Let $\Theta\subset \Psi$. Recall the notation $\Sigma_{\rm red}^+$ from \S\ref{sec:more-roots} and $\langle\Theta\rangle_{\rm red}$ from \S\ref{sec:c-function}. For $\lambda\in i\mathfrak{a}^*$ let
\[
\tilde\beta_S^\Theta(\lambda)=\prod_{\alpha\in\Sigma_{\rm red}^+\setminus\langle\Theta\rangle_{\rm red}}(1+|\langle\lambda,\alpha^\vee\rangle|)^{d_\alpha}.
\]
When $S=\mathbb{H}^{n,r}$ and $\mathcal{H}^m$ the following bounds recover \eqref{first-beta-eq} and \eqref{second-beta-eq}, respectively.

\begin{lemma}\label{lemma:c-bd}
${\, }$
\begin{enumerate}
\item\label{c-bd1} For $\lambda\in i\mathfrak{a}^*$ we have $|{\bf c}^\Theta(\lambda)|^{-2}\ll \tilde\beta_S^\Theta(\lambda)$. 
\item\label{c-bd2} For any $\reg > 0$, if $\lambda\in i\mathfrak{a}^*$ is $\reg$-regular, then we have $|{\bf c}^\Theta(\lambda)|^{-2}\asymp_\reg \tilde\beta_S^\Theta(\lambda)$.

\item\label{c-bd3} For any $c > 0$ and all $\lambda \in i \ga^*$, we have $\int_{ \| \nu - \lambda \| < c } |{\bf c}^\Theta(\nu)|^{-2} d\nu \asymp_c \tilde\beta_S^\Theta(\lambda)$.

\end{enumerate}
\end{lemma}

\begin{proof}
The first estimate follows by directly applying Stirling's formula to \eqref{c-tilde-spherical}. The second follows from \cite[(3.44a)]{DKV}.  For the third, the upper bound follows from \eqref{c-bd1} and Lemma \ref{c-bound} below.  For the lower bound, we choose $\reg >0$ small enough that, for any $\lambda$, at least half the points in the ball $\| \nu - \lambda \| < c$ are $\reg$-regular.  The bound then follows from \eqref{c-bd2} and Lemma \ref{c-bound}.
\end{proof}

We now give some estimates on the spherical Plancherel majorant $\tilde\beta_S(\lambda)$ (and the related functions $\tilde\beta_S^\Theta(\lambda)$) on the entire space $\mathfrak{a}_\C^*$.

\begin{lemma}
\label{c-bound}
For all $\lambda,\mu\in\mathfrak{a}_\C^*$ we have
\[
\tilde\beta_S^\Theta(\lambda+\nu)\ll (1+\|\lambda\|)^{d^\Theta(\Sigma)} \tilde\beta_S^\Theta(\nu),
\]
where $d^\Theta(\Sigma)=\sum_{\alpha\in \Sigma^+_{\rm red}\setminus\langle\Theta\rangle_{\rm red}}d_\alpha$. In particular, we have $\tilde\beta_S^\Theta(\lambda)\ll (1+\|\lambda\|)^{d^\Theta(\Sigma)}$.
\end{lemma}

\begin{proof}
Let $\alpha\in\Sigma$. Then
\[
 1+|\langle \lambda +\nu,\alpha^\vee\rangle |  \leqslant 1+  |\langle \lambda ,\alpha^\vee\rangle | + |\langle\nu,\alpha^\vee\rangle |
 \leqslant  (1+ \|\lambda\|) +|\langle\nu,\alpha^\vee\rangle |
  \leqslant (1+ \|\lambda\|)(1+|\langle\nu,\alpha^\vee\rangle |),
\]
which implies the lemma.
\end{proof}

Clearly $\tilde\beta_S^\Theta(\lambda)\leqslant\tilde\beta_S(\lambda)$ for all $\lambda\in\mathfrak{a}_\C^*$ and all $\Theta\subset\Psi$.

\subsubsection{In which $G$ is Hermitian}

We now wish to bound the $\tau$-spherical Plancherel measure $\mu_{ \text{Pl}, \tau}$ on its support, when $G$ is Hermitian with reduced root system and $\tau$ is a character of $K$. Examining the formulas for each component $\mu_{ \text{Pl}, \tau}^{(j)}$ of $\mu_{ \text{Pl}, \tau}$ given in Section \ref{sec:tau-sph-inversion}, we see that this is equivalent to bounding the growth of $| {\bf c}^{ \Theta_j}(\lambda; \tau) |^{-2}$ for $0 \leqslant j \leqslant r$ and $\lambda \in D_{\tau,j} + i \ga^*_{\Theta_j}$. 

\begin{lemma}
\label{lemma:c-theta-bd}
Let $0 \leqslant j \leqslant r$. For $\lambda \in D_{\tau,j} + i \ga^*_{\Theta_j}$, ${\bf c}^{ \Theta_j}(\lambda; \tau)^{-1}$ is well-defined and
\[
| {\bf c}^{ \Theta_j}(\lambda; \tau) |^{-2} \ll_\tau \widetilde{\beta}^{\Theta_j}_S(\lambda).
\]

\end{lemma}

\begin{proof}

We recall the definition \eqref{defn-D-tau-j} of $D_{\tau,j}$. Because $D_{\tau, j}$ is finite, it suffices to prove the lemma for $\lambda$ of the form $\varrho + \lambda_{\Theta_j}$ with $\varrho \in D_{\tau,j}$ fixed and $\lambda_{\Theta_j} \in i \ga^*_{\Theta_j}$.  Note that $\Sigma^+\setminus\langle\Theta_j\rangle=\{\beta_{j+1},\ldots,\beta_r\}\cup\{(\beta_p\pm\beta_q)/2: p>j, q<p\}$.  Let $\Xi$ denote the subset of $\Sigma^+\setminus\langle\Theta_j\rangle$ consisting of roots of the form $\alpha = \beta_p$, $j < p$, or $\alpha=(\beta_p+\beta_q)/2$, $q\leqslant j<p$, and $\Xi^c$ for its complement in $\Sigma^+\setminus\langle\Theta_j\rangle$.  We further subdivide $\Xi$ into the sets $\Xi_p$ for $j < p$, where $\Xi_p$ contains the roots $\beta_p$ and $(\beta_p+\beta_q)/2$, $q \leqslant j$.  Define ${\bf c}_p(\lambda, \tau) = \prod_{\alpha \in \Xi_p} {\bf c}_\alpha(\lambda,\tau)$.  By the product formula for ${\bf c}^{ \Theta_j}(\lambda; \tau)$, it suffices to show that ${\bf c}_p(\lambda,\tau)$ is regular on $D_{\tau,j} + i \ga^*_{\Theta_j}$ for all $p$ and satisfies
\be
\label{cpbd}
{\bf c}_p(\lambda,\tau)\ll \prod_{\alpha\in\Xi_p}(1 + | \langle \lambda, \alpha^\vee \rangle |)^{m_\alpha},
\ee
and that, for each $\alpha\in\Xi^c$, the function ${\bf c}_\alpha(\lambda,\tau)$ is regular on $D_{\tau,j} + i \ga^*_{\Theta_j}$ and satisfies
\be
\label{calphabd}
{\bf c}_\alpha(\lambda,\tau)\ll (1 + | \langle \lambda, \alpha^\vee \rangle |)^{m_\alpha}.
\ee
It will be enough to establish the regularity, after which an application of Stirling's formula will yield \eqref{cpbd} and \eqref{calphabd}.

We begin with the case of $\alpha\in\Xi^c$. Such $\alpha$ all lie in $\Sigma^+_s$, and either have the form $(\beta_p - \beta_q)/2$ with $p>j$ and $q<p$, or $(\beta_p + \beta_q)/2$ with $j < q < p$.  In this case \eqref{c-tilde-spherical} (or, equivalently, \eqref{eq:Shimeno-c-alpha}) shows that
\[
{\bf c}_\alpha( \lambda; \tau)={\bf c}_\alpha( \lambda)= c \frac{ \Gamma( \tfrac{1}{2} \langle \lambda, \alpha^\vee \rangle )}{ \Gamma( \tfrac{1}{2} m_\alpha + \tfrac{1}{2} \langle \lambda, \alpha^\vee \rangle ) }
\]
for some constant $c \neq 0$.  For the regularity of ${\bf c}_\alpha( \lambda; \tau)^{-1}$ it suffices to show that (the real part of) the argument $\frac12 m_\alpha+\text{Re} \langle \lambda, \alpha^\vee \rangle = \frac12 m_\alpha+ \langle \varrho, \alpha^\vee \rangle$ is positive.  In the case $\alpha = (\beta_p - \beta_q)/2$, a calculation gives
\begin{equation}\label{eq:rho-alpha-check}
\langle \varrho, \alpha^\vee \rangle = \frac{ 2 \langle \varrho, \alpha \rangle}{ \langle \alpha, \alpha \rangle} = \varrho_p - \varrho_q = -\varrho_q \geqslant 0,
\end{equation}
while if $\alpha = (\beta_p + \beta_q)/2$ we likewise have $\langle \varrho, \alpha^\vee \rangle = \varrho_p + \varrho_q = 0$.  In either case we have the positivity of $\frac12 m_\alpha+ \langle \varrho, \alpha^\vee \rangle$, as required.

We now consider ${\bf c}_p(\lambda,\tau)$.  As in \cite[p. 380]{Shimeno}, the proof of the regularity of ${\bf c}_p( \lambda; \tau)^{-1}$ will be divided according to the parity the common value $m'$ of the multiplicities $m_\alpha$ of the short roots $\alpha$.

First we consider the case when $m'$ is even. Under this condition, for any $\alpha\in\Sigma_s^+$ the factor ${\bf c}_\alpha(\lambda,\tau)^{-1}$ is given, up to a non-zero constant, by $\langle \lambda,\alpha^\vee\rangle(\langle \lambda,\alpha^\vee\rangle+2)\cdots (\langle \lambda,\alpha^\vee\rangle+m'-2)$, which is holomorphic. To establish the holomorphy of ${\bf c}_p(\lambda,\tau)^{-1}$ it therefore suffices to establish it for the factor ${\bf c}_{\beta_p}(\lambda,\tau)^{-1}$ corresponding to the unique long root in $\Xi_p$. Up to holomorphic factors, ${\bf c}_{\beta_p}(\lambda,\tau)^{-1}$ is equal to
\[
\Gamma( \tfrac{1}{2} + \tfrac{1}{2} \lambda_p + \tfrac{1}{2} \ell_\tau ) \Gamma( \tfrac{1}{2} + \tfrac{1}{2} \lambda_p - \tfrac{1}{2} \ell_\tau ) \Gamma( \lambda_p )^{-1}.
\]
Because $\lambda_p$ is purely imaginary, the only place at which this expression can have a pole is at $\lambda_p = 0$.  Moreover, if we assume without loss of generality that $\ell_\tau \geqslant 0$, the first factor cannot have a pole, and any potential pole of the second factor is canceled by the zero of $\Gamma( \lambda_p )^{-1}$.  It follows that ${\bf c}_{\beta_p}(\lambda,\tau)^{-1}$, and hence ${\bf c}_p(\lambda,\tau)^{-1}$, is holomorphic, as required.

We now treat the case of $m'$ odd.  We first consider the expression $\prod_{q = 1}^j {\bf c}_{(\beta_p + \beta_q)/2}(\lambda, \tau)^{-1}$, which appears as a factor of ${\bf c}_p(\lambda,\tau)^{-1}$.  Up to a non-zero constant, this is equal to
\[
\prod_{q=1}^j \frac{ \Gamma( (m'+ \varrho_q +  \lambda_p)/2)}{ \Gamma((\varrho_q + \lambda_p)/2)} 
 = \Gamma((m' +\varrho_1 +\lambda_p)/2) \prod_{q=1}^{j-1} \frac{ \Gamma((m' + \varrho_{q+1} +  \lambda_p)/2)}{ \Gamma( (\varrho_q +\lambda_p)/2)} \frac{1}{ \Gamma( (\varrho_j +  \lambda_p)/2)},
\]
where we have reorganized the product by pairing the denominator of one term with the numerator of the next.  We observe that each term in the product from 1 to $j-1$ is holomorphic; this follows from the fact that the difference in the $\Gamma$-arguments, namely $(m' +  \varrho_{q+1} - \varrho_q)/2$, is a positive integer by the definition of $D_{\tau,j}$ and the condition that $m'$ is odd.  As the term $\Gamma( (\varrho_j +  \lambda_p)/2)^{-1}$ is also holomorphic, to prove the holomorphy of ${\bf c}_p(\lambda,\tau)^{-1}$ it suffices to establish it for
\[
{\bf c}_{\beta_p}(\lambda,\tau)^{-1} \Gamma((m' +\varrho_1 +\lambda_p)/2),
\]
which is equal (up to holomorphic factors) to
\[
\Gamma( \tfrac{1}{2} + \tfrac{1}{2} \lambda_p + \tfrac{1}{2} \ell_\tau ) \Gamma( \tfrac{1}{2} + \tfrac{1}{2} \lambda_p - \tfrac{1}{2} \ell_\tau ) \Gamma(\tfrac{1}{2} m' + \tfrac{1}{2} \varrho_1 + \tfrac{1}{2} \lambda_p ) \Gamma( \lambda_p )^{-1}.
\]
As in the case of $m'$ even, we only need to establish holomorphy at the point $\lambda_p = 0$, and if we assume that $\ell_\tau \geqslant 0$, the first factor does not have a pole because the argument is positive there.  The factor $\Gamma( \lambda_p )^{-1}$ has a zero, so it suffices to show that at most one of the factors $\Gamma( \tfrac{1}{2} + \tfrac{1}{2} \lambda_p - \tfrac{1}{2} \ell_\tau )$ and $\Gamma(\tfrac{1}{2} m' + \tfrac{1}{2} \varrho_1 + \tfrac{1}{2} \lambda_p )$ has a pole. If they both had poles, then both $\tfrac{1}{2} - \tfrac{1}{2} \ell_\tau$ and $\frac12 m' + \frac12 \varrho_1$ would be nonpositive integers.  This would imply that both $\ell_\tau$ and $\varrho_1$ were odd integers, but this is not possible in light of the definition of $D_{\tau,j}$. This establishes the holomorphy of ${\bf c}_p(\lambda,\tau)^{-1}$, and completes the proof.
\end{proof}

\subsection{Bounds for $k_\nu$}
\label{sec:test-fn-bds}

Let $k_\nu$ be the test function given by Proposition \ref{test-fn}.  In this section, we prove two bounds for $k_\nu$ that will later be used to control the geometric side of the pre-trace formula.  The first is a standard bound in the spherical case, which establishes decay of $k_\nu$ away from the base point.  We derive it from the bounds of \cite[Theorem 2]{BP} and \cite[ Theorem 8.2]{MT} for $\varphi_\mu$.

\begin{lemma}
\label{knu-decay}

Let $\nu\in i\mathfrak{a}^*$, and let $\tau=\tau_{\rm triv}$ be trivial. Let $0<R<R_0$ and $k_\nu\in C^\infty_R(S,\tau_{\rm triv},\tau_{\rm triv})$ be as in Proposition \ref{test-fn}. For $H \in \ga$, we have
\[
k_\nu(\exp H) \ll (1+\|\nu\| \|H\|)^{-1/2} \widetilde{\beta}_S(\nu).
\]

\end{lemma}

\begin{proof}

Let $\Omega \subset \ga$ be a compact set such that $\text{supp}(k_\nu) \subset K e^\Omega K$.  We may assume without loss of generality that $H \in \Omega$.  We express $k_\nu(\exp H)$ using the spherical Plancherel inversion \eqref{eq:sph-planch-inv}, which gives
\[
k_\nu(\exp H) = \frac{1}{|W|} \int_{ i \ga^*} \widehat{k}_\nu(\lambda)\varphi_{\lambda}( \exp H)|{\bf c}(\lambda)|^{-2}d\lambda = \int_{ i {\ga_+^*}} \widehat{k}_\nu(\lambda)\varphi_{\lambda}( \exp H)|{\bf c}(\lambda)|^{-2}d\lambda,
\]
where we have used the $W$-invariance of the integrand. We may insert the bound
\[
\varphi_{\lambda}(\exp H)\ll_\Omega (1+\|\lambda\| \|H\|)^{-1/2}
\]
for $H \in \Omega$ from \cite[Theorem 2]{BP} and \cite[ Theorem 8.2]{MT}, and the bound $|{\bf c}(\lambda)|^{-2} \ll \widetilde{\beta}_S(\lambda)$ from Lemma \ref{lemma:c-bd} \eqref{c-bd1}, to obtain
\[
k_\nu(\exp H) \ll \int_{ i {\ga_+^*}} \widehat{k}_\nu(\lambda)(1+\|\lambda\| \|H\|)^{-1/2} \widetilde{\beta}_S(\lambda) d\lambda.
\]
We then apply the bound for $\widehat{k}_\nu(\lambda)$ from Property \eqref{4} of Proposition \ref{test-fn} and unfold the sum over $W$ to obtain
\begin{align*}
k_\nu(\exp H) & \ll_{A,R} \int_{ i \ga^*}  (1 + \| \lambda - \nu \|)^{-A}(1+\|\lambda\| \|H\|)^{-1/2} \widetilde{\beta}_S(\lambda) d\lambda \\
& = \int_{ i \ga^*} (1 + \| \lambda \|)^{-A}(1+\|\lambda + \nu\| \|H\|)^{-1/2} \widetilde{\beta}_S(\lambda + \nu) d\lambda.
\end{align*}
Applying Lemma \ref{c-bound} and the bound
\[
(1+\|\lambda + \nu\| \|H\|)^{-1/2} \ll (1+\|\nu\| \|H\|)^{-1/2} (1 + \| \lambda \| )^{1/2},
\]
while taking $A$ sufficiently large to ensure the convergence of the integral, completes the proof.
\end{proof}

We next prove a bound for the $L^\infty$ norm of $D k_\nu$, where $D$ is a differential operator.  We shall do this by combining the following bound for $D \varphi_{\lambda, \tau}$ with the spherical inversion formula.  Note that $\tau$ is no longer assumed to be trivial.

\begin{lemma}
\label{sph-fn-bd}

Let $C > 0$, and let $D$ be a differential operator on $G$ of degree $n$ with smooth coefficients.  For any compact set $B \subset G$ and any $\lambda \in \ga^*_\C$ with $\| \textup{Im} \lambda \| \leqslant C$ we have $\| D \varphi_{\lambda, \tau}|_B \|_\infty \ll_{B,C,D,\tau} (1 + \| \lambda \|)^n$.

\end{lemma}

\begin{proof}

We recall the left action of $G$ on functions $f$ on $G$, given by $(g.f)(h) = f(g^{-1}h)$, and on differential operators $D$, given by $(g.D)(f) = g.(D( g^{-1}.f))$.  If we define $\psi(g) = a(g)^{\rho + \lambda} \tau( \kappa(g)^{-1})$, then the formula \eqref{defn-sph-fn} may be rewritten
\[
\varphi_{\lambda, \tau}(g) = \int_{K} \tau(k) \psi(kg) dk = \int_{K} \tau(k) (k^{-1}.\psi)(g) dk,
\]
which gives
\[
D \varphi_{\lambda, \tau} = \int_{K} \tau(k) D(k^{-1}.\psi) dk = \int_{K} \tau(k) k^{-1}.((k.D)\psi) dk.
\]
It therefore suffices to bound $\| (k.D)\psi |_B \|_\infty$ for $k \in K$.  As the coefficients of $D$ are smooth, the operators $k.D$ have coefficients which are bounded on $B$, uniformly in $k$.  By choosing a basis for the module of differential operators on $B$, it suffices to bound $D\psi$ for a single $D$.

As $\tau$ is fixed, we may write $\psi(g) = b(g) e^{ \lambda (H(g))}$, where $b(g)$ is a fixed smooth function.  The lemma now follows from an elementary argument by writing out $\psi$ and $D$ in a coordinate chart.
\end{proof}

\begin{lemma}\label{k-nu-e}
Let $\nu\in i\mathfrak{a}^*$ and $\tau\in\widehat{K}$ be one-dimensional. Let $0<R<R_0$ and $k_\nu\in C^\infty_R(S,\tau,\tau)$ be as in Proposition \ref{test-fn}. If $D$ is a differential operator of degree $n$ on $G$ with smooth coefficients, we have $\| D k_\nu \|_\infty \ll_{R, D,\tau} \tilde{\beta}_S(\nu) (1 + \| \nu \|)^n$.
\end{lemma}

\begin{proof}
From the $\tau$-spherical inversion formula \eqref{spherical-inversion} we have
\[
D k_\nu = \sum_{j=0}^r \sum_{\varrho\in D_{\tau,j}} \frac{d_j(\varrho,\tau)}{|W_{\Theta_{r-j}}|} \int_{i\mathfrak{a}^*_{\Theta_j}} D\varphi_{\varrho+\lambda_{\Theta_j}, \tau} \widehat{k}_\nu(\varrho+\lambda_{\Theta_j};\tau) |{\bf c}^{\Theta_j}(\varrho+\lambda_{\Theta_j},\tau)|^{-2} d\lambda_{\Theta_j}.
\]
Let $B$ be a compact set containing the support of $k_\nu$.  Applying Lemma \ref{sph-fn-bd} with $C$ chosen so that $C > \| \varrho \|$ for all $\varrho \in D_{\tau, j}$, we find that $\| D\varphi_{\varrho+\lambda_{\Theta_j}, \tau} |_B \|_\infty \ll_{D,\tau} (1 + \| \varrho+\lambda_{\Theta_j} \|)^n$ for all $\varrho\in D_{\tau,j}$ and $\lambda_{\Theta_j} \in i\mathfrak{a}^*_{\Theta_j}$.  This implies that
\[
\| D k_\nu \|_\infty \ll_{D,\tau} \sum_{j=0}^r \sum_{\varrho\in D_{\tau,j}} \frac{d_j(\varrho,\tau)}{|W_{\Theta_{r-j}}|}  \int_{i\mathfrak{a}^*_{\Theta_j}} (1 + \| \varrho+\lambda_{\Theta_j} \|)^n |\widehat{k}_\nu(\varrho+\lambda_{\Theta_j};\tau)| |{\bf c}^{\Theta_j}(\varrho+\lambda_{\Theta_j},\tau)|^{-2} d\lambda_{\Theta_j}.
\]
We shall show that the right-hand side is $\ll_{\tau} \tilde{\beta}_S(\nu) (1 + \| \nu \|)^n$. 

We begin with the two extremal cases: $j=0$ and $j=r$. For $j=0$ we use Lemma \ref{lemma:c-theta-bd}, as well as Proposition \ref{test-fn} \eqref{4}, to obtain
\begin{align*}
\frac{1}{|W|}\int_{i\mathfrak{a}^*} (1 + \| \lambda \|)^n \widehat{k}_\nu(\lambda;\tau)|{\bf c}(\lambda,\tau)|^{-2}d\lambda &\ll_{R, A,\tau} \sum_{w\in W}\int_{{i\mathfrak{a}_+^*}} (1 + \| \lambda \|)^n (1+\|w\lambda-\nu\|)^{-A} \tilde\beta_S(\lambda) d\lambda\\
 &  =\int_{i\mathfrak{a}^*} (1 + \| \lambda \|)^n (1+\|\lambda-\nu\|)^{-A} \tilde\beta_S(\lambda) d\lambda,
\end{align*}
where we have folded up the $W$-sum. After applying Lemma \ref{c-bound}, and the inequality
\begin{equation}\label{eq:lambda-bd}
(1+\|\lambda\|)\leqslant (1+\|\nu\|)(1+\|\lambda-\nu\|),
\end{equation}
this becomes
\[
\ll  \tilde\beta_S(\nu)(1+\|\nu\|)^n\int_{i\mathfrak{a}^*} (1 + \| \lambda - \nu \|)^{n+d^\emptyset(\Sigma)-A} d\lambda,
\]
which, for $A$ large enough, is $\ll_{\tau} \tilde{\beta}_S(\nu) (1 + \| \nu \|)^n$.  For $j = r$, the bound
\[
\sum_{\varrho\in D_{\tau}}d(\varrho,\tau) (1 + \| \varrho \|)^n \widehat{k}_\nu(\varrho;\tau)\ll_{\tau} 1
\]
follows from the fact that $D_\tau$ is finite.

For the intermediate cases $j=1,\ldots ,r-1$, we apply Proposition \ref{test-fn} \eqref{4} and Lemma \ref{lemma:c-theta-bd} to get
\begin{align}
\notag
&\sum_{\varrho\in D_{\tau,j}} \frac{d_j(\varrho,\tau)}{|W_{\Theta_{r-j}}|}\int_{i\mathfrak{a}^*_{\Theta_j}} (1 + \| \varrho+\lambda_{\Theta_j} \|)^n |\widehat{k}_\nu(\varrho+\lambda_{\Theta_j};\tau)| |{\bf c}^{\Theta_j}(\varrho+\lambda_{\Theta_j},\tau)|^{-2} d\lambda_{\Theta_j}\\
\label{j-invert-bound}
&\ll_{R,A,\tau} \max_{\varrho\in D_{\tau,j}}\max_{w\in W}\int_{i\mathfrak{a}^*_{\Theta_j}} (1 + \| \varrho+\lambda_{\Theta_j} \|)^n (1+\|w(\varrho+\lambda_{\Theta_j})-\nu\|)^{-A} \tilde\beta_S^{\Theta_j}(\varrho+\lambda_{\Theta_j})d\lambda_{\Theta_j}.
\end{align}
The inequality \eqref{eq:lambda-bd} gives
\[
1+\|\varrho+\lambda_{\Theta_j}\| \leqslant (1+\|\varrho\|)(1+\|\lambda_{\Theta_j}\|) \ll_\tau 1+\|\lambda_{\Theta_j}\|
\]
and
\[
(1+\|w(\varrho+\lambda_{\Theta_j})-\nu\|) \geqslant \frac{ 1 + \| w\lambda_{\Theta_j} - \nu \| }{ 1 + \| w \varrho \| } \gg_\tau 1 + \| w\lambda_{\Theta_j} - \nu \|,
\]
while Lemma \ref{c-bound} gives
\[
\tilde\beta_S^{\Theta_j}(\varrho+\lambda_{\Theta_j}) \ll (1 + \| \varrho \|)^{d^{\Theta_j}(\Sigma)} \tilde\beta_S^{\Theta_j}(\lambda_{\Theta_j}) \ll_\tau \tilde\beta_S^{\Theta_j}(\lambda_{\Theta_j}).
\]
Applying these to \eqref{j-invert-bound}, the upper bound becomes
\begin{align*}
&\ll_{R,A,\tau} \max_{\varrho\in D_{\tau,j}} \max_{w\in W} \int_{i\mathfrak{a}^*_{\Theta_j}} (1 + \| \lambda_{\Theta_j} \|)^n (1+\|w\lambda_{\Theta_j}-\nu\|)^{-A} \tilde\beta_S^{\Theta_j}(\lambda_{\Theta_j})d\lambda_{\Theta_j}
\\
&= \max_{w\in W} \int_{i\mathfrak{a}^*_{\Theta_j}} (1 + \| \lambda_{\Theta_j} \|)^n (1+\|\lambda_{\Theta_j}-w\nu\|)^{-A} \tilde\beta_S^{\Theta_j}(\lambda_{\Theta_j})d\lambda_{\Theta_j}.
\end{align*}
As in the case $j = 0$, we may bound this integral by $(1 + \| \nu \|)^n \tilde\beta_S^{\Theta_j}(w\nu ) \ll (1 + \| \nu \|)^n \tilde\beta_S(\nu )$, which completes the proof.
\end{proof}

\section{Upper bounds on the cuspidal spectrum with abelian lowest $K$-type}\label{sec:upper-bd-cusp-spec}

In this section we establish upper bounds, which are sharp up to logarithmic powers, on the (multidimensional) spectral counting function of $L^2_\text{cusp}(\Gamma \backslash G, \tau)$, where $\Gamma$ is a lattice in a semi-simple Lie group $G$ and $\tau$ is an abelian $K$-type. The precise conditions on $G,\Gamma,$ and $\tau$ will be spelled out in Section \ref{sec:upper-bd-notation} below. When $\tau\in\widehat{K}$ is non-trivial, this will include the primary case of interest for Theorem \ref{sup-thm}, with $G$ being the real points of the restriction of scalars of a symplectic group over a totally real number field $F$, so that $G=\Sp_{2m}(\R)^{[F:\Q]}$. The main result is Proposition \ref{Weyl-upper-bd}, which corresponds to Step \ref{3sketch} of the introduction, and which, we emphasize, only requires upper bounds (existence will not be our concern).

\subsection{Notation and statement of result}\label{sec:upper-bd-notation}

Let $\bG$ be a connected semisimple $\Q$-group and $G=\bG(\R)$ its real points. Let $r$ be the rank of $G$. We denote the connected component of the identity in $G$ by $G^0$. Let $K$ be a maximal compact subgroup of $G$, and let $K^0 = K \cap G^0$ which is a maximal compact subgroup of $G^0$. Let $S = G^0 / K^0$ be the Riemannian globally symmetric space associated to $G^0$. Let $\tau$ be a character of $K^0$. 

We make the following hypothesis on the pair $(G,\tau)$. Let $G_j$ denote the simple factors of $G^0$, with maximal compacts $K_j$. Write $\tau_j$ for the restriction of $\tau$ to $K_j$. We extend the assumption \eqref{eq:assumption} to the case of semisimple $G$, by assuming
\begin{equation}\label{eq:assumption2}
\textit{when $\tau_j$ is non-trivial, the Hermitian group $G_j$ has reduced root system.}
\end{equation}

Let $\Gamma < \bG(\Q)$ be a congruence arithmetic lattice. We shall work with various spaces of automorphic forms on $\Gamma \backslash G$ that are isotypic for $\tau$.  It would be natural to use notation such as $\cA(\Gamma \backslash G, \tau)$ for these spaces, but unfortunately this conflicts with the notation used in Sections \ref{sec:notation}--\ref{sec:test-function}, where e.g. $C^\infty(S, \tau)$ denotes functions that are $\tau^{-1}$-isotypic under the right action of $K^0$.  As a result, $\cA(\Gamma \backslash G, \tau)$, $L^2_\text{cusp}(\Gamma \backslash G, \tau)$, etc. {\bf will always denote functions that are $\tau^{-1}$-isotypic.}  This is somewhat awkward, but we feel that it is the best compromise available.

We therefore define $L^2_\text{cusp}(\Gamma \backslash G, \tau)$ to be the $\tau^{-1}$-isotypic subspace of $L^2_\text{cusp}(\Gamma \backslash G)$.  If $g_i$ are representatives for $\Gamma \backslash G / G^0$, and $\Gamma_i = g_i^{-1} ( \Gamma \cap G^0 ) g_i$, then $L^2_\text{cusp}(\Gamma \backslash G, \tau) \simeq \oplus L^2_\text{cusp}(\Gamma_i \backslash S, \tau)$.

The algebra $\mathbb{D}(\tau)$ of left $G$-invariant differential operators on $H^0(S,L(\tau))$, defined in Section \ref{sec:tau-spherical}, preserves $L^2_{\rm cusp}(\Gamma\backslash G,\tau)\cap C^\infty(\Gamma \backslash G, \tau)$. For $\lambda\in\mathfrak{a}^*_\C$ we let
\begin{equation}\label{eq:E-lambda-cusp-tau}
\mathcal{E}_\lambda^{\rm cusp}(\Gamma\backslash G,\tau)=\{f\in L^2_{\rm cusp}(\Gamma \backslash G,\tau) \cap C^\infty(\Gamma \backslash G, \tau): D f=\chi_\lambda(D) f, \;\forall D\in\mathbb{D}(\tau)\}.
\end{equation}
Let $\Lambda_{\rm cusp}(\Gamma;\tau)=\{\lambda\in\mathfrak{a}^*_\C/W: \dim \mathcal{E}_\lambda^{\rm cusp}(\Gamma\backslash G,\tau) > 0\}$ be the cuspidal spectrum. Then
\begin{equation}\label{eq:L2-cusp-tau}
L^2_{\rm cusp}(\Gamma \backslash G,\tau)=\bigoplus_{\lambda\in \Lambda_{\rm cusp}(\Gamma;\tau)} \mathcal{E}_\lambda^{\rm cusp}(\Gamma \backslash G,\tau).
\end{equation}

We may now state the main result of this section, which estimates the number of cuspidal $\tau$-spherical spectral parameters $\lambda\in\Lambda_{\rm cusp}(\Gamma;\tau)$ in a bounded window about $\nu$.

\begin{prop}\label{Weyl-upper-bd}
Let $G$ and $\tau$ be as above. For any $\nu\in i\mathfrak{a}^*$ we have
\[
\sum_{\substack{\lambda\in \Lambda_{\rm cusp}(\Gamma;\tau)\\ \|{\rm Im}\,\lambda-\nu\|= O(1)}}\dim \mathcal{E}_\lambda^{\rm cusp}(\Gamma \backslash G,\tau)\ll_{\Gamma,\tau} \left(\log (3+\|\nu\|)\right)^r \tilde\beta_S(\nu).
\]
\end{prop}

The logarithmic factor in Proposition \ref{Weyl-upper-bd} is responsible for the loss of the same quality in Theorem \ref{sup-thm}. As was mentioned in Remark \ref{remark2-intro}, this factor is the price to pay for our use of a partial trace formula; it arises essentially from the estimation of the unipotent conjugacy classes.

Furthermore, the arithmeticity assumption on the lattice $\Gamma$ is used only marginally, to ensure a pleasant theory of lattice reduction and Siegel domains, which is of critical use in the proof.

\subsection{Discussion of literature}

Before we provide a preview of the proof of Proposition \ref{Weyl-upper-bd}, let us place the result in perspective, by recalling the extensive literature surrounding this topic.

We begin by comparing Proposition \ref{Weyl-upper-bd} with known \textit{upper bounds} in other settings. Duistermaat, Kolk, and Varadarajan \cite[Theorem 7.3]{DKV} prove a bound of the above form in the case when $\Gamma$ is a \textit{uniform lattice} (where the notion of cuspidality is empty) in a non-compact connected semisimple real Lie group $G$ with finite center and $\tau$ is the \textit{trivial $K$-type}. On the other hand, for non-compact semisimple $G$ and arbitrary lattices $\Gamma< G$, Donnelly \cite[Theorem 9.1]{Donnelly82} showed that for {\it any} $K$-type $\tau$ one has
\[
\limsup_{\Lambda\rightarrow\infty}\bigg(\Lambda^{-\dim S/2}\sum_{\substack{\lambda\in \Lambda_{\rm cusp}(\Gamma;\tau)\\ \Omega(\lambda)\leqslant\Lambda}}\dim \mathcal{E}_\lambda^{\rm cusp}(\Gamma \backslash G,\tau)\bigg)\leqslant (4\pi)^{-\dim S/2}\frac{{\rm vol}(\Gamma \backslash G)}{\Gamma(\dim S/2+1)}\dim(\tau),
\]
where $\Omega(\lambda)$ is the eigenvalue of the Casimir operator acting on $\mathcal{E}_\lambda(S,\tau)$.  (When the dimension of $\tau$ is greater than 1, $\mathcal{E}_\lambda(S,\tau)$ denotes forms that are isotypic for the dual representation $\tau^\vee$.)  In one respect, Donnelly's result is unnecessarily fine for us, since we do not need the exact constant in the upper bound. In another respect, his result is not fine enough for us, as it does not take into account the singularities of the spectral parameter: it counts all spectral parameters in a ball of increasing radius.

We now compare Proposition \ref{Weyl-upper-bd} with known asymptotic statements (Weyl laws, from which upper bounds can be deduced) in other settings. There are two Weyl laws for the cuspidal automorphic spectrum for split semisimple $\Q$-groups, such as ${\rm Res}_{F/\Q}\Sp_{2m}$, under a spherical assumption at infinity, due to Lindenstrauss and Venkatesh \cite{LV} and, more recently, to Finis and Lapid \cite[Theorem 5.11]{FL2}. Besides the spherical assumption at infinity (which is the same as to say that $\tau$ is trivial), the former paper (like that of Donnelly) counts only by Laplace eigenvalue, while the latter counts in dilated regions in $i\ga^*$ (albeit with a power saving error term).

\subsection{Sketch of proof}
Taking inspiration from Steven D. Miller's thesis \cite{Mi}, we prove Proposition \ref{Weyl-upper-bd} by integrating the automorphic kernel associated with an approximate spectral projector over a sufficiently large truncated Siegel domain $(\Gamma\backslash G)^{\leqslant T}$, rather than the entire automorphic quotient $\Gamma\backslash G$. The resulting formula, which one may call a \textit{partial trace formula}, has the advantage of allowing for an elementary treatment of Eisenstein series: their $L^2$-mass on $(\Gamma\backslash G)^{\leqslant T}$ can be discarded by positivity, regardless of the size of the truncation parameter $T$, since we only need upper bounds.

To ensure that the cuspidal spectrum receives a weight which is bounded away from zero, we must take the parameter $T$ large enough to capture a positive proportion of the $L^2$ mass of a cusp form of spectral parameter $\nu$. This is plausible in light of the rapid decay of cusp forms, a classical fact in the theory of automorphic forms, due to Gelfand and Piatetski-Shapiro. Unfortunately, the standard proof of rapid decay \cite[Lemma 1.2.10]{MW}, which couples a uniform moderate growth estimate with integration by parts, does not come with any explicit uniformity in the spectral parameter. 

We resolve this problem in Lemma \ref{lemma:mod-growth}, where we establish a uniform moderate growth estimate for $\tau^{-1}$-isotypic cusp forms of spectral parameter $\nu$, which is uniform in $\nu$. Such quantitative control is furnished by the sup norm estimate on the $\tau^{-1}$-spherical spectral projector (and its derivatives) in Proposition \ref{k-nu-e}.  (Note that we must apply Proposition \ref{test-fn} to produce a function in $C^\infty(S, \tau^{-1}, \tau^{-1})$ if we want to count $\tau^{-1}$-isotypic forms.)  We insert this information into the standard argument for rapid decay in Lemma \ref{lemma:rapid-decay}, which is then itself uniform in $\nu$. This last result is then shown to imply, in Lemmas \ref{lemma:phi-pointwise} and \ref{L2-mass}, the requisite lower bound on the $L^2$ mass of cusp forms on truncated Siegel domains.

Finally, the logarithmic loss in $\nu$ will come from the contribution of the unipotent elements in $\Gamma$ on the geometric side of the partial trace formula; the relevant estimate is essentially contained in Lemma \ref{cusp-kernel}.

\subsection{Siegel sets}\label{sec:Siegel-sets}

We review some notation related to Siegel domains. For a nice reduction theory, we shall operate under the assumption that $\Gamma$ is congruence arithmetic. The arithmetic assumption on $\Gamma$ can be expressed as follows. There is a faithful $\Q$-representation $\rho:\bG\rightarrow\SL_N$, such that $\rho(\Gamma)$ contains the intersection of $\rho(\bG(\Q))$ with a principal congruence subgroup of $\SL_N(\Z)$. Our primary source for reduction theory in this context is the book of Borel \cite{Borel69}.

We now introduce a few important subgroups of $\bG$ related to the geometry of $\Gamma\backslash G$ at infinity. We shall denote groups over $\Q$ by a boldface letter, and real points by the corresponding roman letter. If $\bH$ is a $\Q$-group, we write $H^0$ for the connected component of the identity in $H$.


Let\footnote{We caution the reader that the group of real points $P_\Q$ of $\bP_\Q$ is not necessarily a minimal $\R$-parabolic of $G$, the latter having been denoted simply by $P$ in Section \ref{sec:gp-decomp}. We include the $\Q$ subscript in $P_\Q$ to distinguish it from $P$ for this reason, but we refrain from including the same subscript on related subgroups in this section, to avoid notational overload.} $\bP_\Q= \bU\bL$ be a minimal $\Q$-parabolic subgroup of $\bG$. Let $\bA$ be a maximal $\Q$-split torus in the center of $\bL$, so that $\bL$ is the centralizer of $\bA$ in $\bG$. Let
\[
\bM= \bigg( \bigcap_{ \chi \in X^*(\bL)_\Q} \ker \chi \bigg)^0
\]
be the maximal connected $\Q$-anisotropic subgroup of $\bL$; then $\bP_\Q =\bU\bM\bA$. Let $\ga_{P_\Q}$ be the Lie algebra of $A$. We assume that $\mathbf{P}_\Q$ and $\mathbf{A}$ are chosen such that $\mathfrak{a}_{P_\Q}$ is orthogonal to $\mathfrak{k}$. Let $\Sigma_\Q$ denote the roots of $\ga_{P_\Q}$ in $\g$, $\Sigma_\Q^+$ the positive roots corresponding to $\bU$, and $\Psi_\Q$ the simple roots. We write the half-sum of the positive roots in this setting as $\rho_\Q$. 

We have the Langlands decompositions $P_\Q = U MA$ and $P_\Q^0 = U M^0 A^0$, the latter of which is a diffeomorphism. Since $K$ meets all connected components of $G$ we have $G = P_\Q^0 K$, and hence $G=UM^0A^0K$. We define the Iwasawa height function $a: G \to A^0$ by $x\in UM^0 a(x) K$. Similarly to \eqref{Iwasawa-measure1} (where the inclusion $M\subset K$ held), we may decompose a fixed Haar measure $dg$ on $G$ as
\begin{equation}\label{Q-Iwasawa-measure1}
dg=a^{-2\rho_\Q} du\, dm \, da\, dk,
\end{equation}
for appropriately normalized Haar measures on the unimodular groups $U, M^0, A^0, K$.

For $t>0$ write
\[
A_t^0=\{a\in A^0: \alpha(\log a)\geqslant t \;\forall\;\alpha\in \Psi_\Q \}.
\]
Let $\omega\subset U.M^0$ be a compact neighborhood of the identity. We call a \textit{Siegel set} (with respect to $\mathbf{P}_\Q$ and $K$) any subset of $G$ of the form $\mathfrak{S}_{t,\omega}=\omega A_t^0 K$. In view of the condition $\mathfrak{a}_{P_\Q}\perp\mathfrak{k}$, such Siegel domains are called \textit{normal} in \cite[Definition 12.3]{Borel69}.  Borel \cite[Th\'eor\`emes 13.1, 15.4]{Borel69} has shown the following properties of Siegel sets.

\begin{enumerate}

\item For every finite subset $A \subset \bG(\Q)$ the set $\{\gamma\in\Gamma: \gamma A\mathfrak{S}_{t,\omega}\cap A\mathfrak{S}_{t,\omega} \neq \emptyset \}$ is finite (the Siegel property).

\item There is a Siegel set $\mathfrak{S}_{t,\omega}$ and a finite set $\Xi\subset\mathbf{G}(\Q)$ such that $G=\Gamma\Xi\mathfrak{S}_{t,\omega}$.

\end{enumerate}

In this case the subset $\Xi\mathfrak{S}_{t,\omega}$ of $G$ is called a \textit{fundamental set}, and the above properties state that the natural map $\Xi\mathfrak{S}_{t,\omega}\rightarrow \Gamma \backslash G$ is surjective with finite fibers. By \cite[Proposition 15.6]{Borel69}, one may take for $\Xi$ a complete set of representatives for the ``set of cusps" $\Gamma\backslash\mathbf{G}(\Q)/\bP_\Q(\Q)$. 

We now suppose that $t$, $\omega$, and $\Xi$ are chosen so that $\Xi\mathfrak{S}_{t,\omega}$ is a fundamental set for $\Gamma$.  For $T > t$, define
\[
A_t^{\leqslant T} = \{a\in A^0 : t\leqslant\alpha(\log a) \leqslant T \;\forall\;\alpha\in \Psi_\Q \} \quad \text{and} \quad  A_t^{> T} = A^0_t \setminus A_t^{\leqslant T}
\]
to be the truncation of $A_t^0$ at height $T$, and its complement.  We also define $\mathfrak{S}_{t,\omega}^{\leqslant T} = \omega A_t^{\leqslant T}K$ and $\mathfrak{S}_{t,\omega}^{> T} = \omega A_t^{> T} K$.  We let $(\Gamma \backslash G)^{\leqslant T}$ be the image of $\Xi \mathfrak{S}_{t,\omega}^{\leqslant T}$ in $\Gamma \backslash G$, and let $(\Gamma \backslash G)^{>T}$ be the complement.

\subsection{Growth of automorphic kernel functions in the cusp}
The following lemma will be used to bound the growth of automorphic kernel functions in the cusp in the proof of Lemma \ref{lemma:mod-growth} and again in Section \ref{sec:proof-Prop-Weyl} (to complete the proof of Proposition \ref{Weyl-upper-bd}). It is valid in the generality of Section \ref{sec:Siegel-sets}, and hence independent of the hypothesis \eqref{eq:assumption2}.  In this section, we let $\Gamma_U$ denote the intersection $\Gamma \cap U$.

\begin{lemma}
\label{cusp-kernel}
Let $B \subset G$ be compact, and let $\mathfrak{S}_{t,\omega}$ be a Siegel set.  If $x \in \mathfrak{S}_{t,\omega}$ and $\xi \in \bG(\Q)$, we have
\[
\sum_{\gamma \in \Gamma} 1_B( x^{-1} \xi^{-1} \gamma \xi x) \ll a(x)^{2 \rho_\Q}.
\]
\end{lemma}

\begin{proof}

By replacing $\Gamma$ with $\xi^{-1} \Gamma \xi$, we may assume that $\xi = 1$.  By applying Lemma \ref{Siegel-stab}, it suffices to restrict the sum over $\Gamma$ to $\Gamma_U \delta$, for some $\delta \in \Gamma$ that is independent of $x$.  In other words, we must show that $\# \{ \gamma_U \in \Gamma_U : x^{-1} \gamma_U \delta x \in B \} \ll a(x)^{2 \rho_\Q}$.  We may assume that $x^{-1} \delta x \in B$, at the expense of allowing the coset representative $\delta$ (but not the coset $\Gamma_U \delta$) to depend on $x$.

Suppose that $\gamma_U$ satisfies $x^{-1} \gamma_U \delta x \in B$.  We may write this condition as $\gamma_U^{x^{-1}} x^{-1} \delta x \in B$, and combined with our assumption that $x^{-1} \delta x \in B$ this gives $\gamma_U^{x^{-1}} \in B B^{-1}$.  If we write $x = u m a k$ with $um \in \omega$ and $a \in A^0_t$, and assume that $B$ is bi-$K$-invariant, we may expand the condition $\gamma_U^{x^{-1}} \in B B^{-1}$ as
\[
a^{-1} m^{-1} u^{-1} \gamma_U u m a \in B B^{-1}.
\]
We next rewrite this as
\begin{align*}
(m^{-1} u^{-1})^{a^{-1}} a^{-1} \gamma_U a (u m)^{a^{-1}} & \in B B^{-1} \\
a^{-1} \gamma_U a  & \in (u m)^{a^{-1}} B B^{-1} (m^{-1} u^{-1})^{a^{-1}} \\
& \subset \omega^{a^{-1}} B B^{-1} (\omega^{-1})^{a^{-1}}.
\end{align*}
The union of the sets $\omega^{a^{-1}}$, taken over $a \in A^0_t$, is compact, and we denote it by $\Omega$.  We deduce that $a^{-1} \gamma_U a$ lies in the fixed compact set $\Omega B B^{-1} \Omega^{-1}$, and hence that the set $\{ \gamma_U \in \Gamma_U : x^{-1} \gamma_U \delta x \in B \}$ whose size we wish to control is contained in $\Gamma_U \cap a \Omega B B^{-1} \Omega^{-1}  a^{-1}$.  The lemma now follows from Lemma \ref{Uball} below.

\end{proof}

\begin{lemma}
\label{Siegel-stab}

Let $B \subset G$ be compact, and let $\mathfrak{S}_{t,\omega}$ be a Siegel set.  There exists a finite set $\delta_1, \ldots, \delta_k \in \Gamma$, depending on $B$ and $\mathfrak{S}_{t,\omega}$, such that if $x^{-1} \gamma x \in B$ for some $x \in \mathfrak{S}_{t,\omega}$, then $\gamma \in \Gamma_U \delta_i$ for some $i$.

\end{lemma}

\begin{proof}

Let $x \in \mathfrak{S}_{t,\omega}$ and $\gamma \in \Gamma$ satisfy $x^{-1} \gamma x \in B$.  We will show that there is a second Siegel set $\mathfrak{S}_{t',\omega'}$, independent of $x$, such that $\gamma_U \gamma x \in \mathfrak{S}_{t',\omega'}$ for some $\gamma_U \in \Gamma_U$. The lemma will follow from this by applying the Siegel property to $\mathfrak{S}_{t,\omega}$ and $\mathfrak{S}_{t',\omega'}$.

We write $x \in \mathfrak{S}_{t,\omega}$ as $x = u_1 m_1 a_1 k_1$ with $u_1 m_1 \in \omega$ and $a_1 \in A^0_t$, and likewise write $\gamma x$ using the decomposition $G=UM^0A^0K$ as $\gamma x = u_2 m_2 a_2 k_2$.  With this notation, we have
\begin{align*}
x^{-1} \gamma x & = k_1^{-1} a_1^{-1} m_1^{-1} u_1^{-1} u_2 m_2 a_2 k_2 \\
& =  k_1^{-1} (u_1^{-1} u_2)^{ a_1^{-1} m_1^{-1}} a_1^{-1} m_1^{-1}  m_2 a_2 k_2.
\end{align*}
If we write $u' = (u_1^{-1} u_2)^{ a_1^{-1} m_1^{-1}}$, this becomes
\[
x^{-1} \gamma x = k_1^{-1} u' m_1^{-1}  m_2 a_1^{-1} a_2 k_2.
\]
If we assume that $B$ is bi-$K$-invariant, the condition $x^{-1} \gamma x \in B$ may be written as $u' m_1^{-1}  m_2 a_1^{-1} a_2 \in B$, and we in fact have the stronger inclusion
\begin{equation}
\label{inBP}
u' m_1^{-1}  m_2 a_1^{-1} a_2 \in B \cap P_\Q^0.
\end{equation}
Because the Langlands decomposition $P_\Q^0 = U M^0 A^0$ is a diffeomorphism, we have $B \cap P_\Q^0 \subset B_U B_M B_A$ for compact sets $B_U \subset U$, $B_M \subset M^0$, and $B_A \subset A^0$, and it follows from this and (\ref{inBP}) that $m_2 \in m_1 B_M$ and $a_2 \in a_1 B_A$.

Let $\omega_U \subset U$ and $\omega_M \subset M^0$ be compact sets such that $\omega \subset \omega_U \omega_M$, which implies that $u_1 \in \omega_U$ and $m_1 \in \omega_M$.  Let $\omega'_U$ be a compact fundamental domain for the action of $\Gamma_U$ on $U$, and let $\gamma_U \in \Gamma_U$ satisfy $\gamma_U u_2 \in \omega_U'$.  The inclusion $m_2 \in m_1 B_M$ implies that $m_2 \in \omega_M B_M$, and $a_2 \in a_1 B_A$ implies that $a_2 \in A^0_{t'}$ for some $t' > 0$.  If we define $\omega' = \omega_U' \omega_M B_M$, then the auxiliary Siegel set we require is $\mathfrak{S}_{t',\omega'}$.  Indeed, we have
\begin{equation}
\label{Siegel-rep}
\gamma_U \gamma x = (\gamma_U u_2) m_2 a_2 k_2 \in \omega_U' (\omega_M B_M) A^0_{t'} K = \mathfrak{S}_{t',\omega'}
\end{equation}
as claimed.

By applying the Siegel property to a Siegel set containing both $\mathfrak{S}_{t,\omega}$ and $\mathfrak{S}_{t',\omega'}$, we see that the set of $\delta \in \Gamma$ such that $\delta \mathfrak{S}_{t,\omega} \cap \mathfrak{S}_{t',\omega'} \neq \emptyset$ is finite, and we enumerate them as $\delta_1, \ldots, \delta_k$.  Equation (\ref{Siegel-rep}) implies that $\gamma_U \gamma$ is equal to one of these $\delta_i$, which gives the lemma.

\end{proof}

\begin{lemma}
\label{Uball}

If $B \subset U$ is compact, and $a \in A^0_t$, then $\# (\Gamma_U \cap a B a^{-1}) \ll a^{2\rho_\Q}$.

\end{lemma}

\begin{proof}

This follows from a standard volume argument.  Take $C \subset U$ open and bounded, and such that $C C^{-1} \cap \Gamma_U = \{ e \}$.  If we let $\Lambda = \Gamma_U \cap a B a^{-1}$, we have
\be
\label{Lambda-include}
\Lambda C \subset a B a^{-1} C = a( B a^{-1} C a ) a^{-1}.
\ee
As in Lemma \ref{cusp-kernel}, the union of the sets $a^{-1} C a$ over all $a \in A^0_t$ is compact, which implies there is a compact set $B'$ containing $B a^{-1} C a$ for all $a$.  Equation (\ref{Lambda-include}) then gives $\Lambda C \subset a B' a^{-1}$.  Our assumption that $C C^{-1} \cap \Gamma_U = \{ e \}$ implies that the sets $\gamma C$ for $\gamma \in \Gamma$ are all disjoint, so
\[
(\# \Lambda) \, \text{vol}(C) = \text{vol}( \Lambda C) \leqslant \text{vol}( a B' a^{-1}) = a^{2\rho_\Q} \text{vol}( B'),
\]
as required.
\end{proof}

\subsection{Pre-trace inequality}

We let $R$ denote the right-regular representation of $G$ on $L^2(\Gamma\backslash G)$.  For a function $\omega \in C(G)$, we define $\omega^*$ by $\omega^*(g) = \overline{\omega}(g^{-1})$.

The following result, which is again valid in the generality of Section \ref{sec:Siegel-sets}, will be used in the proof of the uniform moderate growth estimate of Lemma \ref{lemma:mod-growth}. The idea there will be to bound $D\phi$, for $\phi \in \mathcal{E}_\mu^{\rm cusp}(\Gamma\backslash G,\tau)$ and $D$ a left-invariant differential operator on $G$, by writing it as $R(\omega) \phi$ for some well-chosen $\omega \in C_c(G)$ depending on $D$, $\tau$, and $\mu$.  We may then take advantage of the spectral theory of $L^2(\Gamma\backslash G,\tau)$ by applying the following lemma. This will be the key to accessing the dependency on $\mu$, since, by Proposition \ref{k-nu-e}, the analytic properties of $\omega * \omega^*$ can be understood for $\tau$-spherical eigenfunctions.

\begin{lemma}
\label{pre-trace}

If $\omega \in C_c(G)$ and $\{ \phi_i \}$ is an orthonormal set in $L^2(\Gamma \backslash G)$, then
\[
\sum_i | R(\omega) \phi_i(x) |^2 \leqslant \sum_{\gamma \in \Gamma} (\omega * \omega^*)(x^{-1} \gamma x).
\]
\end{lemma}

\begin{proof}
We have
\[
[R( \omega) \phi_i](x) = \int_G \phi_i(xg) \omega(g) dg = \int_{G} \phi_i(g) \omega(x^{-1} g) dg.
\]
Folding the integral over $\Gamma$ gives
\bes
[R( \omega) \phi_i](x) = \int_{ \Gamma \backslash G} \phi_i(g) \sum_{\gamma \in \Gamma} \omega(x^{-1} \gamma g) dg = \langle \phi_i, \sum_{\gamma \in \Gamma} \omega(x^{-1} \gamma \; \cdot) \rangle,
\ees
where $\langle \cdot , \cdot \rangle$ denotes the inner product in $L^2(\Gamma \backslash G)$.  Because the $\phi_i$ are orthonormal, we may apply Parseval to obtain
\[
\sum_i | R(\omega) \phi_i(x) |^2 \leqslant \bigg\| \sum_{\gamma \in \Gamma} \omega(x^{-1} \gamma \; \cdot) \bigg\|_2 = \int_{ \Gamma \backslash G} \bigg| \sum_{\gamma \in \Gamma} \omega(x^{-1} \gamma g) \bigg|^2 dg.
\]
Expanding the square and unfolding again, we obtain
\begin{equation}\label{amp1}
\sum_i | R(\omega) \phi_i(x) |^2  \leqslant \int_{ \Gamma \backslash G}  \sum_{\gamma_1, \gamma_2 \in \Gamma} \omega(x^{-1} \gamma_1 g) \overline{\omega(x^{-1} \gamma_2 g)} dg 
 = \sum_{\gamma \in \Gamma} \int_{G} \omega(x^{-1} \gamma g) \overline{\omega(x^{-1} g)} dg.
\end{equation}
Because $\overline{\omega(x^{-1} g)} = \omega^*(g^{-1} x)$, we have
\begin{align*}
\int_{G} \omega(x^{-1} \gamma g) \overline{\omega(x^{-1} g)} dg  = \int_{G} \omega(x^{-1} \gamma g) \omega^*(g^{-1} x)  dg  = \omega * \omega^* ( x^{-1} \gamma x).
\end{align*}
Inserting this into \eqref{amp1} completes the proof.
\end{proof}

\subsection{Uniform moderate growth of cusp forms}\label{sec:UMG}

We next prove a uniform moderate growth estimate for $\tau^{-1}$-isotypic eigenfunctions. For the statement and proof, we introduce some notation relating to the differential operators on $G$.

Let $D(G)$ and $D(G)^\vee$ be the algebras of left-invariant and right-invariant differential operators on $G$.  We define the involution $\vee : x \mapsto x^{-1}$ of $G$, which also acts on functions by $f^\vee(g) = f(g^{-1})$.  If $D$ is a differential operator on $G$ we let $D^\vee$ be its image under $\vee$, defined by $D^\vee(f) = (D f^\vee)^\vee$.  Then the map $\vee : D \mapsto D^\vee$ gives an algebra isomorphism $\vee : D(G) \to D(G)^\vee$.

We may explicate the map $\vee$ on left-invariant vector fields as follows: if $X \in \g$, and $\widetilde{X} \in D(G)$ is the left-invariant vector field associated to $X$, then $- \widetilde{X}^\vee$ is the corresponding right-invariant vector field.  We also have
\begin{equation}
\label{convolve}
(D_1^\vee f_1) * (D_2 f_2) = D_1^\vee D_2 (f_1 * f_2)
\end{equation}
for $f_1, f_2 \in C^\infty_c(G)$ and $D_1, D_2 \in D(G)$.  If $D$ is a differential operator on $G$, we let $\overline{D}$ denote the complex conjugate operator, defined by $\overline{D} f = \overline{ D \overline{f} }$.

\begin{lemma}\label{lemma:mod-growth}
Let $\mathfrak{S}_{t,\omega}$ be a Siegel set.  Let $\phi \in \mathcal{E}_\mu^{\rm cusp}(\Gamma\backslash G;\tau)$ satisfy $\| \phi \|_2 = 1$.  For any $D \in D(G)$ of degree $n$, and any $\xi \in \bG(\Q)$, we have
\begin{equation}\label{eq:mod-growth}
D \phi (\xi x)\ll_D (1+\|\mu\|)^{n} \widetilde{\beta}_S(\mu)^{1/2} a(x)^{ \rho_\Q}, \qquad x \in\mathfrak{S}_{t,\omega}.
\end{equation}

\end{lemma}

\begin{proof}

We shall apply Lemma \ref{pre-trace} to the orthonormal set containing the single element $\phi$. To describe our choice of $\omega$, we first let $k_{-\mu}^0 \in C^\infty_c(S, \tau^{-1}, \tau^{-1})$ be the function appearing in the proof of Proposition \ref{test-fn}, with $\nu$ there chosen to be $-\mu$.  By scaling $k_{-\mu}^0$ by a bounded constant, we may assume that $R(k_{-\mu}^0) \phi = \phi$.  This is equivalent to $\phi = \phi * (k_{-\mu}^0)^\vee$, and so
\[
D \phi = D( \phi * (k_{-\mu}^0)^\vee ) = \phi * D (k_{-\mu}^0)^\vee = R( (D (k_{-\mu}^0)^\vee) ^\vee ) \phi = R( D^\vee k_{-\mu}^0 ) \phi.
\]
It follows that if we apply Lemma \ref{pre-trace} with $\omega = D^\vee k^0_{-\mu}$, we obtain
\[
| D\phi(\xi x) |^2 \leqslant \sum_{\gamma \in \Gamma} [ (D^\vee k_{-\mu}^0) * (D^\vee k_{-\mu}^0)^* ](x^{-1} \xi^{-1} \gamma \xi x).
\]
Because the support of our test function is uniformly bounded, Lemma \ref{cusp-kernel} gives
\begin{equation}
\label{Xphi-intermed}
| D\phi(\xi x) |^2 \ll \| (D^\vee k_{-\mu}^0) * (D^\vee k_{-\mu}^0)^* \|_\infty\, a(x)^{2 \rho_\Q}.
\end{equation}
It therefore remains to show that
\be
\label{Dkbd}
\| (D^\vee k_{-\mu}^0) * (D^\vee k_{-\mu}^0)^* \|_\infty \ll_D (1+\|\mu\|)^{n} \widetilde{\beta}_S(\mu)^{1/2}.
\ee
We have
\begin{align*}
(D^\vee k_{-\mu}^0) * (D^\vee k_{-\mu}^0)^* = (D^\vee k_{-\mu}^0) * (\overline{D} (k_{-\mu}^0)^* ).
\end{align*}
Applying \eqref{convolve} to this gives
\begin{align*}
(D^\vee k_{-\mu}^0) * (D^\vee k_{-\mu}^0)^* & = D^\vee \overline{D} ( k_{-\mu}^0 * (k_{-\mu}^0)^* ) \\
& = D^\vee \overline{D} k_v.
\end{align*}
Lemma \ref{k-nu-e} gives $\| D^\vee \overline{D} k_v \|_\infty \ll_D (1 + \| \mu \|)^{2n} \widetilde{\beta}_S(\mu)$.  This implies (\ref{Dkbd}), and completes the proof.
\end{proof}

\subsection{Rapid decay estimates and control of $L^2$ mass}

For the rest of Section \ref{sec:upper-bd-cusp-spec}, we fix a finite set $\Xi \subset \bG(\Q)$ and a Siegel set $\mathfrak{S}_{t,\omega}$ such that $\Xi \mathfrak{S}_{t,\omega}$ is a fundamental set for $\Gamma$, and recall the notation for truncations of $\Gamma \backslash G$ introduced at the end of Section \ref{sec:Siegel-sets}.  We show that a cusp form with spectral parameter $\mu$ has a positive proportion of its $L^2$ mass in $(\Gamma \backslash G)^{\leqslant T}$, provided $T \gg (1 + \| \mu \|)^{1 + \epsilon}$.  To do this, we first prove a version of the standard rapid decay estimate for cusp forms which is uniform in the spectral parameter.  The proof of the following lemma is adapted from Lemma I.2.10 of \cite{MW}.

\begin{lemma}
\label{lemma:rapid-decay}
Let $\phi \in \mathcal{E}_\mu^{\rm cusp}(\Gamma\backslash G;\tau)$ satisfy $\| \phi \|_2 = 1$. Let $\alpha \in \Psi_\Q$ be a simple root. For any $N \geqslant 2$ and $\xi \in \Xi$, we have
\[
\phi( \xi x) \ll_N \widetilde{\beta}_S(\mu)^{1/2} ( 1 + \| \mu \|)^N a(x)^{ \rho_\Q - N \alpha}, \qquad x \in\mathfrak{S}_{t,\omega}.
\]
\end{lemma}

\begin{proof}
As there are only finitely many choices for $\xi$, we will fix one throughout the proof.

Let $\Theta = \Psi_\Q \setminus \{ \alpha \}$, and let $\bP_\Theta=\bU_\Theta \bL_\Theta$ be the associated maximal $\Q$-parabolic subgroup. We fix a sequence $\bV_0 = \{ 1 \} < \bV_1 < \cdots < \bV_k = \bU_\Theta$ of subgroups defined over $\Q$ such that for all $i$, $\bV_i$ is normal in $\bU_\Theta$ and $\bV_i / \bV_{i-1} \simeq \mathbb{G}_a$.  We set $\Gamma_i = \xi^{-1} \Gamma \xi \cap V_i$, which is a lattice in $V_i$.

We have $V_i / V_{i-1} \simeq \R$, and the image of $\Gamma_i$ in $V_i / V_{i-1}$ is a lattice.  We choose a generator $\ell \in V_i / V_{i-1}$ for this lattice, and let $Y \in \text{Lie}(V_i) / \text{Lie}(V_{i-1})$ be such that $e^Y = \ell \in V_i / V_{i-1}$.

For all $i$, we define
\[
\psi_i(x) = \int_{ \Gamma_i \backslash V_i} \phi( \xi u x) du, \quad x \in G.
\]
We have $\psi_0(x) = \phi( \xi x)$, and $\psi_k = 0$ by the cuspidality of $\phi$.  We may therefore bound $\phi$ by bounding the successive differences $\psi_i - \psi_{i-1}$, using a standard integration by parts argument.  For each $x$, we may consider the function of a real variable $t$ given by $t \mapsto \psi_{i-1}( e^{tY} x)$.  Because $\psi_{i-1}$ is left-invariant under $V_{i-1}$ and $\Gamma_i$, and $e^Y$ lies in $V_{i-1} \Gamma_i$, we see that $\psi_{i-1}( e^{tY} x)$ is in fact a function on $\R / \Z$.  For $\eta \in \Z$ we define $\psi_{i-1}^\eta$ to be the $\eta$-Fourier coefficient of $\psi_{i-1}( e^{tY} x)$, given by
\[
\psi_{i-1}^\eta(x) = \int_{ \R / \Z } \int_{ \Gamma_{i-1} \backslash V_{i-1} } e^{- 2 \pi i \eta t} \phi( \xi u e^{tY} x ) du dt. 
\]
We have $\psi_{i-1}^0 = \psi_i$, and hence
\begin{equation}
\label{psi-difference}
\psi_{i-1} - \psi_i = \sum_{\eta \neq 0} \psi_{i-1}^\eta.
\end{equation}

We now assume that $x \in \mathfrak{S}_{t,\omega}$, and estimate $\psi_{i-1}^\eta$ for $\eta \neq 0$ by integration by parts.  If we define $Y'_x = \Ad(x^{-1}) Y$, then we have
\[
\psi_{i-1}^\eta(x) = \int_{ \R / \Z } \int_{ \Gamma_{i-1} \backslash V_{i-1} } e^{- 2 \pi i \eta t} \phi( \xi u x e^{tY'_x} ) du dt,
\]
and integrating by parts $N$ times with respect to $t$ gives
\begin{align}
\notag
\psi_{i-1}^\eta(x) & = (2 \pi i \eta)^{-N} \int_{ \R / \Z } \int_{ \Gamma_{i-1} \backslash V_{i-1} } e^{- 2 \pi i \eta t} ( Y_x^{'N} \phi)( \xi u x e^{tY'_x} ) du dt \\
\label{phi-tau}
& = (2 \pi i \eta)^{-N} \int_{ \R / \Z } \int_{ \Gamma_{i-1} \backslash V_{i-1} } e^{- 2 \pi i \eta t} ( Y_x^{'N} \phi)( \xi u e^{tY} x ) du dt.
\end{align}

We shall need to use the fact that $Y'_x$ becomes small as $\alpha(x)$ grows.  Let $X_j$ be a basis for $\g$, and let $X_j^*$ be the dual basis.  Then we may show in the same way as Moeglin--Waldspurger on \cite[p. 31]{MW} that we have
\[
X_j^*( Y'_x) \ll a(x)^{-\alpha}.
\]
It follows from this that we may write $Y_x^{'N} = a(x)^{-N \alpha} D_x$, where $D_x \in D(G)$ depends on $x$, and has degree $N$ and bounded coefficients when written in a basis of monomials in $X_j$.

We now apply Lemma \ref{lemma:mod-growth} to bound $(D_x \phi)( \xi u e^{tY} x )$.  We do not necessarily have $u e^{tY} x \in \mathfrak{S}_{t,\omega}$, but as we may assume that $u e^{tY}$ lies in a bounded subset of $U_\Theta$ we see that $u e^{tY} x$ lies in a larger Siegel set to which we may apply Lemma \ref{lemma:mod-growth}.  We therefore have
\[
( Y_x^{'N} \phi)( \xi u e^{tY} x ) = a(x)^{-N \alpha} (D_x \phi)( \xi u e^{tY} x ) \ll (1+\|\mu\|)^N \widetilde{\beta}_S(\mu)^{1/2} a(x)^{ \rho_\Q - N \alpha}.
\]
Applying this in \eqref{phi-tau} gives
\[
\psi_{i-1}^\eta(x) \ll_N \eta^{-N} (1+\|\mu\|)^N \widetilde{\beta}_S(\mu)^{1/2} a(x)^{ \rho_\Q - N \alpha}.
\]
If $N \geqslant 2$, applying this in \eqref{psi-difference} and summing over $\eta$ gives
\[
\psi_{i-1}(x) - \psi_i(x) \ll_N (1+\|\mu\|)^N \widetilde{\beta}_S(\mu)^{1/2} a(x)^{ \rho_\Q - N \alpha}.
\]
Summing this over $i$ completes the proof.
\end{proof}

We next use Lemma \ref{lemma:rapid-decay} to show that, when $x \in \Gamma \backslash G$ is above height $(1 + \| \mu \|)^{1 + \epsilon}$ in the cusp, $\phi$ is bounded by a constant independent of $\mu$.

\begin{lemma}
\label{lemma:phi-pointwise}
For any $c, \epsilon > 0$ there is a constant $C_\epsilon > 0$ such that the following holds.  If $\phi \in \mathcal{E}^{\rm cusp}_\mu(\Gamma\backslash G;\tau)$ satisfies $\| \phi \|_2 = 1$, $T > C_\epsilon (1 + \| \mu \|)^{1 + \epsilon}$, and $g \in (\Gamma \backslash G)^{> T}$, then $| \phi(g) | < c$.

\end{lemma}

\begin{proof}

We know that $g = \xi x$ for some $\xi \in \Xi$, $x \in \mathfrak{S}_{t,\omega}$.  Because $g \in (\Gamma \backslash G)^{> T}$, we must have $x \in \mathfrak{S}_{t,\omega}^{> T}$, and hence $a(x)^\alpha > T$ for some $\alpha \in \Psi_\Q$. Choose $\alpha \in \Psi_\Q$ so that $a(x)^\alpha$ is maximal.  If we apply Lemma \ref{lemma:rapid-decay} with this $\alpha$, we get that for any $N \geqslant 2$ there is $C_N > 0$ such that
\[
|\phi( \xi x)| \leqslant C_N \widetilde{\beta}_S(\mu)^{1/2} ( 1 + \| \mu \|)^N a(x)^{ \rho_\Q - N \alpha}.
\]
Because $a(x)^\alpha \geqslant a(x)^\beta$ for $\beta \in \Psi_\Q$, there is $M_1 > 0$ such that $a(x)^{M_1 \alpha} \geqslant a(x)^{ \rho_\Q}$.  We then choose $M_2 > 0$ such that $(1 + \| \mu \|)^{M_2 \epsilon} \geqslant \widetilde{\beta}_S(\mu)^{1/2} ( 1 + \| \mu \|)^{M_1}$, and let $N = M_1 + M_2$.  With this choice, we have
\begin{align*}
|\phi( \xi x)| & \leqslant C_N \widetilde{\beta}_S(\mu)^{1/2} ( 1 + \| \mu \|)^N a(x)^{ \rho_\Q - N \alpha} \\
& \leqslant C_N a(x)^{ \rho_\Q - M_1 \alpha} \widetilde{\beta}_S(\mu)^{1/2} ( 1 + \| \mu \|)^{M_1} ( a(x)^\alpha / (1 + \| \mu \|) )^{-M_2} \\
& \leqslant C_N \widetilde{\beta}_S(\mu)^{1/2} ( 1 + \| \mu \|)^{M_1} C_\epsilon^{-M_2} (1 + \| \mu \|)^{-M_2 \epsilon} \\
& \leqslant C_N C_\epsilon^{-M_2}.
\end{align*}
Choosing $C_\epsilon$ sufficiently large completes the proof.
\end{proof}

Finally, we use the bounds we have established for $\phi$ in the cusp to prove the desired control on its $L^2$ mass.

\begin{lemma}\label{L2-mass}

Given $0 < \kappa < 1$ and $\epsilon > 0$, there is $C_\epsilon > 0$ such that the following holds.  If $\phi \in \mathcal{E}^{\rm cusp}_\mu(\Gamma\backslash G;\tau)$ satisfies $\| \phi \|_2 = 1$, and $T > C_\epsilon (1 + \| \mu \|)^{1 + \epsilon}$, then
\[
\int_{(\Gamma\backslash G)^{\leqslant T} } | \phi(g) |^2 dg \geqslant \kappa.
\]
\end{lemma}

\begin{proof}
Since $\|\phi\|_2=1$, we need to show that $\int_{(\Gamma\backslash G)^{> T} } | \phi(g) |^2 dg$ can be made as small as possible.
If we apply Lemma \ref{lemma:phi-pointwise} with the given $\epsilon$ and $c = 1$, we obtain $C_\epsilon > 0$ such that if $T > C_\epsilon (1 + \| \mu \|)^{1 + \epsilon}$ and $g \in (\Gamma \backslash G)^{> T}$, then $| \phi(g) | < 1$.  It follows that
\[
\int_{(\Gamma\backslash G)^{> T} } | \phi(g) |^2 dg \leqslant \text{vol}( (\Gamma\backslash G)^{> T} ).
\]
Because $\text{vol}( (\Gamma\backslash G)^{> T} ) \to 0$ as $T \to \infty$, making $C_\epsilon$ sufficiently large gives the result.
\end{proof}

\subsection{Proof of Proposition \ref{Weyl-upper-bd}}\label{sec:proof-Prop-Weyl}
We now return to Proposition \ref{Weyl-upper-bd} and its proof. 

We apply Proposition \ref{test-fn} to each simple factor $G_j$ of $G^0$, relative to the spectral parameters $-\nu_j$ obtained by restricting $-\nu$ to the $j$-th component of $i\mathfrak{a}$, and with $\tau$ there chosen to be $\tau_j^{-1}$.  For each $j$, we let $k^0_{-\nu_j}$ be the test function appearing in the proof of Proposition \ref{test-fn}, and define $k^0_{-\nu} = \prod k^0_{-\nu_j}$ and $k_{-\nu} = k^0_{-\nu} * ( k^0_{-\nu} )^*$.  We let

\begin{equation}\label{defK}
K_{-\nu} (x,y)=\sum_{\gamma\in\Gamma}k_{-\nu} (x^{-1}\gamma y)\in C^\infty(G \times G,\tau \times\tau^{-1})
\end{equation}
be the associated automorphic kernel function.

Let $\phi_i$ be an orthonormal basis of eigenfunctions for $L^2_\text{cusp}( \Gamma \backslash G, \tau)$, with spectral parameters $\lambda_i \in \Lambda_\text{cusp}( \Gamma; \tau)$.  After possibly scaling $k^0_{-\nu}$, Proposition \ref{test-fn} implies that if $\| \text{Im} \, \lambda_i - \nu \| \leqslant 1$, we have $R( k^0_{-\nu}) \phi_i = c_i \phi_i$ with $|c_i| \geqslant 1$.  We may therefore apply Lemma \ref{pre-trace} to the orthonormal set containing those $\phi_i$ with $\| \text{Im} \, \lambda_i - \nu \| \leqslant 1$ to obtain
\begin{equation}\label{eq:pre-trace-inequality}
\sum_{\substack{\lambda\in \Lambda_{\rm cusp}(\Gamma;\tau)\\ \|{\rm Im}\,\lambda-\nu\| \leqslant 1}} | \phi_i(x) |^2 \leqslant K_{-\nu} (x, x).
\end{equation}
We apply Lemma \ref{L2-mass} with $\kappa = 1/2$ and any $\epsilon > 0$ to obtain $C_\epsilon > 0$ such that if $T = C_\epsilon ( 1 + \| \nu \|)^{1 + \epsilon}$, then
\[
\int_{(\Gamma\backslash G)^{\leqslant T} } | \phi_i(x) |^2 dx \geqslant 1/2
\]
for all $\| \text{Im} \, \lambda_i - \nu \| \leqslant 1$. Integrating the pretrace inequality \eqref{eq:pre-trace-inequality} over the truncated domain $(\Gamma\backslash G)^{\leqslant T}$, we therefore have
\[
\sum_{\substack{\lambda\in \Lambda_{\rm cusp}(\Gamma;\tau)\\ \|{\rm Im}\,\lambda-\nu\| \leqslant 1}} \dim \mathcal{E}_\lambda^{\rm cusp}(\Gamma\backslash G,\tau) \leqslant 2 \int_{(\Gamma\backslash G)^{\leqslant T} } K_{-\nu} (x, x) dx.
\]
Because $(\Gamma\backslash G)^{\leqslant T}$ is the image of $\Xi \mathfrak{S}_{t,\omega}^{\leqslant T}$ in $\Gamma \backslash G$, we have
\[
\int_{(\Gamma\backslash G)^{\leqslant T} } K_{-\nu} (x, x) dx \leqslant \sum_{\xi \in \Xi} \int_{ x \in \mathfrak{S}_{t,\omega}^{\leqslant T} } K_{-\nu} (\xi x, \xi x) dx.
\]

Let $B$ be a fixed compact set that contains the support of $k_{-\nu}$ for all $\nu$.  Because $k_{-\nu} \ll_\tau \widetilde{\beta}_S(\nu) 1_B$ by Lemma \ref{k-nu-e}, we may apply Lemma \ref{cusp-kernel} to obtain
\[
\int_{ x \in \mathfrak{S}_{t,\omega}^{\leqslant T} } K_{-\nu} (\xi x, \xi x) dx \ll_{\Gamma, \tau} \widetilde{\beta}_S(\nu) \int_{ x \in \mathfrak{S}_{t,\omega}^{\leqslant T} } a(x)^{2 \rho_\Q} dx.
\]
Writing out the measure on $\mathfrak{S}_{t,\omega}^{\leqslant T}$ in Iwasawa coordinates, as in \eqref{Q-Iwasawa-measure1}, we see that
\[
\int_{ x \in \mathfrak{S}_{t,\omega}^{\leqslant T} } a(x)^{2 \rho_\Q} dx\asymp \int_{A_t^{\leqslant T}} da\asymp (\log T)^r \ll_\epsilon ( \log ( 3 + \| \nu \| ) )^r,
\]
which completes the proof of Proposition \ref{Weyl-upper-bd}.

\section{Mean square estimates for discrete automorphic periods}\label{sec:points}

Let $G = \Oo(n,m)$, with maximal compact $K = \Oo(n) \times \Oo(m)$.  Let $\Gamma$ be a uniform lattice in $G$, with associated (compact) locally symmetric space $Y = \Gamma \backslash \mathbb{H}^{n,m}$. In this section we examine the asymptotic size of finite sums of point evaluations of Maass forms on $Y$.  Our main result is Proposition \ref{local-Weyl}.  Before proving this, we first recall some basic results on spherical representations and spectral parameters for the Zariski-disconnected group $G$ in Section \ref{sec:O-spherical}.  

While the background material on spectral parameters will be stated for general $n$ and $m$, we shall only prove our results on point evaluations in the case $n > m$.  We do this for simplicity, and because it is the only case we need in this paper.  However, the results could be extended to the case $n = m$ with a small amount of extra work.

\subsection{Notation}
\label{sec:O-spherical-notn}

We let the real Lie algebras $\g$, $\gk$, $\ga$, and $\gn$ have their usual meaning, and we let $\mathfrak{m}$ be the centralizer of $\ga$ in $\gk$.  We let $\Sigma = \Sigma(\ga, \g)$ be the roots of $\ga$ in $\g$.  We let $\Sigma^+$ be a choice of positive roots, with corresponding half-sum $\rho_\ga$.  We let $W_\ga$ be the Weyl group of $\Sigma$, i.e., the group generated by the reflections in the root hyperplanes.  We let $\widetilde{W}_\ga = N_K(\ga) / Z_K(\ga)$ be the extended Weyl group of $\Sigma$ corresponding to the disconnected group $G$.  If $n > m$, then $\widetilde{W}_\ga$ is the same as $W_\ga$.  If $n = m$, and we identify $\ga$ with $\R^m$ in the standard way, then $W_\ga$ is the group of linear transformations that permute the coordinates and change an even number of signs, and $\widetilde{W}_\ga$ is the group that permutes the coordinates and changes an arbitrary number of signs.

Let $\gh_\gk$ be a Cartan subalgebra of $\mathfrak{m}$, so that $\gh = \ga + \gh_\gk$ is a Cartan subalgebra of $\g$.  We let $\Delta = \Delta(\gh, \g)$ be the roots of $\gh$ in $\g$.  We let $\Delta^+$ be a choice of positive roots that is compatible with $\Sigma^+$.  We let $\rho$ be the half-sum of $\Delta^+$, and let $\rho_{\mathfrak{m}}$ be the half-sum of a system of positive roots for $\mathfrak{m}$ that is compatible with $\Delta^+$, so that $\rho = \rho_\ga + \rho_{\mathfrak{m}}$.  We let $W$ be the Weyl group of $\Delta$, which is generated by the reflections in the root hyperplanes, and let $\widetilde{W} = N_{G(\C)}(\ga_\C) / Z_{G(\C)}(\ga_\C)$ be the extended Weyl group corresponding to the disconnected group $G$.  We have $W = \widetilde{W}$ if and only if $n+m$ is odd.

We let $\cU(\g)$ denote the universal enveloping algebra of $\g$, and $\cZ(\g)$ its center.  We shall also need to consider the algebra $\cZ(\g)^G$ of $G$-invariants in $\cZ(\g)$, which may be strictly smaller than $\cZ(\g)$ because $G$ is not Zariski-connected. We have the Harish-Chandra isomorphism $\gamma_{{\rm HC},\gh}: \cZ(\g) \xrightarrow{\,\sim\,} S(\gh)^W$, where $S(\gh)$ is the symmetric algebra of $\gh$ with complex coefficients.

\subsection{Spherical representations of $\Oo(n,m)$.}
\label{sec:O-spherical}

This section contains basic results on spherical representations of $G$.  We will prove these results by passing between $G$ and its Zariski-connected subgroup $G^0 = \SO(n,m)$, which has maximal compact subgroup $K^0 = \text{S}( \Oo(n) \times \Oo(m) )$.  Let $\mathbb{D}(G/K)$ and $\mathbb{D}(G^0/K^0)$ be the rings of differential operators on $\mathbb{H}^{n,m}$ that are invariant under $G$ and $G^0$ respectively.  We have the Harish-Chandra isomorphism $\gamma_{\text{HC},\ga} : \mathbb{D}(G^0/K^0) \to S(\ga)^{W_\ga}$ for $G^0$, see for instance \cite[Theorem 2.6.7]{GaVa}.

Our first result establishes versions of the Harish-Chandra isomorphisms $\gamma_{\textup{HC},\gh}$ and $\gamma_{\textup{HC},\ga}$ for the disconnected group $G$.

\begin{lemma}
\label{O-Harish-Chandra}

The map $\gamma_{\textup{HC},\gh}$ induces an isomorphism $\cZ(\g)^G \simeq S(\gh)^{\widetilde{W}}$, and $\gamma_{\textup{HC},\ga}$ induces an isomorphism $\mathbb{D}(G/K) \simeq S(\ga)^{\widetilde{W}_\ga}$.

\end{lemma}

\begin{proof}

We shall only prove this for $\gamma_{\textup{HC},\ga}$, as the case of $\gamma_{\textup{HC},\gh}$ is simpler.  It will be helpful to explicate the relation between the algebras $\mathbb{D}(G/K)$ and $\mathbb{D}(G^0/K^0)$ using the commutative diagram
\begin{equation*}
\xymatrix{
\mathbb{D}(G/K) \ar[r]  & \mathbb{D}(G^0/K^0) \\
\cU(\g)^K / \cU(\g)^K \cap \cU(\g) \gk  \ar[r]\ar[u]^{\wr}& \ar[u]^{\wr} \cU(\g)^{K^0} / \cU(\g)^{K^0} \cap \cU(\g) \gk .
}
\end{equation*}
The horizontal maps are the obvious inclusions. That the vertical maps are isomorphisms is well-known; see, for example, \cite[Prop. 1.7.2]{GaVa}. The right-hand vertical map is equivariant for the $K$-action, where $K$ acts on $\mathbb{D}(G^0/K^0)$ through its conjugation action on $G^0/K^0$.
Moreover, the left-hand column is equal to the $K$-invariants in the right-hand column.

Because $\mathbb{D}(G/K)$ is equal to the $K$-invariants in $\mathbb{D}(G^0/K^0)$, we may prove the lemma by understanding how $\gamma_{\textup{HC},\ga}$ interacts with conjugation by $K$, and to this end we recall how $\gamma_{\textup{HC},\ga}$ is defined \cite[Theorem 2.6.7]{GaVa}. If $D \in \mathbb{D}(G^0/K^0)$ is represented by an element $u_D \in \cU(\g)^{K^0}$, then $\gamma_{\textup{HC},\ga}(D)$ is defined by first choosing the unique $a_D \in S(\ga)$ such that
\[
u_D - a_D \in \gn \, \cU(\g) + \cU(\g) \gk,
\]
then setting $\gamma_{\textup{HC},\ga}(D) = \rho^* a_D$, where $\rho^*$ is the automorphism of $S(\ga)$ that sends $X$ to $X + \rho(X)$ for every $X \in \ga$. In particular, every $D \in \mathbb{D}(G^0/K^0)$ is determined by its `$\ga$-part' $a_D$.

We now consider the action of $K$ on $\mathbb{D}(G^0/K^0)$.  To understand this action, it will be useful to choose an element $k \in K - K^0$ that preserves the Iwasawa decomposition of $\g$.  We note that such a $k$ exists, because any $k' \in K - K^0$ transforms the Iwasawa decomposition $\g = \gn \oplus \ga \oplus \gk$ into one of the form $\g = \Ad(k')\gn \oplus \Ad(k')\ga \oplus \gk$, and this new decomposition may be transformed back to the standard one by an element $k_0$ of $K^0$; we then take $k = k_0 k'$.  The action of $\Ad(k)$ on $\ga$ is via an element $w \in \widetilde{W}_\ga$ that preserves the positive roots.  Because $\Ad(k)$ preserves the Iwasawa decomposition, we have $a_{\Ad(k) D} = \Ad(k) a_D = wa_D$, and this implies that $\gamma_{\textup{HC},\ga}(\Ad(k)D) = w\gamma_{\textup{HC},\ga}(D)$.

If $n > m$, then $W_\ga = \widetilde{W}_\ga$, and $w$ is equal to the identity because the Dynkin diagram of $\Sigma$ has no non-trivial automorphisms.  This implies that $S(\ga)^{W_a} = S(\ga)^{\widetilde{W}_\ga}$ and $\mathbb{D}(G/K) = \mathbb{D}(G^0/K^0)$, which gives the lemma.  If $n = m$, then $w \in \widetilde{W}_\ga - W_\ga$ is nontrivial, and the lemma again follows.
\end{proof}

Lemma \ref{O-Harish-Chandra} allows us to parametrize irreducible spherical representations of $G$.  (We shall technically work with $(\g,K)$-modules, but refer to them as representations by slight abuse of terminology.)  If $(\pi, V)$ is such a representation, we note that $\dim V^K = 1$, which follows from the commutativity of $\mathbb{D}(G/K)$ in the usual way. We may then use Lemma \ref{O-Harish-Chandra} to define the spectral parameter of $\pi$ as the element $\lambda \in \ga_\C^* / \widetilde{W}_\ga$ corresponding to the character by which $\mathbb{D}(G/K) \simeq S(\ga)^{\widetilde{W}_\ga}$ acts on $V^K$.  The following result shows that $\pi$ is classified by its spectral parameter.

\begin{prop}
\label{O-spherical}
For any $\lambda \in \ga_\C^* / \widetilde{W}_\ga$, there is a unique irreducible admissible spherical representation $\pi_\lambda$ of $G$ with that spectral parameter.  The infinitesimal character of $\pi_\lambda$ is given by $\lambda + \rho_{\mathfrak{m}} \in \gh^*_\C / \widetilde{W}$.

\end{prop}

\begin{proof}

We begin with the uniqueness of $\pi_\lambda$.  Let $(\pi_1, V_1)$ and $(\pi_2, V_2)$ be two irreducible spherical representations with the same spectral parameter.  Let $v \in V_1^K \oplus V_2^K$ be a spherical vector that doesn't lie in either $V_1^K$ or $V_2^K$, and let $V$ be the subrepresentation it generates as a $(\g,K)$-module.  We claim that $V^K = \C v$.  Indeed, if $w \in V^K$, then there are $k_i \in K$ and $X_i \in \cU(\g)$ such that
\[
w = \sum_i X_i k_i v = \Big( \sum_i X_i \Big) v = Yv,
\]
where $Y \in \mathbb{D}(G/K)$.  Moreover, $\mathbb{D}(G/K)$ acts by scalars on $V_1^K \oplus V_2^K$, since, by assumption, both summands share the same spectral parameter.  This implies that $w \in \C v$ as required.  Finally, because $V^K = \C v$, we see that $V$ is a nonzero submodule of $V_1 \oplus V_2$ that is not equal to $V_1$, $V_2$, or $V_1 \oplus V_2$.  $V$ therefore induces an isomorphism $\pi_1 \simeq \pi_2$.

For existence, we construct $\pi_\lambda$ by parabolic induction. Let $P$ be the minimal parabolic subgroup of $G^0$ defined over $\R$.  We let $P = NAM$ be the Langlands decomposition of $P$, so that $A$ is the connected Lie group with Lie algebra $\ga$.  Let $\lambda \in \ga_\C^*$, and let $e^\lambda$ be the corresponding character of $A$, which we extend to $P$.  The unitarily induced representation $\text{Ind}_P^G \, e^\lambda$ has a unique irreducible spherical subquotient $\pi_\lambda$, and it may be checked that $\pi_\lambda$ has spectral parameter $\lambda$ as required.

Finally, let $\gamma_\lambda:\cZ(\g)^G \rightarrow \C$ be the infinitesimal character of $\pi_\lambda$.  Under the Harish-Chandra isomorphism $\gamma_{{\rm HC},\gh}$ of Lemma \ref{O-Harish-Chandra}, $\gamma_\lambda$ determines an element in $\gh_\C^*/\widetilde W$, which we must show can be represented by $\lambda+\rho_{\mathfrak{m}}$.  If we also let $\gamma_{{\rm HC},\ga} : \mathbb{D}(G/K) \xrightarrow{\,\sim\,} S(\ga)^{\widetilde{W}_\ga}$ be as in Lemma \ref{O-Harish-Chandra}, the spectral parameter $\lambda\in\ga_\C^*/\widetilde W_\ga$ of $\pi_\lambda$ satisfies $\chi_\lambda(D)=\lambda(\gamma_{{\rm HC},\ga}(D))$ for a uniquely determined algebra homomorphism $\chi_\lambda: \mathbb{D}(G/K)\rightarrow\C$, and $\gamma_\lambda$ is given by the composition
\[
\gamma_\lambda: \cZ(\g)^G \to \mathbb{D}(G/K) \xrightarrow{\,\chi_\lambda \,} \C.
\]
We may compute this by comparison with $G^0$.  We choose an element $\lambda' \in \ga_\C^* / W_\ga$ that lifts $\lambda$; then the associated algebra homomorphism $\chi_{\lambda'}: \mathbb{D}(G^0/K^0) \to \C$ extends $\chi_\lambda$. It follows that the restriction to $\cZ(\g)^G$ of the character
\[
\cZ(\g) \to \mathbb{D}(G^0/K^0) \xrightarrow{\,\chi_{\lambda'}\,} \C
\]
coincides with $\gamma_\lambda$. Since this character of $\cZ(\g)$ corresponds to $\lambda' + \rho_{\mathfrak{m}} \in \gh^*_\C / W$ under $\gamma_{\textup{HC},\gh}$, see e.g. the proof of Proposition 2.1 in \cite{Helgason2}, it follows that $\gamma_\lambda$ corresponds to $\lambda + \rho_{\mathfrak{m}} \in \gh^*_\C / \widetilde{W}$ under $\gamma_{\textup{HC},\gh}|_{\cZ(\g)^G}$.
\end{proof}

\subsection{A second moment asymptotic for discrete periods}

We now return to the global setting. We continue to let $G = \Oo(n,m)$ and $K = \Oo(n) \times \Oo(m)$. For a uniform lattice $\Gamma<G$ we denote by $Y=\Gamma\backslash \mathbb{H}^{n,m}$ the associated (compact) locally symmetric space.

In the rest of this section we assume that $n > m$. This constraint ensures that objects such as spherical functions, spherical representations, and Maass forms behave in essentially the same way for $G$, $G^0$, and the connected Lie group $G^{00} = \SO(n,m)$ with maximal compact subgroup $K^{00} = \SO(n) \times \SO(m)$. We explain this in the next paragraph.

Recall from \S\ref{sec:O-spherical-notn} that  $W_{\mathfrak{a}}=\widetilde{W_{\mathfrak{a}}}$ whenever $n>m$. Under this condition, a function on $\mathbb{H}^{n,m}=G/K=G^{00}/K^{00}$ is seen to be $K$-invariant if and only if it is $K^{00}$-invariant. Moreover, combining the Weyl group identity with Lemma \ref{O-Harish-Chandra}, we deduce in this case that $\mathbb{D}(G/K) = \mathbb{D}(G^0/K^0)$, and these are both equal to $\mathbb{D}(G^{00}/K^{00})$ because $G^0$ lies in the Harish-Chandra class. It follows that the notion of a joint eigenfunction, and its spectral parameter, on $\mathbb{H}^{n,m}$ is the same regardless of whether the latter is viewed as $G/K$ or $G^{00}/K^{00}$. Likewise, the definitions of spherical functions, and the Harish-Chandra transform, are insensitive to the choice of presentation of $\mathbb{H}^{n,m}$ as a symmetric space for $G$ or $G^{00}$. Finally, we observe that occurrences of $\pi_\lambda$ in $L^2(\Gamma \backslash G)$ are equivalent to Maass forms on $Y=\Gamma \backslash \mathbb{H}^{n,m}$ with spectral parameter $\lambda$, and that if $\pi_\lambda$ is unitarizable, then $\text{Re} \, \lambda$ lies in the convex hull of the set $W_\ga \rho_\ga$.

In preparation for the proof of Proposition \ref{local-Weyl}, we establish the following result.

\begin{lemma}\label{upper-local-Weyl} For any $x\in Y$, $\nu\in i\mathfrak{a}^*$, and orthonormal basis $\{f_\mu\}$ of $L^2(Y)$ consisting of Maass forms $f_\mu$ of spectral parameter $\mu$, we have
\[
\sum_{\|{\rm Im}\,\mu-\nu\|\leqslant 1}|f_\mu(x)|^2\ll \tilde\beta_S(\nu),
\]
the implied constant depending on $\Gamma$. In particular, dropping all but one term, we obtain the local (or trivial) bound $\|f_\nu\|_\infty\ll \tilde\beta_S(\nu)^{1/2}$.
\end{lemma}

\begin{proof}

We may assume without loss of generality that $\Gamma$ is torsion free, by using Selberg's lemma to pass to a torsion-free subgroup of finite index if necessary.

We let $k_{-\nu}$ be the test function on $G^{00} = \SO(n,m)^0$ defined in Proposition \ref{test-fn}, which may be viewed as a $K$-biinvariant function on $G$ by the remarks preceding the statement of the lemma.  We spectrally expand the associated automorphic kernel function $K_{-\nu}$. Since $Y$ is compact, there is no continuous spectrum, and we obtain
\begin{equation}\label{eq:compact-expansion}
\sum_{\mu}\widehat{k}_{-\nu}(-\mu)f_\mu(x)\overline{f_\mu(y)}=K_{-\nu}(x,y)=\sum_{\gamma\in\Gamma}k_{-\nu}(x^{-1}\gamma y).
\end{equation}
We now take $x=y$ so that
\[
\sum_{\mu}\widehat{k}_{-\nu}(-\mu)|f_\mu(x)|^2=\sum_{\gamma\in\Gamma}k_{-\nu}(x^{-1}\gamma x).
\]
Taking the support of $k_{-\nu}$ small enough, only $\gamma = e$ contributes to the sum (since $\Gamma$ is torsion-free), so that
\begin{equation}\label{simple-pre-trace}
\sum_{\mu}\widehat{k}_{-\nu}(-\mu)|f_\mu(x)|^2= k_{-\nu}(e).
\end{equation}
We then apply the estimate $k_{-\nu}(e)\ll \tilde\beta_S(\nu)$ of Proposition \ref{k-nu-e} to bound the right-hand side.

For the left-hand side of \eqref{simple-pre-trace}, we apply Property \eqref{1} of Proposition \ref{test-fn}, which allows us to drop all terms but those for which $\|{\rm Im}\,\mu-\nu\|\leqslant 1$. For the remaining terms, we use Property \eqref{3} of Proposition \ref{test-fn} to find
\[
c\sum_{\|{\rm Im}\,\mu-\nu\|\leqslant 1}|f_\mu(x)|^2\leqslant \sum_{\mu}\widehat{k}_{-\nu}(-\mu)|f_\mu(x)|^2.
\]
Putting these bounds together establishes the lemma.
\end{proof}

We now complement Lemma \ref{upper-local-Weyl} by giving (upper and) lower bounds for the mean square of a finite discrete period on $Y$. We retain the same notations and assumptions as above.

\begin{prop}\label{local-Weyl}
There is $Q\geqslant 1$ such that for any finite collection of points $x_1,\ldots ,x_h\in Y$, together with non-zero weights $c_i\in \C$, any spectral parameter $\nu\in i\ga^*$ of large enough norm depending on $x_i$, and any orthonormal basis $\{f_\mu\}$ of $L^2(Y)$ consisting of Maass forms $f_\mu$ of spectral parameter $\mu$, we have
\begin{equation}\label{cor-cmpt-packet}
\sum_{\|{\rm Im}\,\mu-\nu\|\leqslant Q}\bigg|\sum_{i=1}^h c_i f_\mu(x_i)\bigg|^2 \asymp \widetilde{\beta}_S(\nu).
\end{equation}
The implied constant in the lower bound depends only on $\Gamma$ and the weights $c_i$, while in the upper bound it depends additionally on $Q$.

\end{prop}

\begin{proof}

It suffices to prove the lower bound, since the upper bound follows from Lemma \ref{upper-local-Weyl}. We set up the proof as for Lemma \ref{upper-local-Weyl}. We may therefore take as our starting point the pre-trace formula \eqref{eq:compact-expansion}.

We begin with an analysis of the spectral side of \eqref{eq:compact-expansion}, for arbitrary $x,y\in Y$. In contrast to the proof of Lemma \ref{upper-local-Weyl}, we shall need to show that the spectral summation is centered about $\nu$, up to a negligible error. We claim that for any $Q>1$, we have
\begin{equation}\label{2show}
\sum_{\|{\rm Im}\, \mu-\nu\|>Q}\widehat{k}_{-\nu}(-\mu)f_\mu(x)\overline{f_\mu(y)}\ll_A Q^{-A} \widetilde{\beta}_S(\nu).
\end{equation}
For this, we cover the region $\{\lambda\in i\mathfrak{a}^*: \|\lambda-\nu\|>Q\}$ by the union of unit balls centered at points $\lambda_n$. From the rapid decay estimate \eqref{4} from Proposition \ref{test-fn}, the Cauchy--Schwarz inequality, and Lemma \ref{upper-local-Weyl} we get
\begin{align*}
\sum_{\|{\rm Im}\,\mu-\nu\|>Q}\widehat{k}_{-\nu}(-\mu)&f_\mu(x)\overline{f_\mu(y)}\\
&\ll_A \sum_n \|\lambda_n-\nu\|^{-A}\sum_{\|{\rm Im}\,\mu-\lambda_n\|\leqslant 1}|f_\mu(x)f_\mu(y)|\\
&\leqslant \sum_n \|\lambda_n-\nu\|^{-A}\bigg(\sum_{\|{\rm Im}\,\mu-\lambda_n\|\leqslant 1}|f_\mu(x)|^2\bigg)^{1/2}\bigg(\sum_{\|{\rm Im}\,\mu-\lambda_n\|\leqslant 1}|f_\mu(y)|^2\bigg)^{1/2}\\
&\ll \sum_n \|\lambda_n-\nu\|^{-A}\tilde\beta_S(\lambda_n).
\end{align*}
Inserting Lemma \ref{c-bound}, we find
\[
\sum_{\|{\rm Im}\,\mu-\nu\|>Q}\widehat{k}_{-\nu}(-\mu)f_\mu(x)\overline{f_\mu(y)}\ll_A \tilde\beta_S(\nu)\sum_n\|\lambda_n-\nu\|^{-A}\ll_A Q^{-A}\tilde\beta_S(\nu),
\]
which is \eqref{2show}.  By expanding the square, \eqref{2show} implies that
\begin{equation}
\label{spectral-trunc}
\sum_{\|{\rm Im}\,\mu-\nu\|>Q} \widehat{k}_{-\nu}(-\mu) \bigg|\sum_{i=1}^h c_i f_\mu(x_i)\bigg|^2 \ll_A Q^{-A}\tilde\beta_S(\nu).
\end{equation}

\medskip

We now move to the geometric side of \eqref{eq:compact-expansion}. Recall the bi-$K$-invariant distance function $d:G\rightarrow\R$ from \S\ref{sec:gp-decomp}. For two points $x=\Gamma gK$ and $y=\Gamma hK$ on $Y$, we put $d_Y(x,y)=\min_{\gamma\in\Gamma}d(g,\gamma h)$. We begin by showing that for any $x, y \in Y$ we have
\begin{equation}
\label{eq:x-y}
\sum_{\mu}\widehat{k}_{-\nu}(-\mu)f_\mu(x)\overline{f_\mu(y)} \ll \big(1+\|\nu\| d_Y(x,y)\big)^{-1/2} \widetilde{\beta}_S(\nu).
\end{equation}
 We may restrict the geometric side of \eqref{eq:compact-expansion} to those $\gamma \in \Gamma$ with $x^{-1} \gamma y \in \text{supp}(k_{-\nu})$, and this is a finite set whose size depends only on $\Gamma$.  For each of these $\gamma$ we write $x^{-1} \gamma y = k_1 e^H k_2$, where $H \in \ga$ satisfies $\| H \| = d(x, \gamma y)$. Applying the bound of Lemma \ref{knu-decay}, and using $d(x,y)\geqslant d_Y(x,y)$, we find
\[
k_{-\nu}(x^{-1} \gamma y) \ll \big(1+\|\nu\| \| H \| \big)^{-1/2} \widetilde{\beta}_S(\nu) \leqslant \big(1+\|\nu\| d_Y(x,y)\big)^{-1/2} \widetilde{\beta}_S(\nu).
\]
Applying this bound to each term of the geometric side in \eqref{eq:compact-expansion} gives \eqref{eq:x-y}.

We next consider the case when $x=y$.  We split the geometric side of \eqref{eq:compact-expansion} into those $\gamma \in \Gamma$ that fix $x$, and the complement, which gives
\[
\sum_{\mu}\widehat{k}_{-\nu}(-\mu)|f_\mu(x)|^2= |\text{Stab}_\Gamma(x)| k_{-\nu}(e) + \sum_{ \gamma \in \Gamma - \text{Stab}_\Gamma(x)} k_{-\nu}( x^{-1} \gamma x).
\]
As in the case where $x \neq y$, the sum on the right-hand side is over a finite set, and each term may be bounded by $O_x( \| \nu \|^{-1/2} \widetilde{\beta}_S(\nu))$ so that
\[
\sum_{\mu}\widehat{k}_{-\nu}(-\mu)|f_\mu(x)|^2 \asymp k_{-\nu}(e) + O_x( \| \nu \|^{-1/2} \widetilde{\beta}_S(\nu)).
\]
Combining this with \eqref{eq:x-y} gives
\begin{equation}
\label{geo-main}
\sum_{\mu} \widehat{k}_{-\nu}(-\mu) \bigg|\sum_{i=1}^h c_i f_\mu(x_i)\bigg|^2 \asymp k_{-\nu}(e) + O_{x_i}( \| \nu \|^{-1/2} \widetilde{\beta}_S(\nu)).
\end{equation}

Combining the geometric estimate \eqref{geo-main} with the spectral truncation \eqref{spectral-trunc} gives
\[
\sum_{\|{\rm Im}\,\mu-\nu\| \leqslant Q} \widehat{k}_{-\nu}(-\mu) \bigg|\sum_{i=1}^h c_i f_\mu(x_i)\bigg|^2 \asymp k_{-\nu}(e)+ O_{x_i}( \| \nu \|^{-1/2} \widetilde{\beta}_S(\nu)) + O_{A}(Q^{-A}\widetilde{\beta}_S(\nu)).
\]
Since $\widehat{k}_{-\nu}(-\mu)\ll 1$ we obtain
\begin{equation}\label{eq:pretrace-w-main}
k_{-\nu}(e)+ O_{x_i}( \| \nu \|^{-1/2} \widetilde{\beta}_S(\nu)) + O_{A}(Q^{-A}\widetilde{\beta}_S(\nu)) \ll \sum_{\|{\rm Im}\,\mu-\nu\| \leqslant Q} \bigg|\sum_{i=1}^h c_i f_\mu(x_i)\bigg|^2.
\end{equation}
From Properties \eqref{1} and \eqref{3} of Proposition \ref{test-fn} we have
\[
k_{-\nu}(e)\geqslant c \int_{\substack{\mu \in i\mathfrak{a}^*\\\|\mu + \nu\|\leqslant  1}}|{\bf c}(\mu)|^{-2} d\mu.
\]
We may apply Lemma \ref{lemma:c-bd} to show that the last integral is $\asymp \widetilde{\beta}_S(\nu)$. Fixing $A>1$ and taking $Q>1$ large enough, and also taking $\nu$ large enough depending on the $x_i$, we can make the error term in \eqref{eq:pretrace-w-main} smaller by a constant factor than the main term, which gives \eqref{cor-cmpt-packet}.
\end{proof}

\section{Review of the theta correspondence}\label{sec:theta-review}

In this section, we review several properties of the local and global theta correspondence which will be useful in the proofs of our main theorems.

\subsection{Type I algebraic reductive dual pairs}\label{Type1}
Let $k$ be a field of characteristic zero. 

Let $W$ denote a symplectic space over $k$, with symplectic form $\langle\,,\,\rangle_W$. Then $W$ admits a natural left-action by $\GL(W)$ and we denote by $\bG'=\Sp(W)$ the associated symplectic group. Let $V$ be a non-degenerate quadratic space over $k$, and let $\bG=\Oo(V)$. The tensor product $\mathcal{W}=W\otimes V$ is naturally given the structure of a symplectic space by the rule $\langle w_1\otimes v_1,w_2\otimes v_2\rangle_{\mathcal{W}}=\langle w_1,w_2\rangle_W\langle v_1, v_2\rangle_V$. We again let $\GL(\mathcal{W})$ act on the left, writing $\Sp(\mathcal{W})$ for the symplectic group consisting of those automorphisms preserving $\langle \,,\,\rangle_{\mathcal{W}}$.

Let $W=U\oplus U^*$ be a complete polarization of $W$ and put $\mathcal{X}=U\otimes V$ and $\mathcal{Y}=U^*\otimes V$. With respect to the decomposition $\mathcal{W}=\mathcal{X}\oplus\mathcal{Y}$, we may view the elements in $\Sp(\mathcal{W})$ as $\langle\,,\,\rangle_{\mathcal{W}}$-preserving invertible matrices
$\left(\begin{smallmatrix}a & b\\ c & d\end{smallmatrix}\right)$, where $a\in {\rm End}(\mathcal{X})$, $b\in \Hom(\mathcal{Y},\mathcal{X})$,  $c\in \Hom(\mathcal{X},\mathcal{Y})$, and $d\in {\rm End}(\mathcal{Y})$. 

Let $\mathcal{P}_{\rm Siegel}$ be the Siegel parabolic of $\Sp(\mathcal{W})$, the stabilizer of $\mathcal{X}$. Thus $\mathcal{P}_{\rm Siegel}$ corresponds to the condition $c=0$ in the above coordinates. Let $\mathcal{P}_{\rm Siegel}=\mathcal{M}\mathcal{N}$ be the standard Levi decomposition, where $\mathcal{M}$ is the Levi subgroup consisting of elements in $\Sp(\mathcal{W})$ which preserve the decomposition $\mathcal{X}\oplus\mathcal{Y}$ and $\mathcal{N}$ is the unipotent radical of $\mathcal{P}_{\rm Siegel}$. In the above coordinates we have
\[
\mathcal{M}=\left\{m(a)=\begin{pmatrix}a & \\  & a^\vee \end{pmatrix}: a \in  \GL(\mathcal{X})\right\},\qquad \mathcal{N}=\left\{n(b)=\begin{pmatrix}1 & b\\  & 1 \end{pmatrix}: b \in  \Hom^+(\mathcal{Y},\mathcal{X}) \right\}.
\]
Here, $\Hom^+(\mathcal{Y},\mathcal{X})$ denotes the subspace of symmetric homomorphisms in $\Hom(\mathcal{Y},\mathcal{X})$, i.e those satisfying $\langle by_1, y_2 \rangle_{\mathcal{W}} = \langle by_2, y_1 \rangle_{\mathcal{W}}$ for all $y_1, y_2 \in \mathcal{Y}$, and $a^\vee\in\GL(\mathcal{Y})$ is uniquely determined by the relation $\langle ax,a^\vee y\rangle_{\mathcal{W}}=\langle x,y\rangle_{\mathcal{W}}$ for all $x\in\mathcal{X}$, $y\in\mathcal{Y}$.  If we use $\langle\,,\,\rangle_{\mathcal{W}}$ to identify $\mathcal{Y}$ with $\mathcal{X}^*$, then $\Hom^+(\mathcal{Y},\mathcal{X})$ corresponds to the symmetric maps in $\Hom(\mathcal{X}^*,\mathcal{X})$, and $a^\vee$ corresponds to ${}^t a^{-1} \in \GL(\mathcal{X}^*)$.

We now recall the embeddings of $\bG=\Oo(V)$ and $\bG'=\Sp(W)$ in $\Sp(\mathcal{W})$. The group $\bG$ acts on $\mathcal{W}$, sending $w\otimes v$ to $w\otimes (gv)$ and preserving the symplectic form $\langle\,,\,\rangle_{\mathcal{W}}$. This induces an embedding of $\bG$ into $\Sp(\mathcal{W})$, given by $g\mapsto m({\rm Id}_U\otimes g)$. Similarly, $\bG'$ acts on $\mathcal{W}$, via $w\otimes v\mapsto (g'w)\otimes v$, which again clearly preserves $\langle\,,\,\rangle_{\mathcal{W}}$. To describe the resulting embedding into $\Sp(\mathcal{W})$ explicitly, we first use the polarization $W=U\oplus U^*$, to write matrices in $\Sp(W)$ in the form $\left(\begin{smallmatrix}A&B\\C&D\end{smallmatrix}\right)$, where $A\in {\rm End}(U)$, $B\in\Hom(U^*,U)$, $C\in\Hom(U,U^*)$, $D\in {\rm End}(U^*)$. Thus, $\bG'\hookrightarrow\Sp(\mathcal{W})$ is given by
\[
\begin{pmatrix}A & B\\ C & D \end{pmatrix}\mapsto \begin{pmatrix}  A \otimes {\rm Id}_V&  B\otimes {\rm Id}_V \\  C\otimes {\rm Id}_V  &  D\otimes {\rm Id}_V \end{pmatrix}.
\]
Then $\bG$ and $\bG'$ form a Type I algebraic reductive dual pair in $\Sp(\mathcal{W})$: for any extension $K/k$ the groups $\bG(K)$ and $\bG'(K)$ are mutual commutants in $\Sp(\mathcal{W})(K)$, in the sense that
\[
{\rm Cent}_{\Sp(\mathcal{W})(K)}(\bG(K))=\bG'(K)\quad\textrm{and}\quad {\rm Cent}_{\Sp(\mathcal{W})(K)}(\bG'(K))=\bG(K).
\]

\subsection{Local theta correspondence}\label{sec:Schroedinger}

In this section, we shall take the field $k$ to be either $\R$ or a finite extension of $\Q_p$. Throughout we shall write $G=\bG(k)={\rm O}(V)(k)$ and $G'=\bG'(k)=\Sp(W)(k)$. 

Let $\widetilde{\Sp}(\mathcal{W})_k$ denote the metaplectic group\footnote{As is well-known, we may write
\[
\widetilde{\Sp}(\mathcal{W})_k=\widehat{\Sp}(\mathcal{W})_k\times_{\mu_2} S^1,
\]
where $\widehat{\Sp}(\mathcal{W})_k$ is the unique non-split 2-fold central extension of $\Sp(\mathcal{W})_k$ and $\mu_2=\{\pm 1\}$. We prefer to work with $\widetilde{\Sp}(\mathcal{W})_k$ rather than $\widehat{\Sp}(\mathcal{W})_k$, since we shall work with subgroups of $\Sp(\mathcal{W})_k$ which split over the former without splitting over the latter; see \cite[II.9, Remarque]{MVW}. For our purposes, we could also use the degree $8$ cover $\widehat{\Sp}(\mathcal{W})_k\times_{\mu_2} \mu_8$.}, defined as the unique non-split central extension of $\Sp(\mathcal{W})(k)$ by $S^1$; thus
\[
1\rightarrow S^1\rightarrow \widetilde{\Sp}(\mathcal{W})_k\rightarrow\Sp(\mathcal{W})(k)\rightarrow 1.
\]
We may write $\widetilde{\Sp}(\mathcal{W})_k$ as the product $\Sp(\mathcal{W})(k)\times  S^1$ with the group multiplication given by a certain cocycle in $H^2(\Sp(\mathcal{W})(k), S^1)$, explicated by means of the Leray invariant by Perrin \cite{Perrin} and Rao \cite{RR}. This cocycle (but not its class) depends on the choice of Lagrangian subspace $\mathcal{X}$ of $\mathcal{W}$, as well as the choice of a non-trivial additive character of $k$, that we fix below.

Fix a non-trivial additive character $\psi$ of $k$ as follows. If $k$ is an extension of $\Q_p$, let $\psi_0$ denote the standard unramified character $x\mapsto e^{-2\pi ix}$ of $\Q_p$, and let $\psi$ be the pullback of $\psi_0$ by the trace map. If $k=\R$, we shall take $\psi(x)=e^{2\pi ix}$. 

Let ${\rm Heis}(\mathcal{W},k)$ denote the Heisenberg group of $\mathcal{W}$. When $k$ is non-archimedean, we let $\rho=\rho_\psi$ denote the unique, up to isomorphism, irreducible smooth representation of ${\rm Heis}(\mathcal{W},k)$ with central character $\psi$. Let $\omega=\omega_\psi$ be the corresponding Weil representation of $\widetilde{\Sp}(\mathcal{W})_k$, acting on the space of $\rho$. The latter can be taken to be the space $\mathscr{S}(\mathcal{Y})$ of Schwartz--Bruhat functions on $\mathcal{Y}$ (the isomorphism class of $\rho$ and of $\omega$ is independent of the choice of the Langrangian subspace $\mathcal{Y}$). Similarly, for $k=\R$, we let $\rho^{\rm Hilb}=\rho_\psi^{\rm Hilb}$ denote the unique, up to isomorphism, irreducible unitary representation of ${\rm Heis}(\mathcal{W},\R)$ with central character $\psi$. Then the corresponding Weil representation $\omega^{\rm Hilb}=\omega_\psi^{\rm Hilb}$ of $\widetilde{\Sp}(\mathcal{W})_\R$ acts on $L^2(\mathcal{Y})$. We shall often write simply $\omega^{\rm Hilb}$ for the space on which $\omega^{\rm Hilb}$ acts.

For the archimedean theory, we prefer to work with the \textit{algebrization} of $\omega^{\rm Hilb}$.  To describe this, we first specify a maximal compact subgroup of $\Sp(\mathcal{W})(\R)$.  Let $J_W \in G'$ be a positive complex structure, that is, an element satisfying $J_W^2= - \text{Id}_W$ and $\langle J_W w, w \rangle > 0$ for all nonzero $w \in W$.  We further assume that $J_W$ preserves the polarization $W = U \oplus U^*$, in the sense that $J_W U = U^*$ and $J_W U^* = U$.  Choose a decomposition $V = V_+ \oplus V_-$ of $V$ into positive and negative definite subspaces.  If we define $J \in \Sp(\mathcal{W})(\R)$ by
\[
J = J_W \otimes \left(\begin{smallmatrix} \text{Id}_{V_+} & 0 \\ 0 & -\text{Id}_{V_-} \end{smallmatrix}\right),
\]
then $J$ is a positive complex structure on $\mathcal{W}$ that preserves the polarization $\mathcal{W}=\mathcal{X}\oplus\mathcal{Y}$.  If we define $\mathcal{K}$ to be the subgroup of $\Sp(\mathcal{W})(\R)$ commuting with $J$, then $\mathcal{K}$ is a maximal compact subgroup.  In fact, if we view $\mathcal{W}$ as a complex vector space via $J$, then $\mathcal{K}$ is equal to the unitary group of the positive definite Hermitian form $(w_1, w_2) \mapsto \langle Jw_1, w_2 \rangle + i \langle w_1, w_2 \rangle$.  $\mathcal{K}$ has the property that $\mathcal{K} \cap G$ and $\mathcal{K} \cap G'$ are maximal compact subgroups of $G$ and $G'$, equal to ${\rm O}(V_+) \times {\rm O}(V_-)$, and the centralizer of $J_W$, respectively.

If we let $\widetilde{\mathcal{K}}$ denote the preimage of $\mathcal{K}$ in $\widetilde{\Sp}(\mathcal{W})_\R$, then $\widetilde{\mathcal{K}}$ is a maximal compact subgroup of $\widetilde{\Sp}(\mathcal{W})_\R$.  We let $\mathfrak{sp}_{\dim\mathcal{W}}$ denote the (real) Lie algebra of $\Sp(\mathcal{W})(\R)$, so that the Lie algebra of $\widetilde{\Sp}(\mathcal{W})_\R$ is $\mathfrak{sp}_{\dim\mathcal{W}} \oplus \R$.  We write $\omega$ for the Harish-Chandra module associated with $\omega^{\rm Hilb}$, which is the admissible $(\mathfrak{sp}_{\dim\mathcal{W}} \oplus \R,\widetilde{\mathcal{K}})$-module consisting of the $\widetilde{\mathcal{K}}$-finite vectors in $\omega^{\rm Hilb}$. The space of $\omega$ can be identified with the space of all products of complex valued polynomials on $\mathcal{Y}$ with the Gaussian $\mathcal{G}(y)=e^{-(\pi/2) \langle Jy, y \rangle}$ (see e.g. \cite[p. 388]{HST} for the constant appearing in front of $\langle Jy, y \rangle$), which we denote $\mathscr{S}_{\rm alg}(\mathcal{Y})$, and is dense in $\omega^{\rm Hilb}$.

Besides $\mathscr{S}_{\rm alg}(\mathcal{Y})$, there is another realization of $\omega$, called the Fock model, which we denote by $\mathcal{F}$.  To describe it, we consider $\mathcal{W}$ as a complex vector space using the complex structure $J$, and denote this space by $\mathcal{W}_J$.  The space of $\mathcal{F}$ is then the holomorphic polynomials on $\mathcal{W}_J^*$.  We may describe the action of $\widetilde{\mathcal{K}}$ in $\mathcal{F}$ by choosing a function $\det^{1/2}$ on $\widetilde{\mathcal{K}}$ whose square is the pullback of the determinant on $\mathcal{K} = {\rm U}( \mathcal{W}_J)$.  Then, if $\widetilde{u} \in \widetilde{\mathcal{K}}$ projects to $u \in \mathcal{K}$, the action of $\widetilde{u}$ in $\mathcal{F}$ is given by $\widetilde{u} f(z) = \det^{1/2}( \widetilde{u}) f( {}^t u z)$.

\begin{remark}\label{rem:dense}
We shall also need to consider the subspace $\omega^{\rm sm}$ of smooth vectors in $\omega^{\rm Hilb}$, realized as the space of Schwartz functions $\mathscr{S}(\mathcal{Y})$, even in the archimedean setting  (see the proof of Corollary \ref{distinction}). When endowed with its smooth (Fr\'echet) topology, $\omega^{\rm sm}$ is an irreducible admissible representation of $\widetilde{\Sp}(\mathcal{W})_\R$. It is a classical fact, proved using the G\"arding subspace, that $\omega^{\rm sm}$ is dense in $\omega^{\rm Hilb}$.

Now, from the chain of inclusions $\omega\subset\omega^{\rm sm}\subset\omega^{\rm Hilb}$, and the density of $\omega$ in $\omega^{\rm Hilb}$, it follows that $\mathscr{S}_{\rm alg}(\mathcal{Y})$ is dense in $\mathscr{S}(\mathcal{Y})$, \textit{when the latter is given the subset topology induced by that of $L^2(\mathcal{Y})$}. We will prove in Section \ref{sec:lemma-schwartz-space} that $\mathscr{S}_{\rm alg}(\mathcal{Y})$ is in fact dense in $\mathscr{S}(\mathcal{Y})$, with its natural Schwartz topology.
\end{remark}

Let $\widetilde{G}$ and $\widetilde{G}'$ denote the respective inverse images in $\widetilde{\Sp}(\mathcal{W})_k$ of $G$ and $G'$ under the surjective homomorphism $\widetilde{\Sp}(\mathcal{W})_k\rightarrow\Sp(\mathcal{W})(k)$. It follows from \cite[Lemme II.5]{MVW} that $\widetilde{G}$ and $\widetilde{G}'$ are mutual commutants inside $\widetilde{\Sp}(\mathcal{W})_k$. For simplicity, 
\begin{equation}\label{parity-assumption}
\textit{we shall assume that $\dim V$ is even}, 
\end{equation}
in which case $\widetilde{\Sp}(\mathcal{W})_k$ splits over $GG'$ \cite[Chapitre 3]{MVW}. We shall take the explicit section $GG'\rightarrow \widetilde{GG'}=\widetilde{G}\widetilde{G}'$ of Kudla \cite[Theorem 3.1, case 1]{Kudla}, which trivializes the restriction of the Perrin--Rao cocycle to $GG'$. Composing with the surjective multiplication map $G\times G'\rightarrow GG'$ we obtain
\[
\iota: G\times G'\rightarrow \widetilde{G}\widetilde{G}'\subset\widetilde{\Sp}(\mathcal{W})_k.
\]
For $k$ non-archimedean, we shall examine the restriction of $\omega$ to the product $\widetilde{G}\widetilde{G}'$, and then pull it back to $G\times G'$ via $\iota$.

When $k=\R$, we shall rather look at the infinitesimal version of this restriction. For this, we let $\mathfrak{g},\mathfrak{g}'\subset\mathfrak{sp}_{\dim\mathcal{W}}$ denote the (real) Lie algebras of $G$ and $G'$, respectively.  Pullback by $\iota$ then gives $\omega$ the structure of a $(\g \oplus \g', K \times K')$-module.

We are now in a position to recall the local theta correspondence. In the next two paragraphs, in order to put the real and non-archimedean theory on equal footing, we shall make the following conventions: by a \textit{representation} we mean a smooth  admissible representation in the archimedean setting and an admissible $(\mathfrak{g},K)$-module in the real setting. Here, $K$ is a choice of a maximal compact subgroup. In both cases we shall write ${\rm Hom}_G(\pi,\sigma)$ for the space of $G$ or $(\mathfrak{g},K)$-module homomorphisms. Finally, tensor products are algebraic tensor products for $(\mathfrak{g},K)$-modules.

With these conventions in place, let $\pi$ be an irreducible representation of $G$. We say that $\pi$ \textit{occurs in $\omega$} if ${\rm Hom}_G(\omega,\pi)\neq 0$. Given such a $\pi$, we set $\Gamma(\pi)=\omega/\omega[\pi]$, where
\[
\omega[\pi]=\bigcap_{f\in {\rm Hom}_G(\omega, \pi)} \ker (f).
\] 
Then $\Gamma[\pi]$ is a representation of $G \times G'$, equal to the maximal quotient of $\omega$ on which $G$ acts as a multiple of $\pi$ (the maximal $\pi$-isotypic quotient of $\omega$). In fact, $\Gamma[\pi]$ is of the form $\pi\otimes\Theta(\pi,W)$, where $\Theta(\pi,W)$ is a non-zero representation of $G'$; see \cite[Ch. 2, Lemma III.4]{MVW} for $k$ archimedean and \cite[(2.5)]{Howe1989} for $k=\R$.

By theorems of Howe \cite{Howe1989} (for $k=\R$), Waldspurger \cite{Walds} (for $k$ non-archimedian of odd residual characteristic), and Gan--Takeda \cite{GanTakeda} (for arbitrary $k$ non-archimedean), $\Theta(\pi,W)$ admits a unique irreducible quotient -- a representation of $G'$ which we shall denote by $\theta(\pi,W)$. We call $\theta(\pi,W)$ the \textit{theta lift} of $\pi$. For $\pi$ not appearing in $\omega$ we define its theta lift $\theta(\pi,W)$ to be zero.

The above facts apply equally well with the roles of $G$ and $G'$, and $\pi$ and $\pi'$, reversed; this sets up a bijection between irreducible $\pi$ on $G$ and $\pi'$ on $G'$ appearing in $\omega$. In that case, we shall write $\theta(\pi',V)$ for the theta lift of $\pi'$ to $G=\Oo(V)(k)$. For example, the statement of Lemma \ref{lemma:Theta-map} uses the theta lifting in this direction (from $G'$ to $G$).

\subsection{Local unramified theta correspondence}\label{sec:unramified}

We recall the local unramified theta correspondence.

Let $k$ be a finite extension of $\Q_p$.  We assume that $k/\Q_p$ is unramified, so that the character $\psi$ is also unramified.  Let $V$ be a non-degenerate orthogonal space over $k$ of even dimension $d$. We say that $V$ is \textit{unramified} if it admits a self-dual lattice.  There are two equivalence classes of such $V$:

\begin{enumerate}
\item $V=V_{r,r}$ is the split orthogonal space, or
\item $V=V_0+V_{r,r}$, where $V_0$ is an unramified quadratic extension $K/k$ equipped with the norm form.
\end{enumerate}
We extend the definition of $V_0$ to case (1), by setting $V_0 = 0$.  We define the character $\chi_V$ of $k^\times$ to be trivial in case (1) and the character of $K/k$ in case (2).  We assume in this paragraph that $V$ unramified and again write $G=\bG(k)=\Oo(V)(k)$. Let $L\subset V$ be a self-dual lattice and denote by $K_0$ its stabilizer in $G$. Then $K_0$ is a maximal compact hyperspecial subgroup of $G$. We say a smooth irreducible representation $\pi$ of $G$ is \textit{unramified} if $\pi^{K_0}\neq 0$.

Now let $W$ be a symplectic space over $k$. Then $W$ is always unramified. As before we put $G'=\bG'(k)=\Sp(W)(k)$. Let $L'$ be a self-dual lattice in $W$ and write $K_0'$ for its stabilizer in $G'$. Then $K_0'$ is a maximal compact hyperspecial subgroup of $G'$. We call a smooth irreducible representation $\pi'$ of $G'$ \textit{unramified} if $\pi'^{K_0'}\neq 0$.


\begin{theorem}[Howe \cite{Howe}, Theorem 7.1.(b)]\label{Howe-unram}
With hypotheses and notations as above, let $\pi$ be a smooth irreducible representation of $G$ and let $\theta(\pi,W)$ be its theta lift to $G'$. If $\theta(\pi,W) \neq 0$, then $\pi$ is unramified if and only if $\theta(\pi,W)$ is unramified.  The same is true for the lift from $G'$ to $G$.
\end{theorem}

In fact we may be more precise, by recalling the theta correspondence for unramified representations from \cite{Rallis}\footnote{This is commonly quoted from Kudla's notes \cite[Proposition 3.2]{KudlaCastle}.  We take this opportunity to correct a typographical error in the statement of this proposition.  Namely, in part (ii), the equation $\theta(\lambda) = (\chi_V\lambda_1,\ldots ,\chi_V\lambda_n, |\cdot |^{\frac{m}{2}-n},\ldots ,|\cdot |^{\frac12 \dim V_0})$ should read $\theta(\lambda) = (\chi_V\lambda_1,\ldots ,\chi_V\lambda_n, |\cdot |^{\frac{m}{2}-n-1},\ldots ,|\cdot |^{\frac12 \dim V_0})$.}. For this, we introduce the following notation. Let $P$ be a minimal parabolic subgroup of $G$, which is defined to be the stabilizer in $G$ of a full isotropic flag in $V$; see \cite[Remark 4.1]{Rallis} for the precise definition.  We assume that this flag is compatible with the lattice $L$ as in \cite[VI.3]{KudlaCastle}.  Similarly, we denote by $P'$ a minimal parabolic for $G'$ compatible with $L'$.

The Levi of $P$ is given by $M = {\rm GL}_1^r \times \Oo(V_0)$.  If $\lambda = (\lambda_1,\ldots ,\lambda_r)$ is an $r$-tuple of unramified characters of $k^\times$, we may think of $\lambda$ as a character of $M$ by extending it to be trivial on $\Oo(V_0)$.  We may then define $\pi(\lambda)$ to be the unique smooth irreducible unramified representation of $G$ appearing as a subquotient of ${\rm Ind}_P^G \lambda$.  Likewise, if $\lambda' = (\lambda_1',\ldots ,\lambda_m')$ is an $m$-tuple of unramified characters of $k^\times$, we may define $\pi'(\lambda')$ to be the unique smooth irreducible unramified representation of $G'$ appearing as a subquotient of ${\rm Ind}_{P'}^{G'} \lambda'$.

\begin{theorem}[Rallis \cite{Rallis}, Remark 4.4]
\label{thm-Rallis} Let $r$ and $d$ be the rank and dimension of the unramified quadratic space $V$. Let $\chi_V$ and $V_0$ be defined as above. Let $2m$ be the dimension of the symplectic space $W$. Assume $r\geqslant m$. 

Let $\lambda' = (\lambda_1',\ldots ,\lambda_m')$ be an $m$-tuple of unramified characters of $k^\times$. Then $\theta(\pi'(\lambda'),V)$ is non-zero, and equal to $\pi( \lambda)$ where
\[
\lambda = (\chi_V\lambda_1,\ldots ,\chi_V\lambda_m, |\cdot |^{\frac{d}{2}-m-1},\ldots ,|\cdot |^{\frac12 \dim V_0}).
\]
In particular, if $d>2m+2$ then $\theta(\pi'(\lambda'),V)$ is not tempered.
\end{theorem} 

While we do not use this result in the proof of Theorem \ref{sup-thm}, it provides the justification for the remark about non-temperedness made after the statement of the theorem in Section \ref{sec:exceptional}.

\subsection{Tempered spherical transfer from archimedean $\Oo(n,m)$ to $\Sp_{2m}(\R)$}
\label{sec:arch-lift}

We now let $k = \R$.  Let $n\geqslant m\geqslant 1$ be integers with $n+m$ even, and let $1\leqslant m'\leqslant m$.  We let $V$ have signature $(n,m)$, and let $W = W_{2m'}$ have dimension $2m'$.  We make the standard choices of coordinates on $V$ and $W_{2m'}$, so that $G$ and $G'$ are identified with ${\rm O}(n,m)$ and $\Sp_{2m'}(\R)$ respectively, with standard maximal compacts $\Oo(n) \times \Oo(m)$ and ${\rm U}(m')$.

Let $\pi$ be an irreducible spherical representation of ${\rm O}(n,m)$, and let $\theta(\pi,W_{2m'})$ be its theta lift to $\Sp_{2m'}(\R)$.  We cannot deduce from Theorem \ref{Howe-unram} (which applies to non-archimedean local fields, and unramified local data) that $\theta(\pi,W_{2m'})$ is itself spherical. Nevertheless, $\theta(\pi,W_{2m'})$, when non-zero, is not far from spherical, in that it admits an abelian ${\rm U}(m')$-type. 

\begin{lemma}\label{lem-K-type} Let $\pi$ be an irreducible spherical representation of ${\rm O}(n,m)$, where $n\geqslant m\geqslant 1$. Let $1\leqslant m'\leqslant m$ and suppose that $\theta(\pi,W_{2m'})\neq 0$. Then $\theta(\pi,W_{2m'})$ admits $\det^{(n-m)/2}$ as a ${\rm U}(m')$-type.
\end{lemma}

\begin{proof}
For brevity, during the proof we will denote the maximal compact subgroups of ${\rm O}(n,m)$ and $\Sp_{2m'}(\R)$ by $K$ and $K'$.  We shall work in the Fock model $\mathcal{F}$ of the oscillator representation.  For an integer $d\geq 0$, let $\mathcal{F}^d$ denote the subspace of polynomials of degree equal to $d$, a stable subspace under the commuting action of $K$ and $K'$.  Let $\mathcal{H}\subset\mathcal{F}$ denote the space of joint harmonics for both $K$ and $K'$. Irreducible representations $\rho$ of $K$ and $\rho'$ of $K'$ are said to \textit{correspond} if $\rho\otimes\rho'$ is a direct summand of the representation of $K\times K'$ on $\mathcal{H}$. This correspondence is a bijection between those irreducible representations of $K$ and $K'$ which appear in $\mathcal{H}$ \cite{Howe1989}. For a finite dimensional representation $\sigma$ of $K$ or $K'$, which occurs in the Fock space, its \textit{degree} is the minimal $d$ for which the $\sigma$-isotypic subspace of $\mathcal{F}^d$ is non-zero.

The trivial representation of $K$ and the character $\det^{(n-m)/2}$ of $K'$ correspond to each other, in the above sense. Indeed, this can be extracted from \cite[Proposition 2.10 (i)]{Paul}, upon recalling that, under the standard parametrization of the irreducible finite dimensional representations of $K$ and $K'$ by their highest weight, the trivial representation of ${\rm O}(n)\times {\rm O}(m)$ has parameter
\[
(\underbrace{0,\ldots ,0}_{n\; \text{times}};1)\otimes (\underbrace{0,\ldots ,0}_{m\; \text{times}};1).
\]
and the character $\det^{(n-m)/2}$ of ${\rm U}(m')$ has parameter  
\[
\underbrace{\left((n-m)/2,\ldots ,(n-m)/2\right)}_{m'\; \text{times}}.
\]



Now the spherical representation $\pi$ of $\Oo(n,m)$ has the trivial representation as a $K$-type; as the degree of the latter is 0 (see \cite[(2.13)]{Paul}), it is necessarily a $K$-type of minimal degree for $\pi$. 
Now a theorem of Howe \cite{Howe1989} states that the $K'$-types which correspond with the $K$-types of minimal degree in $\pi$ must themselves appear in $\theta(\pi,W_{2m'})$. 
\end{proof}

Motivated by Lemma \ref{lem-K-type}, we now recall the classification of irreducible representations of $\Sp_{2m'}(\R)$ that admit abelian ${\rm U}(m')$-types.  Let $\gh'_\C = \ga'_\C$ be the Cartan subalgebra of $\mathfrak{sp}_{2m'}$ defined in Section \ref{2examples}, and let $W'$ be the Weyl group.  Let $P'$ be a Borel subgroup of ${\rm Sp}_{2m'}(\R)$, with Langlands decomposition $P' = N' A' M'$.  We note that $M'$ is finite, and isomorphic to $(\Z / 2)^{m'}$.  If $(\pi',V')$ is a representation of $\Sp_{2m'}(\R)$, and $v_\tau \in V'$ is a nonzero vector of abelian ${\rm U}(m')$-type $\tau^{-1}$, then the algebra $\mathbb{D}(\tau)$ defined in Section \ref{sec:tau-spherical} acts on $v_\tau$ by scalars, and this defines an element $\lambda \in \ga'^*_\C / W'$ that we call the spectral parameter of $\pi'$.

\begin{lemma}
\label{tau-spherical-class}

For any character $\tau$ of ${\rm U}(m')$ and $\lambda \in \ga'^*_\C / W'$, there is a unique irreducible representation $\pi'_{\lambda, \tau}$ of ${\rm Sp}_{2m'}(\R)$ with spectral parameter $\lambda$ and containing $\tau^{-1}$ as a ${\rm U}(m')$-type.  Moreover, the infinitesimal character of $\pi'_{\lambda, \tau}$ is also equal to $\lambda$.

\end{lemma}

\begin{proof}

The uniqueness of $\pi'_{\lambda, \tau}$ may be proved as in Proposition \ref{O-spherical}.  For existence, $\pi'_{\lambda, \tau}$ may be realized as the unique irreducible subquotient of
\[
{\rm Ind}_{P'}^{{\rm Sp}_{2m'}(\R)}( e^\lambda \otimes \tau^{-1}|_{M'} )
\]
containing $\tau^{-1}$ as a ${\rm U}(m')$-type.

The assertion about the infinitesimal character follows as in the proof of \cite[Prop 2.1]{Helgason2}.  We briefly recall this argument, with the modifications needed to handle the $\tau$-spherical case considered here.  For the duration of this proof only, we will work with the $KAN$ form of the Iwasawa decomposition, and let $\kappa(g) \in {\rm U}(m')$ and $H(g) \in \ga'$ be uniquely determined by $g \in \kappa(g) e^{H(g)} N'$.  If we consider the function
\[
f(g) = e^{ (\lambda - \rho) H(g)} \tau^{-1}( \kappa(gk) ),
\]
then it follows from the formula for $\varphi_{\lambda, \tau}$ given on p. 332 of \cite{Shimeno}, and Proposition \ref{prop:sph-fn-2-vars}, that $\varphi_{\lambda, \tau}$ is obtained by averaging $f$ as
\[
\varphi_{\lambda, \tau}(g) = \int_{{\rm U}(m')} \tau(k) f(gk) dk.
\]
We let $\cZ(\g')$ be the center of the universal enveloping algebra of $\g'$, and let $\gamma_{\rm HC} : \cZ(\g') \to S( \ga')^{W'}$ be the Harish-Chandra homomorphism.  We may now apply the argument from \cite[Prop 2.1]{Helgason2} to show that $f$, and hence also $\varphi_{\lambda, \tau}$, behaves under $\cZ(\g')$ as $Z \varphi_{\lambda, \tau} = \gamma_{\rm HC}(Z)( \lambda) \varphi_{\lambda, \tau}$.  (This calculation is simplified by the fact that ${\rm Sp}_{2m'}(\R)$ is split.)  Equation \eqref{defn-sph-fn} implies that $\varphi_{\lambda, \tau}$ is a matrix coefficient of $\pi'_{\lambda, \tau}$, which completes the proof.
\end{proof}

To continue, we recall a result of Przebinda \cite{Prz}, which gives a relation between the infinitesimal characters of the archimedean theta correspondence.  Let $\gh$ be the Cartan subalgebra of $\mathfrak{o}(n,m)$ defined in Section \ref{sec:O-spherical-notn}, with subspace $\ga$, and let the Weyl groups $\widetilde{W}_\ga$ and $\widetilde{W}$ be as in that section.  We let $e_1, \ldots, e_{(n+m)/2}$ be the standard basis of $\gh^*_\C$, and $e_1', \ldots, e_{m'}'$ be the standard basis of $\gh'^*_\C$.  Since $m' \leqslant m$, we have an inclusion $\iota: \gh'^*_\C \to \gh^*_\C$ defined by $\iota(e_j')  = e_j$.  We define the element $\tau \in \gh^*_\C$ by
\be
\label{taudef}
\tau = \sum_{j = m'+1}^{(n+m)/2} ( \tfrac{n+m}{2} - j)e_j = ( \overbrace{0, \ldots, 0}^{m'},  \tfrac{n+m}{2} - m'-1, \ldots, 0 ).
\ee
Let $\pi'=\theta(\pi,W_{2m'})$.  We let $\gamma_\pi \in \gh^*_\C / \widetilde{W}$ be the infinitesimal character of $\pi$, and if $\pi' \neq 0$ then we let $\gamma_{\pi'} \in \gh'^*_\C / W'$ be its infinitesimal character.  If $\pi' \neq 0$, then \cite[Theorem 1.19]{Prz} states that
\begin{equation}\label{eq:inf-char}
\gamma_\pi=\iota(\gamma_{\pi'})+ \tau.
\end{equation}
We next show that the theta lift of a spherical representation $\pi$ vanishes if $m' < m$, under the assumption that $\pi$ is tempered.

\begin{lemma}\label{first-occ-sph}
Let $\pi$ be an irreducible spherical {\rm tempered} representation of ${\rm O}(n,m)$, where $n\geqslant m\geqslant 1$. For $m'<m$ we have $\theta(\pi,W_{2m'})=0$. 
\end{lemma}

\begin{proof}

Let $\lambda \in i \ga^* / \widetilde{W}_\ga$ be the spectral parameter of $\pi$, as defined in Section \ref{sec:O-spherical}.  By Proposition \ref{O-spherical}, we have
\be
\label{O-inf-char}
\gamma_\pi = \lambda + \rho_{\mathfrak{m}} = \lambda + ( \overbrace{0, \ldots, 0}^{m},  \tfrac{n-m}{2} -1, \ldots, 0 ).
\ee
If we had $\gamma_\pi=\iota(\gamma_{\pi'})+ \tau$ for some $\gamma_{\pi'} \in \gh'^*_\C$, then $\gamma_\pi$ would have an entry with real part equal to $\tfrac{n+m}{2} - m'-1$ in absolute value, but if $m' < m$ this is incompatible with \eqref{O-inf-char} and the condition $\lambda \in i \ga^*$.
\end{proof}

We now remove the tempered assumption on $\pi$, and examine $\theta(\pi,W_{2m})$ for $\pi$ spherical on ${\rm O}(n,m)$.  In the case under consideration when  $m' = m$, we recall that there is a natural identification of $\ga^*_\C / \widetilde{W}_\ga$ and $\ga'^*_\C / W'$.  Moreover, if we identify $\ga^*$ with a subspace of $\gh^*$ via the splitting $\gh = \ga + \gh_\gk$, we may choose our bases of $\gh^*_\C$ and $\gh'^*_\C$ so that $\iota( \ga'^*) = \ga^*$, and this induces the identification of $\ga^*_\C / \widetilde{W}_\ga$ and $\ga'^*_\C / W'$.

\begin{lemma}\label{lemma:1st-occ}
Let $n\geqslant m\geqslant 1$.  Let $\pi$ be the irreducible spherical representation of ${\rm O}(n,m)$ with spectral parameter $\lambda \in \ga^*_\C / \widetilde{W}_\ga$, as in Proposition \ref{O-spherical}.  We view $\lambda$ as an element of $\ga'^*_\C / W'$.  Then if $\theta(\pi,W_{2m}) \neq 0$, it is the unique irreducible subquotient of
\[
{\rm Ind}_{P'}^{{\rm Sp}_{2m}(\R)}( e^\lambda \otimes \det\! {}^{(n-m)/2}|_{M'} )
\]
containing $\det^{(n-m)/2}$ as a ${\rm U}(m)$-type.
\end{lemma}

\begin{proof}
We know from Lemma \ref{lem-K-type} that $\theta(\pi,W_{2m})$ admits the character $\det^{(n-m)/2}$ as a ${\rm U}(m)$-type.  Moreover, because $m' = m$, the definition \eqref{taudef} of $\tau$ implies that $\tau = \rho_{\mathfrak{m}}$.  It then follows from the relations \eqref{eq:inf-char} and \eqref{O-inf-char} that $\gamma_{\pi'} = \lambda \in \gh'^*_\C / W'$.  A result by Zhu \cite{Zhu} states that an irreducible admissible representation of a classical real group such as $\Sp_{2m}(\R)$, which admits an abelian $K$-type (where $K$ is a maximal compact subgroup), is uniquely determined by its infinitesimal character.
\end{proof}

We finish this section by describing the lift of the trivial representation $1_n$ from ${\rm O}(n)$ to ${\rm Sp}_{2m'}(\R)$, where $n$ is even and $n \geqslant m'$.

\begin{lemma}
\label{compact-Arch-lift}
Let $n$ be even, with $n \geqslant m' \geqslant 1$.  Let $1_n$ be the trivial representation of ${\rm O}(n)$. Then if $\theta(1_n,W_{2m'}) \neq 0$, it is the unique irreducible subquotient of
\[
{\rm Ind}_{P'}^{{\rm Sp}_{2m'}(\R)}( e^\lambda \otimes \det\! {}^{n/2}|_{M'} )
\]
containing $\det^{n/2}$ as a ${\rm U}(m')$-type, where $\lambda = (\tfrac{n}{2} - 1, \ldots, \tfrac{n}{2} - m') \in \ga'^*_\C$.
\end{lemma}

\begin{proof}

Assume that $\theta(1_n,W_{2m'}) \neq 0$.  We may show as in Lemma \ref{lem-K-type} that $\theta(1_n,W_{2m'})$ admits $\det^{n/2}$ as a ${\rm U}(m')$-type.  Moreover, the infinitesimal character of $1_n$ is equal to $(\tfrac{n}{2} - 1, \tfrac{n}{2} - 2, \ldots, 0)$, and combining this with \eqref{eq:inf-char} shows that the infinitesimal character of $\theta(1_n,W_{2m'})$ is $(\tfrac{n}{2} - 1, \ldots, \tfrac{n}{2} - m')$.  The result now follows as in Lemma \ref{lemma:1st-occ}.

\end{proof}

\subsection{Review of the global theta correspondence}\label{sec:global-theta}

We return to the setting of Section \ref{Type1} and now let $k=F$ be a totally real number field with ring of integers $\cO_F$.

\subsubsection{Adelic structures}\label{sec:adelic-structures}
We begin by fixing the restricted tensor product structure of the adelic points of $\Sp(\mathcal{W})$, where $\mathcal{W}=W\otimes V$. Let $S_V$ denote the set of finite places $v$ of $F$ at which $V$ is ramified together with all archimedean places. Fix a lattice $L$ in $V$ which is self-dual at all $v\notin S_V$. Similarly, let $L'$ be a lattice in $W$ which is self-dual at all finite places. Then the tensor product $\mathcal{L}=L'\otimes L$ is a lattice in $\mathcal{W}$. For each $v\notin S_V$ let $L_{v}=L\otimes \cO_{F_v}$ and $L_{v}'=L'\otimes \cO_{F_v}$ be the localizations at $v$ and write $\mathcal{K}_{0,v}$ for the stabilizer of $\mathcal{L}_{v}=L_{v}'\otimes L_{v}$ in $\Sp(\mathcal{W})(F_v)$. Then $\mathcal{K}_{0,v}$ is a hyperspecial subgroup of $\Sp(\mathcal{W})(F_v)$. We define $\Sp(\mathcal{W})(\A)$ as the restricted tensor product with respect to $\{\mathcal{K}_{0,v}\}_{v\notin S_V}$. 

For $v\notin S_V$, the local metaplectic group splits over the maximal compact subgroup $\mathcal{K}_{0,v}$ (this fact is independent of the parity assumption \eqref{parity-assumption} on $V$). We may therefore identify the latter with a subgroup of $\widetilde{\Sp}(\mathcal{W})_{F_v}$ and form the restricted tensor product $\prod_v  \widetilde{\Sp}(\mathcal{W})_{F_v}$ with respect to the system $\{\mathcal{K}_{0,v}\}_{v\notin S_V}$. Unlike its symplectic counterpart, the adelic metaplectic group is \textit{not} equal to $\prod_v  \widetilde{\Sp}(\mathcal{W})_{F_v}$ \cite[\S 3]{Howe}. There is, however, a unique non-split $S^1$-central extension of $\Sp(\mathcal{W})(\A)$ which is a surjective image of $\prod_v  \widetilde{\Sp}(\mathcal{W})_{F_v}$; we shall denote this extension by $\widetilde{\Sp}(\mathcal{W})_\A$.  Similarly to the local case in \S \ref{sec:Schroedinger}, we may write $\widetilde{\Sp}(\mathcal{W})_\A = \Sp(\mathcal{W}) \times S^1$, with the multiplication given by the Perrin--Rao cocycle associated with the Lagrangian subspace $\mathcal{X}(\A)$ of $\mathcal{W}(\A)$ and the additive character $\psi$.

The choice of $\mathcal{K}_{0,v}$ was made to facilitate the description of restricted tensor product structure of the adelic points of the subgroups $\bG={\rm O}(V)$ and $\bG'=\Sp(W)$. Indeed, for every $v\notin S_V$ let $K_{0,v}$ and $K_{0,v}'$ denote the stabilizers of $L_{v}$ and $L_{v}'$, respectively; both are subgroups of $\mathcal{K}_{0,v}$. Then $\bG(\A)$ and $\bG'(\A)$ are the restricted tensor products taken with respect to the system of hyperspecial subgroups $\{K_{0,v}\}_{v\notin S_V}$ and $\{K_{0,v}'\}_{v\notin S_V}$. Let $\widetilde{\bG}_\A$ and $\widetilde{\bG}'_\A$ denote the inverse images of $\bG(\A)$ and $\bG'(\A)$ in $\widetilde{\Sp}(\mathcal{W})_\A$. It follows from the above description of $\widetilde{\Sp}(\mathcal{W})_\A$, as well as the corresponding local properties, that $\widetilde{\bG}_\A$ and $\widetilde{\bG}'_\A$ commute with each other. Furthermore, recalling our parity assumption \eqref{parity-assumption}, the covering $\widetilde{\Sp}(\mathcal{W})_\A\rightarrow\Sp(\mathcal{W})(\A)$ splits over $\bG(\A)\bG'(\A)$, and we may view both $\bG(\A)$ and $\bG'(\A)$ as subgroups of $\widetilde{\Sp}(\mathcal{W})_\A$.

\subsubsection{Real structures}

For each real place $v$, we choose a positive complex structure $J_{W_v}$ on $W_v$.  As in Section \ref{sec:Schroedinger}, when combined with a choice of decomposition of $V_v$ into positive and negative subspaces, this determines a positive complex structure $J_v$ on $\mathcal{W}_v$.  We use these data to define maximal compact subgroups of $\bG$, $\bG'$, and $\Sp(\mathcal{W})$ at the real places.

\subsubsection{The Weil representation and theta lifts}

Let $\psi = \otimes \psi_v$ be the product of the local characters defined in Section \ref{sec:Schroedinger}, which gives a character of $\A / F$.  For almost all non-archimedean places $v$ the local additive character $\psi_v$ is unramified, in the sense of Section \ref{sec:Schroedinger}.  The choice of $\psi$ defines a Weil representation $\omega_v=\omega_{v,\psi_v}$ for each local metaplectic group $\widetilde{\Sp}_{F_v}$.  There is a global Weil representation $\omega$ of $\widetilde{\Sp}(\mathcal{W})_\A$, with the property that the pullback of $\omega$ to $\prod_v  \widetilde{\Sp}(\mathcal{W})_{F_v}$ is the restricted tensor product of the $\omega_v$.  As before, we have corresponding smooth and unitary representations $\omega^\text{sm}$ and $\omega^\text{Hilb}$.

Let $\mathcal{X}\oplus\mathcal{Y}$ be a complete polarization of $\mathcal{W}$. An explicit model for $\omega^\text{sm}$ (the \textit{Schroedinger model}) is realized on $\mathscr{S}(\mathcal{Y}(\A))$, the space of Schwartz--Bruhat functions on $\mathcal{Y}(\A)$.  We shall henceforth fix this model for $\omega^\text{sm}$, describing it in some detail shortly.  We likewise have a model for $\omega$ on $\mathscr{S}(\mathcal{Y}(\A_f))\otimes \mathscr{S}_{\rm alg}(\mathcal{Y}(F_\infty))$, where the algebraic Schwartz space $\mathscr{S}_{\rm alg}(\mathcal{Y}(F_\infty))$ is defined in \S\ref{sec:Schroedinger}.

A theorem of Weil \cite[\S 40]{Weil} states that there is a unique injective homomorphism $\iota_F: \Sp(\mathcal{W})(F)\rightarrow \widetilde{\Sp}(\mathcal{W})_\A$ lifting the diagonal embedding $\Sp(\mathcal{W})(F)\rightarrow\Sp(\mathcal{W})(\A)$. We will often identify $\Sp(\mathcal{W})(F)$ with its image in $\widetilde{\Sp}(\mathcal{W})_\A$ under the map $\iota_F$. Weil furthermore showed \cite[\S 41]{Weil} that the distribution 
\[
\Theta: \alpha\mapsto \sum_{y\in \mathcal{Y}(F)}\alpha(y),\qquad  \alpha\in\mathscr{S}(\mathcal{Y}(\A)),
\]
 is $\Sp(\mathcal{W})(F)$-invariant: $\Theta(\omega(\gamma)\alpha)=\Theta(\alpha)$ for all $\gamma\in \Sp(\mathcal{W})(F)$.  For $\alpha\in\mathscr{S}(\mathcal{Y}(\A))$ the function
$\Theta(\omega(\cdot)\alpha)$ is an $\Sp(\mathcal{W})(F)$-invariant function of moderate growth on $\widetilde{\Sp}(\mathcal{W})_\A$; see \cite[Th\'eor\`eme 6]{Weil} and \cite[\S 4]{Howe}. 

We shall be interested in the restriction of $\Theta(\omega(\cdot)\alpha)$ to $\bG(\A)\bG'(\A)$, when the latter is viewed as a subgroup of $\widetilde{\Sp}(\mathcal{W})_\A$ through the splitting $\iota_\A: \bG(\A)\bG'(\A)\rightarrow\widetilde{\bG}(\A)\widetilde{\bG}'(\A)=\widetilde{\bG\bG'}(\A)$ obtained by the restricted tensor product of all local splittings defined in \S\ref{sec:Schroedinger}. It is known (see the discussion in \cite[\S 2.2]{CM}) that the restrictions of $\iota_\A$ and $\iota_F$ to $\bG(F)\bG'(F)$ coincide. We denote
\[
\Theta(g,g';\alpha)=\sum_{y\in \mathcal{Y}(F)}(\omega(gg')\alpha)(y)
\]
for $(g,g')\in\bG(\A)\times\bG'(\A)$. For a function $\phi\in C^\infty([\bG])$ of rapid decrease we let
\begin{equation}\label{eq:Theta-fn}
\Theta (\phi,\alpha;W)(g')=\int_{[\bG]} \overline{\phi(g)} \Theta(g,g';\alpha)dg
\end{equation}
be the theta lift of $\phi$ to $C^\infty([\bG'])$.

\subsubsection{Explicating the theta kernel}\label{sec:thetakernel}
We now give a more explicit form of the theta kernel $\Theta(g,g';\alpha)$, for elements $(g,g')\in\bG(\A)\times\bP_{\rm Siegel}(\A)$. Here, similarly to the definition of $\mathcal{P}_{\rm Siegel}$ in Section \ref{Type1}, the Siegel parabolic $\bP_{\rm Siegel}$ in $\bG'=\Sp(W)$ is the stabilizer of $U\subset W$.

We begin by describing the action of the Siegel parabolic $\cP_\text{Siegel}$ of $\Sp(\mathcal{W})$ in the Schroedinger model.  With the notation from Section \ref{Type1}, $\mathcal{P}_{\rm Siegel}(\A)$ acts via the Weil representation $\omega^\text{sm}$ on functions $\alpha\in\mathscr{S}(\mathcal{Y}(\A))$ through the formulae
\begin{equation}\label{action0}
\begin{aligned}
\omega^\text{sm}(m(a),1)\alpha(y)&=|\det a|_\A^{1/2}\alpha({}^t a y),\;\;\qquad m(a)\in \mathcal{M}(\A),\\
\omega^\text{sm}(n(b),1)\alpha(y)&=\psi\big(\frac12\langle by, y\rangle_{\mathcal{W}}\big)\alpha(y),\quad n(b)\in \mathcal{N}(\A).
\end{aligned}
\end{equation}
Note that we are using the identification $\widetilde{\Sp}(\mathcal{W})_\A = \Sp(\mathcal{W}) \times S^1$ here.  

Restricting the first equation of \eqref{action0} to $\bG(\A)=\Oo(V)(\A)$, viewed as a subgroup of $\mathcal{M}(\A)$ via the embedding $g\mapsto m({\rm Id}_U\otimes g)$ from Section \ref{Type1}, we obtain
\[
\omega^\text{sm}(g,1)\alpha(x)=\alpha(g^{-1}x),\qquad g\in\bG(\A).
\]
Moreover, an arbitrary element of $\bP_{\rm Siegel}(\A)$ can be written as $m(a\otimes {\rm Id}_V)n(b\otimes {\rm Id}_V)$, where $a\in\GL(U)$ and $b\in {\rm sym}(U^*,U)$. We deduce from the second equation in \eqref{action0} that
\begin{equation}\label{eq:explicit-theta}
\Theta(g,m(a)n(b);\alpha)=|\det a|_\A^{1/2}\sum_{y\in \mathcal{Y}(F)}\psi\bigg(\frac12\langle (b\otimes {\rm Id}_V)y, y\rangle_{\mathcal{W}}\bigg) \alpha(g^{-1}\, {}^t a y),
\end{equation}
as a function on $\bG(\A)\times\bP_{\rm Siegel}(\A)$.

\subsection{Global representation theoretic properties}\label{sec:global-theta-properties}

Henceforth, all decompositions of an automorphic representation of $\bG={\rm O}(V)(\A)$ or $\bG'(\A)=\Sp(W)(\A)$ as a restricted tensor product will be taken with respect to the systems $\{K_{0,v}\}_{v\notin S_V}$ or $\{K_{0,v}'\}_{v<\infty}$ from Section \ref{sec:global-theta}, respectively.

As with the local archimedean theory, we shall define the global theta lift in the context of $K$-finite automorphic forms.  Given a cuspidal automorphic representation $\pi$ of $\bG(\A)$, the space $\Theta(\pi,W)$ generated by $\Theta(\phi,\alpha;W)$, for $\phi\in \pi$ and $\alpha\in\mathscr{S}(\mathcal{Y}(\A_f))\otimes \mathscr{S}_{\rm alg}(\mathcal{Y}(F_\infty))$, is an invariant subspace, possibly zero, of the space of automorphic forms on $\bG'(\A)$.

The automorphic representation $\Theta(\pi,W)$ is of finite length. Its relation to the local theta correspondence is as follows.  Suppose that $\pi'\simeq\otimes_v\pi'_v$ is an irreducible quotient of $\Theta(\pi,W)$.  Because the space $\{ \overline{\phi} : \phi \in \pi \}$ is isomorphic to the contragredient $\pi^\vee$ of $\pi$, we have a non-trivial intertwining operator of $\bG(\A) \times \bG'(\A)$-modules $\omega \to \pi' \times \pi$, so that $\pi_v' \simeq \theta( \pi_v, W)$; see \cite[Proposition 7.1.2]{KR94}. In fact, we have the following result of Kudla--Rallis \cite[Corollary 7.1.3]{KR94}:

\begin{theorem}[Kudla--Rallis]\label{global-irred}
Let $\pi$ be an irreducible cuspidal representation of $\bG(\A)$. If $\Theta(\pi,W)$ is square-integrable, then $\Theta(\pi,W)$ is irreducible and
\[
\Theta(\pi, W)\simeq\otimes_v \theta(\pi_v, W).
\]
\end{theorem}

Regarding the cuspidality of global theta lifts, we have the following theorem of Rallis, \cite[Theorem I.1.1]{Rallis84}, establishing the cuspidality of the ``first occurence'' of $\Theta(\pi,W)$. To formulate this principle properly, we need to add in the dimension $2m$ of the symplectic space $W$ to the notation, writing $W_{2m}$ in place of $W$ and $\Sp_{2m}=\Sp(W_{2m})$ in place of $\bG'$.

\begin{theorem}[Rallis]\label{thm:cuspidality} Let $\pi$ be an irreducible cuspidal automophic representation of $\bG(\A)$. Let $m\geqslant 2$. If $\Theta(\pi,W_{2m})$ is not contained in the space of cusp forms $L^2_{\rm cusp}({\rm Sp}_{2m}(F)\backslash {\rm Sp}_{2m}(\A))$ then $\Theta(\pi,W_{2(m-1)})\neq 0$.
\end{theorem}

Finally, the following lemma will be of use in the proof of Theorem \ref{sup-thm}.

\begin{lemma}\label{cuspidal-theta}
Let $v_0$ be a fixed archimedean place, and assume $V$ is of signature $n\geqslant m\geqslant 1$ at $v_0$. Let $\pi$ be an irreducible cuspidal automophic representation of $\bG(\A)$, whose local component at $v_0$ is spherical and tempered. Then the global lift $\Theta(\pi,W_{2m})$ is cuspidal (or zero) and $\Theta(\pi,W_{2m})\simeq\otimes_v \theta(\pi_v,W_{2m,v})$.
\end{lemma}

\begin{proof}
Since $\pi_{v_0}$ is spherical and tempered, it follows from Lemma \ref{first-occ-sph} that $\Theta(\pi_{v_0},W_{2m',v_0})$ vanishes for $m'<m$. The global lift $\Theta(\pi,W_{2m'})$ therefore vanishes for $m'<m$. The first statement of the lemma then follows from Theorem \ref{thm:cuspidality}. The factorization statement is a consequence of Theorem \ref{global-irred}. 
\end{proof}

\section{Period relation}\label{sec:Period-relation}

The goal of this section is to prove the adelic period relation Proposition \ref{period-relation} and to interpret it classically in Section \ref{adelic2classical}.

\subsection{The Maass period relation}\label{Maass-period-relation}

We begin by recalling a result of Rudnick--Sarnak \cite[\S 3.1]{RS} in which certain discrete orthogonal periods of a Hecke--Maass form on $\SO(3,1)$ are identified with the negative Fourier coefficients of a corresponding theta lift, a weight one Maass form on $\SL_2(\R)$. This extends the classical result of Maass \cite{Maass59} concerning weight zero CM Maass forms, which are theta lifts of characters $\chi$ on $\SO(1,1)$ (defined relative to a real quadratic extension of $\Q$ with norm form ${\rm N}_{E/\Q}$), whose Fourier coefficients are given by $\sum_{{\rm N}_{E/\Q}(\mathfrak{a})=n}\chi(\mathfrak{a})$. These results, along with the basic method of proof, will form the prototype of the period relation we establish in the next paragraph.

Let $(V,Q)$ be an anisotropic quadratic space over $\Q$, of dimension $4$ and signature $(3,1)$. Denote $\bG=\Oo(V)$. Let $L\subset V$ be an integral lattice and let $\Gamma$ be the subgroup of $\bG(\R)$ preserving $L$.  Let $D$ be the discriminant of $L$. As a model for hyperbolic $3$-space, we take $\mathbb{H}^3$ to be one of the sheets of the two-sheeted hyperboloid $\{x\in V(\R): Q(x)=-1 \}$, or equivalently the space of negative definite lines in $V(\R)$. Let $\bG(\R)^+$ be the neutral component of $\bG(\R)$ for the Hausdorff topology; then $\bG(\R)^+$ preserves $\mathbb{H}^3$.  Letting $K_\infty^+<\bG(\R)^+$ denote the stabilizer of a point in $\mathbb{H}^3$, we may isometrically identify the Riemannian symmetric space $\mathbb{H}^3$ with $\bG(\R)^+/K_\infty^+$. Let $f_\lambda$ be a weight zero Maass form of spectral parameter $\lambda$ on $\Gamma^+\backslash\mathbb{H}^3$, where $\Gamma^+=\Gamma\cap\bG(\R)^+$. Then $f_\lambda$ is the restriction to $\bG(\R)^+$ of an adelic automorphic form $\phi_\lambda$ on $\bG(\A)$, which is invariant under $K_\infty^+$ and an open compact subgroup $K_f\subset\bG(\A_f)$ satisfying $\Gamma=\bG(\Q)\cap K_f$, and which vanishes on all but one connected component of the double quotient $\bG(\Q)\backslash\bG(\A)/K_f$.

Let $\mathcal{G}\in\mathscr{S}_{\rm alg}(\R^4)$ be the fixed Gaussian in the algebraic Schwartz space model of the Weil representation from \S\ref{sec:Schroedinger}. Then $F_\lambda=\Theta (\phi_\lambda,\mathcal{G})$ on $\Sp_2=\SL_2$ can be viewed classically as a weight one Maass form on $\mathbb{H}^2$ for the congruence subgroup $\Gamma_0(4D)$, with spectral parameter $\lambda$, and Nebentypus $\chi_D(n) = \big( \frac{D}{|n|} \big)$ where $\big( \frac{D}{|n|} \big)$ is the Kronecker symbol. The Fourier expansion of $F_\lambda$ can be written as
\[
F_\lambda(u+iv)=\sum_{m\in\Z}a_m(F_\lambda;v)e^{2 \pi i mu}\qquad (v>0).
\]
The important point that Rudnick and Sarnak observe, following Maass \cite{Maass59}, is that the Fourier coefficients $a_m(F_\lambda;v)$ \textit{for negative $m$} can be expressed as a discrete automorphic period of $f_\lambda$.

To define this period, let $Y_m$ ($m$ a negative integer) denote the quadric in $V$ given by $Q(x)=m$. Let $Y_m(L)=Y_m(\Q)\cap L$ denote the set of ``$L$-integral points" in $Y_m(\Q)$. Then $Y_m(L)$ admits an action by $\Gamma^+$, and the set $\Gamma^+ \backslash Y_m(L)$ of $\Gamma^+$-orbits is finite; we denote by $y_1,\ldots ,y_h$ their projections to $\Gamma^+\backslash\mathbb{H}^3$. Then the period relation of Rudnick--Sarnak states that
\begin{equation}\label{RS-relation}
a_m(F_\lambda;v)=\widehat{\mathcal{G}}(\lambda;v)\sum_{j=1}^h\frac1{w_j}f_\lambda(y_j),
\end{equation}
where $w_j$ is the order of the stabilizer of $y_j$ in $\Gamma^+$ and
\begin{equation}\label{I-lambda}
\widehat{\mathcal{G}}(\lambda;v) = v \int_{\bG(\R)^+}\mathcal{G}(\sqrt{v}g^{-1}y_0)\varphi_\lambda(g)dg
\end{equation}
is (up to the factor of $v$) the spherical transform of $\mathcal{G}(\sqrt{v}g^{-1}y_0)$, viewed as a function on $g\in\bG(\R)^+$. This transform is non vanishing at $\lambda$ for $v$ large enough.

A crucial ingredient in the proof of the relation \eqref{RS-relation} is a uniqueness result: the Harish-Chandra spherical function $\varphi_\lambda$ on $\bG(\R)^+$ is the unique up to scalar bi-$K_\infty^+$-invariant eigenfunction of the Laplacian of frequency $\lambda$.

\subsection{Notation}\label{sec:periodnotation}

We now establish the notation that will remain in effect through to the end of the paper.  We return to the setting of Section \ref{sec:global-theta}, and continue to let $F$ be a totally real number field with ring of integers $\cO_F$.  We continue to use the notation established in Section \ref{sec:theta-review}, with the exception that the Lagrangian subspace $U$ of $W$ is now denoted $U_W$.

We assume that $V$ is anisotropic over $F$.  Let $2m$ be the dimension of $W$, and let $d$ be the dimension of $V$.  We continue to assume that $d \geqslant 4$ is even, and that $d = n + m$, where $n > m \geqslant 1$.  We fix a real embedding $v_0$ of the number field $F$.  We suppose that $V$ is positive definite at all real places distinct from $v_0$ and of signature $(n,m)$ at $v_0$.

Pick an injective linear map $y_0 : U_W \to V$, and let $U_V = y_0(U_W)$ (we will later identify the spaces $U_V$ and $U_W$, hence the choice of notation).  We assume that $U_V$ is negative definite at $v_0$.  We let $q$ be the quadratic form on $U_W$ obtained by pulling back $Q$ under $y_0$; then $q$ is negative definite at $v_0$ and positive definite at all other real places.

\subsection{Orthogonal and Bessel periods}\label{sec:2periods}

We now generalize the periods featured in the Rudnick--Sarnak period relation to the groups $\bG$ and $\bG'$, using the adelic language.  We shall express an \textit{orthogonal period} of an automorphic form on $\bG$ in terms of a \textit{Bessel period} of its theta lift to $\bG'$.  This period relation was also considered more recently by Gan in \cite{Gan}.

\subsubsection{Orthogonal periods on orthogonal groups}\label{sec:orth-orth}  Let $\bH=\Oo(U_V^\perp)\times\Oo(U_V)$.  Let $\pi$ be an irreducible unitary automorphic representation of $\bG(\A)$. Define an adelic automorphic period map by
\[
\mathscr{P}_\bH (\phi)=\int_{[\bH]} \phi(h) dh\qquad (\phi\in V_\pi^\infty),
\]
where $dh$ is the $\bH(\A)$-invariant probability measure on $[\bH]$. The integral converges absolutely, since $[\bH]$ is compact ($V$ being anisotropic). The association $\phi\mapsto \mathscr{P}_\bH (\phi)$ then yields an element in $\Hom_{\bH(\A)}(\pi,1)$.

\subsubsection{Bessel periods on the symplectic group}
\label{sec:Bessel-period}

As in Section \ref{sec:thetakernel}, we let $\bP_{\rm Siegel}=\bM\bN$ be the Siegel parabolic of $\bG'$ associated to $U_W$. Recall from Section \ref{Type1} that $\bM$ is isomorphic to $\GL(U_W)$, through the map $\GL(U_W)\rightarrow\bM$ given by $a\mapsto m(a)={\rm diag}(a,a^\vee)$. We let $\Oo(U_W)$ be the isometry group of $q$, viewed as a subgroup of $\bM$. Then $\bR=\Oo(U_W)\ltimes\bN$ is the associated {\it Bessel subgroup} $\bR$ of $\bG$.


We shall define Bessel periods by integrating against a character of $\bR$. We begin by recalling from Section \ref{Type1} that $\bN$ can be identified with the space of symmetric morphisms $\Hom^+( U_W^*, U_W)$. Because $q$ may be viewed as an element of $\Hom( U_W, U_W^*)$, we can define a map $\bN(F)\backslash\bN(\A) \to F \backslash \A$ that sends $n(b)$ to $\tr( qb)$. We then define a character $\psi_N$ of $\bN(F)\backslash\bN(\A)$ by $\psi_N(n(b)) = \psi(\tfrac{1}{2} \tr(q b) )$.  Then $\psi_N$ is invariant under conjugation by ${\rm O}(U_W)$, and therefore extends to a character of $\bR(\A)$ that we will denote $1 \times \psi_N$.

Let $\pi'$ be an irreducible unitary automorphic representation of $\bG'(\A)=\Sp(W)(\A)$.  Define an adelic automorphic period map by
\[
\mathscr{B}_\bR^{\, \psi_N}(\phi')=\int_{[\Oo(U_W)]}\int_{[\bN]} \phi'(m(t)n)\psi_N^{-1}(n)dtdn\qquad (\phi'\in V_{\pi'}^\infty),
\]
where $dt$ and $dn$ are the invariant probability measures on $[\Oo(U_W)]$ and $[\bN]$, respectively. The integral converges absolutely, since $[\Oo(U_W)]\times [\bN]$ is compact, yielding an element in $\Hom_{\bR(\A)}(\pi', 1 \times \psi_N)$.

\subsection{An affine variety}\label{sec:affine}
To establish a relation between the Bessel period associated with $U_W$ and the orthogonal period on $\bG$, we shall need to consider the set of embeddings of $(U_W,q)$ into $(V,Q)$ as a quadratic subspace. This motivates the following definition.

Recall that $\mathcal{Y}=U_W^*\otimes V={\rm Hom}(U_W,V)$. We let
\begin{equation}\label{eq:defn-YU}
\mathcal{Y}_q=\{y\in {\rm Hom}(U_W,V): Q\circ y=q\}\subset\mathcal{Y}
\end{equation}
denote the space of all isometric embeddings of $U_W$ into $V$, as quadratic spaces. In other words, $\mathcal{Y}_q$ is the variety of representations of the quadratic form $q$ by $Q$.  Then the natural action of $\bG\times \Oo(U_W)$ on $\mathcal{Y}$ preserves $\mathcal{Y}_q$, which makes $\mathcal{Y}_q$ into an affine $\bG\times \Oo(U_W)$-variety.  Concretely, $(g,t)\in\bG\times\Oo(U_W)$ sends $y\in\mathcal{Y}_q$ to $g\circ y\circ t^{-1}\in\mathcal{Y}_q$. The set of rational points $\mathcal{Y}_q(F)$ is non-empty, as it contains the embedding $y_0$ defined above, and it therefore forms a homogeneous space under the action of $\bG(F)\times\Oo(U_W)(F)$ by Witt's theorem.

For any $F$-algebra $R$, we let $Y_q(R)$ denote the quotient set $\Oo(U_W)(R)\backslash \mathcal{Y}_q(R)$. We denote the points in $Y_q(R)$ by $[y]$, where $y\in \mathcal{Y}_q(R)$. Then the stabilizer of the base point $[y_0]\in Y_q(F)$ in $\bG(F)$ is $\bH(F)=\Oo(U_V^\perp)(F)\times\Oo(U_V)(F)$. Indeed, this is clear for the subgroup $\Oo(U_V^\perp)$, as the latter is already the stabilizer in $\bG$ of $y_0\in\mathcal{Y}_q$ (before quotienting). In the quotient $Y_q$, observe that for every $g\in\Oo(U_V)\subset\bG$ there is $t\in\Oo(U_W)$ such that $g \circ y_0= y_0\circ t^{-1}$. We deduce that
\begin{equation}\label{XS}
\bG(F)/\bH(F)\simeq Y_q(F),
\end{equation}
the isomorphism being induced by the orbit map.

\subsection{Auxiliary structures}\label{sec:level-signature}

We keep the notations of the paragraph Section \ref{sec:2periods}, and now make some choices of auxiliary data to be used throughout the rest of this section.

\subsubsection{Compact subgroups}\label{sec:lattice} Recall that in Section \ref{sec:global-theta} we fixed a lattice $L$ in $V$ which is self-dual at all $v\notin S_V$, and we used $L$ to define, at every finite place $v$, a compact open subgroup $K_{0,v}$ of $\Oo(V)(F_v)$.  For every finite $v$, let $K_v$ be an open subgroup of $K_{0,v}$ such that $K_v = K_{0,v}$ for almost all $v$, and let $K_f = \prod_{v < \infty} K_v$.

We assume that the decomposition of $V_{v_0}$ into positive and negative subspaces chosen in Section \ref{sec:global-theta} is given by $V_{v_0} = U_{V,v_0}^\perp \oplus U_{V,v_0}$.  We let $K_{v_0}$ be the corresponding maximal compact subgroup of $G_{v_0}$, which satisfies $H_{v_0} = K_{v_0}$.  At other real places $v$ we let $K_v = G_v$, and define $K_\infty = \prod_{v \mid \infty} K_v$.  Put $K=K_fK_\infty$.

In a similar way, we fixed a globally self-dual lattice $L'$ in $W$, which we used to define compact open subgroups $K_{0,v}'$ of $\Sp(W)(F_v)$.  For every finite $v$, let $K_v'$ be an open subgroup of $K_{0,v}'$ such that $K_v' = K_{0,v}'$ for almost all $v$, and let $K_f' = \prod_{v < \infty} K_v'$.

For real $v$, we let $K_v'$ be the maximal compact subgroup of $G_v'$ corresponding to the positive complex structure $J_{W_v}$ chosen in Section \ref{sec:global-theta}.  We assume that $J_{W_v}$ is compatible with $q$, in the sense that the quadratic form $\langle J_{W_v} u_1, u_2 \rangle$ on $U_{W,v}$ is equal to $-q$ if $v = v_0$, and equal to $q$ at other real places.  This implies that $\Oo(U_W)(F_v)$ is contained in $K_v'$.  Put $K_\infty' = \prod_{v \mid \infty} K_v'$, and $K'=K_f'K_\infty'$.

\subsubsection{Adelic Schwartz function}\label{sec:adelic-schwartz}

We let $L'_{U_W}$ be a lattice in $U_W$, with dual lattice $L'_{U^*_W}$ in $U_W^*$. We may assume that $L' = L'_{U_W} \oplus L'_{U_W^*}$.  We put $\mathcal{L}_\mathcal{Y} = L'_{U_W^*} \otimes L$.  Then $\mathcal{L}_\mathcal{Y}$ is a lattice in $\mathcal{Y}=U_W^*\otimes V$.  Let $\widehat{\mathcal{L}}_\mathcal{Y}$ be the closure of $\mathcal{L}_\mathcal{Y}$ in $\mathcal{Y}(\A_f)$ and write $\alpha_f$ for the characteristic function of $\widehat{\mathcal{L}}_\mathcal{Y}$.  We may alternately define $\widehat{\mathcal{L}}_\mathcal{Y}$ as the set of maps in $\Hom(U_W,V)(\A_f)$ that map $\widehat{L}'_{U_W}$ into $\widehat{L}$.  Choose any function $\alpha_\infty = \prod_{v | \infty} \alpha_v \in \mathscr{S}_\text{alg}(\mathcal{Y}(F_\infty))$ that is invariant under $\Oo(U_W)(F_\infty) \times K_\infty$. Finally, let
\[
\alpha=\alpha_f\otimes\alpha_\infty\in\mathscr{S}(\mathcal{Y}(\A_f))\otimes \mathscr{S}_\text{alg}(\mathcal{Y}(F_\infty)).
\]

\subsubsection{Integral sets}\label{sec:integral-sets}

Having fixed a base point $y_0\in\mathcal{Y}_q(F)$, we obtain a right $1 \times K_f$-invariant function on $\Oo(U_W)(\A_f)\times \bG(\A_f)$ given by $(t,g)\mapsto \alpha_f ((t^{-1},g^{-1}).y_0)$.  In view of ${\rm Stab}_\bG(y_0)=\Oo(U_V^\perp)$, the function $\alpha((t^{-1},g^{-1})y_0)$ is well-defined on
\[
\Oo(U_W)(\A) \times \left(\Oo(U_V^\perp)(\A)\backslash \bG(\A)/K\right).
\]
Note that $\alpha_f((t^{-1},g^{-1})y_0)$ is the characteristic function of the set  
\[
\{(t,g)\in\Oo(U_W)(\A_f)\times\bG(\A_f):(t^{-1},g^{-1})y_0 \in \widehat{\mathcal{L}}_\mathcal{Y} \}.
\]

For $[y]\in Y_q(\A)=\Oo(U_W)(\A)\backslash\mathcal{Y}_q(\A)$ we put
\begin{equation}\label{defn:alphaU}
\beta([y])=\int_{\Oo(U_W)(\A)}\alpha(t^{-1}y)\, dt.
\end{equation}
Since ${\rm Stab}_\bG([y_0])=\bH$, the function $g\mapsto \beta(g^{-1}[y_0])$ is well-defined on $\bH(\A)\backslash\bG(\A)/K$. The integral in \eqref{defn:alphaU} is factorizable, so that $\beta=\beta_f\otimes\beta_\infty$, with the obvious definitions. Since $\alpha_\infty$ is assumed to be invariant under $\Oo(U_W)(F_\infty)$, we in fact have $\beta_\infty([y])=\alpha_\infty(y)$ under an appropriate measure normalization.

Moreover, we may describe the support of $\beta_f$ as follows.  If $\widehat\Lambda = \Oo(U_W)(\A_f)\widehat{\mathcal{L}}_\mathcal{Y} \subset \mathcal{Y}(\A_f)$, then the support of $\beta_f(g_f^{-1}[y_0])$ is the set
\begin{equation}\label{eq:defn-G-int}
\bG(\A_f)^{\rm int}=\{g_f\in\bG(\A_f):g_f^{-1}[y_0]\in\widehat\Lambda\},
\end{equation}
which is left $\bH(\A_f)$-invariant and right $K_f$-invariant.  For any finite place $v$ we define $G_v^\text{int}$ in the natural way, so that $\bG(\A_f)^{\rm int} = \prod_{v < \infty} G_v^\text{int}$.  Note that $G_v^\text{int}$ may be characterized as the set of $g \in G_v$ such that $g^{-1} U_{V,v}$ contains a free $\cO_v$-submodule $M$ such that $(M,Q)$ is isometric to $(L_{U_W,v}',q)$.

\begin{lemma}\label{lemma:finite-double-quotient}

The set $\bH(\A_f) \backslash \bG(\A_f)^{\rm int} / K_f$ is finite.

\end{lemma}

\begin{proof}

We need to show that $H_v \backslash G_v^\text{int} / K_v$ is finite for all $v$, and has cardinality one for almost all $v$. We do this by defining $\cM$ to be the set of free $\cO_v$-submodules of $L_v$ that are isometric to $L_{U_W,v}'$.  We define a map $\tau : \cM \to H_v \backslash G_v^\text{int}$ by sending $M \in \cM$ to the unique coset $H_v g \in H_v \backslash G_v$ such that $g^{-1} U_{V,v} = M \otimes F_v$, which clearly lies in $H_v \backslash G_v^\text{int}$.  Moreover, $\tau$ is surjective: if $g \in G_v^\text{int}$, then $g^{-1} U_{V,v}$ contains some $M \in \cM$, and hence $\tau(M) = H_vg$.  It may also be checked that $\tau$ is $K_v$-equivariant.

For a general place $v$, the set $\cM$ is compact, and $K_v$ acts with open orbits.  (The second assertion is perhaps not entirely trivial, but we leave the details to the reader.)  It follows that the quotient $K_v \backslash \cM$ is finite, which implies the same for $H_v \backslash G_v^\text{int} / K_v$, by the surjectivity and $K_v$-equivariance of $\tau$.

We next assume that $L_v$ and $L_{U_W,v}'$ are self-dual, and that $K_v$ is equal to the stabilizer of $L_v$, which happens for almost all $v$.  In this case, any $M \in \cM$ must also be self-dual.  By \cite[I Prop 3.2]{Baeza}, this implies that the orthogonal complement of $M$ in $L_v$, which we denote by $M^\perp$, must satisfy $L_v = M \oplus M^\perp$, which implies that $M^\perp$ is also self-dual.  Now, let $M, M' \in \cM$.  Because $M$ and $M'$ have the same discriminant, it follows that $M^\perp$ and $M'^\perp$ also have the same discriminant, and combined with their self-duality this means they are isometric.  As a result, there is $g \in G_v$ that maps $M$ isometrically to $M'$ and $M^\perp$ isometrically to $M'^\perp$.  We see from this that $g \in K_v$.  Therefore $K_v$ acts transitively on $\cM$, and hence on $H_v \backslash G_v^\text{int}$.
\end{proof}

%
%

\subsection{An archimedean integral}

Once again we let $v_0$ be the fixed real place of Section \ref{sec:periodnotation}.  In this section only, we let $\alpha_{v_0} \in \mathscr{S}(\mathcal{Y}(F_{v_0}))$ be an arbitrary Schwartz function, with no assumption of invariance under $\Oo(U_W)(F_\infty) \times K_\infty$.  For $\lambda\in\ga^*_\C$ we let
\begin{equation}\label{eq:I-lambda}
\widehat{\alpha}_{v_0}(\lambda)=\int_{G_{v_0}}  \alpha_{v_0}(g^{-1}y_0)\varphi_{-\lambda}(g)dg,
\end{equation}
where $\varphi_{-\lambda}$ is the Harish-Chandra spherical function. For simplicity, we have suppressed the dependency on $y_0$ in the notation.

\begin{lemma}
\label{tempered-dist}

For any $\alpha_{v_0} \in \mathscr{S}(\mathcal{Y}(F_{v_0}))$ and any $\lambda\in\ga^*_\C$, the integral \eqref{I-lambda} converges absolutely, and $\alpha_{v_0} \mapsto \widehat{\alpha}_{v_0}(\lambda)$ defines a tempered distribution.

\end{lemma}

\begin{proof}

As the place $v_0$ shall be fixed throughout this proof, we omit it from the notation, and write $\mathcal{Y}$ for $\mathcal{Y}(F_{v_0})$, $G$ and $K$ for $G_{v_0}$ and $K_{v_0}$, etc.  We choose a basis $(e_1, \ldots, e_m)$ for $U_W$ so that $q$ is the standard negative definite quadratic form, and a basis $(v_1, \ldots, v_d)$ for $V$ so that $Q$ is given by $z_1^2 + \ldots + z_n^2 - z_{n+1}^2 - \ldots -z_d^2$.  We may also assume that $y_0(e_i) = v_{n+i}$ for $1 \leqslant i \leqslant m$.  We define $\| \cdot \|_V$ to be the norm on $V$ corresponding to the standard positive definite quadratic form.  We define $\| \cdot \|_\mathcal{Y}$ to be the mapping norm on $\mathcal{Y}= \Hom(U_W,V)$ with respect to the norms $-q$ and $\| \cdot \|_V$.  We let $B(r)$ be the open ball about zero of radius $r$ in $\mathcal{Y}$ with respect to $\| \cdot \|_\mathcal{Y}$.

We define the subspaces $\gk$ and $\p$ of $\g$ in the usual way.  We let $X_i \in \g$, $1 \leqslant i \leqslant m$, be the vectors which act on $V$ by fixing all basis vectors except for $v_i$ and $v_{n+i}$, and whose action on $\text{span}\{ v_i, v_{n+i} \}$ is given by the matrix $\big(\begin{smallmatrix} 0 & 1\\ 1 & 0 \end{smallmatrix}\big)$.  Then the $X_i$ span a maximal commuting subalgebra $\ga$ of $\p$, with corresponding subgroup $A = \exp(\ga)$.  For $H \in \ga$, we let $H_i$ be its coordinates with respect to the basis $\{ X_i \}$, and we equip $\ga$ with the standard norm with respect to these coordinates.  As in Section \ref{sec:gp-decomp}, we define $B_\ga(0,R)$ to be the ball in $\ga$ centered at 0 and of radius $R > 0$ relative to this norm. (Note that the norm we are using here is different to the one used in Section \ref{sec:gp-decomp}.)  We define $A_R = \exp(B_\ga(R))$, and $G_R = K A_R K$.

We now turn to bounding the integral \eqref{I-lambda}.  We may assume without loss of generality that $\lambda$ is real, so that $\varphi_{-\lambda} > 0$.  If we let $\gamma : G \to \mathcal{Y}$ be the orbit map $\gamma(g) = g^{-1} y_0$, then the integral defining $\widehat{\alpha}(\lambda)$ may be thought of as the integral of $\alpha$ against the measure $\mu = \gamma_*(\varphi_{-\lambda}(g) dg)$.  To show that $\mu$ is a tempered distribution, it suffices to show that $\mu( B(r))$ grows polynomially in $r$. 

We do this by first bounding the set $\gamma^{-1}(B(r)) \subset G$, and in particular showing that $\gamma^{-1}(B(r)) \subset G_{c \log r}$ for some $c > 0$ and $r$ large. Because $\gamma^{-1}(B(r))$ is bi-invariant under $K$, it suffices to show that $A \cap \gamma^{-1}(B(r)) \subset A_{c \log r}$.  In order to have $e^H \in \gamma^{-1}(B(r))$, we must have $e^{-H} y_0 \in B(r)$.  This implies that $\| e^{-H} y_0(e_i) \|_V \leqslant r$ for all $i$, which gives
\begin{align*}
\| e^{-H} v_{n+i} \|_V & \leqslant r \\
\| -\sinh(H_i) v_i + \cosh(H_i) v_{n+i} \|_V & \leqslant r \\
\cosh(H_i) & \leqslant r \\
|H_i| & \leqslant \log (2r).
\end{align*}
This implies that $\| H \| \ll \log r$, and hence that $\gamma^{-1}(B(r)) \subset G_{c \log r}$ for some $c > 0$ as required.

The quantity $\mu( B(r))$ is equal to the integral of $\varphi_{-\lambda}(g) dg$ over $\gamma^{-1}(B(r))$, and this is bounded above by the volume of $G_{c \log r}$ times the maximum of $\varphi_{-\lambda}(g)$ on $G_{c \log r}$.  The usual formula for the volume form on $G$ in Cartan coordinates implies that $\text{vol}(G_{c \log r}) \ll r^a$ for some $a$, and the bound $| \varphi_{-\lambda}(e^H) | \ll e^{ \beta \| H \| \| \lambda \|}$ for some $\beta > 0$ gives a polynomial upper bound for $\varphi_{-\lambda}$.  This completes the proof.
\end{proof}

\subsection{Period relation}

We are now ready to state and prove the fundamental period relation which lies at the heart of Theorem \ref{sup-thm}. We shall state it here in its adelic form. A more classical formulation will be given in Section \ref{adelic2classical}.

We retain the notation of Sections \ref{sec:2periods}--\ref{sec:level-signature}, including the spaces $V$ and $W$, the quadratic space $(U_W,q)$, the base point $y_0 \in \mathcal{Y}_q(F)$, and the function $\alpha \in \mathscr{S}(\mathcal{Y}(\A_f))\otimes\mathscr{S}_{\rm alg}(\mathcal{Y}(F_\infty))$.

\begin{prop}\label{period-relation}
%
%

Let $\phi_\lambda$ be a $K$-invariant automorphic form on $\bG(\A)$ that is an eigenfunction of the ring of invariant differential operators on $G_{v_0} / K_{v_0}$ with spectral parameter $\lambda\in\ga_\C^*$.  Let $\phi'_\lambda=\Theta(\phi_\lambda,\alpha;W)$ be its theta lift to $\bG'(\A)$. Then
\begin{multline*}
\mathscr{B}_\bR^{\, \psi_N}(\phi_\lambda')= \widehat{\alpha}_{v_0}( -\overline\lambda) \prod_{v\mid\infty, v\neq v_0}\alpha_v(y_0) \\
\sum_{[b]\in \bH(\A_f)\backslash \bG(\A_f)^{\rm int}/K_f}  \textup{vol}(\bH(\A_f)bK_f)\beta_f( b^{-1} [y_0] ) \mathscr{P}_\bH (R(b) \overline{\phi_\lambda} ),
\end{multline*}
where $\textup{vol}(\bH(\A_f)bK_f)$ is the volume of $\bH(\A_f)bK_f$ considered as a subset of $\bH(\A_f)\backslash\bG(\A_f)$. \end{prop}

\begin{proof}
Inserting the definition of the theta lift in \eqref{eq:Theta-fn}, we have
\[
\mathscr{B}_\bR^{\, \psi_N}(\phi_\lambda')=\int_{[\Oo(U_W)]}\int_{[\bN]}\int_{[\bG]} \overline{\phi_\lambda}(g)\Theta(g,m(t)n;\alpha) \psi_N^{-1}(n)dg dn dt.
\]
Using the explicit formulae \eqref{eq:explicit-theta} for the action of the Weil representation (and recalling that ${}^t t=t^{-1}$), we have
\[
\Theta(g,m(t)n(b);\alpha)=\sum_{y\in \mathcal{Y}(F)}\psi\left(\frac12\langle (b\otimes {\rm Id}_V)y, y\rangle_{\mathcal{W}}\right) \alpha((g^{-1},t^{-1})y).
\]
It may be shown that $\langle (b\otimes {\rm Id}_V)y, y\rangle_{\mathcal{W}} = \tr( b\, {}^t y Q y)$, where we are viewing $Q$ as an element of $\text{Hom}(V, V^*)$. Recall the definition of the variety $\mathcal{Y}_q$ given in \eqref{eq:defn-YU} along with the choice of additive character $\psi_N$ in \S\ref{sec:Bessel-period}. From orthogonality, we obtain
\begin{align*}
\int_{[\bN]}\Theta(g,m(t)n;\alpha) \psi_N^{-1}(n)dn&=\sum_{y\in \mathcal{Y}(F)}\alpha((g^{-1},t^{-1})y)\int_{ [ \mathbf{Hom}^+]}\psi\left(\frac12 \tr( b({}^t y Q y - q) ) \right)db\\
&=\sum_{y\in \mathcal{Y}_q(F)}\alpha((g^{-1},t^{-1})y),
\end{align*}
where $[\mathbf{Hom}^+]$ denotes the adelic quotient associated to the vector space $\text{Hom}^+(U_W^*, U_W)$.  Inserting this into the expression for the Bessel period, we obtain
\[
\mathscr{B}_\bR^{\, \psi_N}(\phi_\lambda')=\int_{[\bG]} \overline{\phi_\lambda}(g)\int_{[\Oo(U_W)]}\sum_{y\in \mathcal{Y}_q(F)}\alpha((g^{-1},t^{-1})y) dtdg.
\]
Unfolding the integral over $[\Oo(U_W)]$ with the quotient $Y_q(F)=\Oo(U_W)(F)\backslash\mathcal{Y}_q(F)$, and recalling the definition of $\beta$ in \eqref{defn:alphaU}, we find
\begin{align*}
\mathscr{B}_\bR^{\, \psi_N}(\phi_\lambda')&=\int_{[\bG]} \overline{\phi_\lambda}(g)\int_{[\Oo(U_W)]}\sum_{\gamma\in \Oo(U_W)(F)}\sum_{[y]\in Y_q(F)}\alpha((g^{-1},t^{-1}\gamma^{-1}) y) dt dg\\
&=\int_{[\bG]} \overline{\phi_\lambda}(g)\sum_{[y]\in Y_q(F)}\int_{\Oo(U_W)(\A)}\alpha((g^{-1},t^{-1})y) dt dg\\
&=\int_{[\bG]} \overline{\phi_\lambda}(g)\sum_{[y]\in Y_q(F)}\beta(g^{-1}[y]) dg.
\end{align*}

Recall from \eqref{XS} that the orbit map $\gamma\mapsto \gamma^{-1}[y_0]$ (note the inverse) identifies $Y_q(F)$ with $\bH(F)\backslash\bG(F)$. From this, another unfolding yields
\begin{align*}
\mathscr{B}_\bR^{\, \psi_N}(\phi_\lambda')&=\int_{[\bG]} \overline{\phi_\lambda}(g)\sum_{\gamma\in\bH(F)\backslash\bG(F)}\beta(g^{-1}\gamma^{-1}[y_0] )dg\\
&=\int_{\bH(F)\backslash\bG(\A)} \overline{\phi_\lambda}(g)\beta(g^{-1}[y_0])dg\\
&=\int_{\bH(\A)\backslash\bG(\A)}\int_{\bH(F)\backslash\bH(\A)} \overline{\phi_\lambda}(hg) \beta(g^{-1}h^{-1}[y_0]) dhdg.
\end{align*}
From the left $\bH(\A)$-invariance of $g\mapsto \beta(g^{-1}[y_0])$ we get
\[
\mathscr{B}_\bR^{\, \psi_N}(\phi_\lambda')=\int_{\bH(\A)\backslash\bG(\A)}\beta(g^{-1}[y_0]) \mathscr{P}_{\bH}(R(g)\overline{\phi_\lambda}) dg.
\]
We write this as
\be
\label{period-expr}
\mathscr{B}_\bR^{\, \psi_N}(\phi_\lambda')=\int_{\bH(\A_f)\backslash\bG(\A_f)} \beta_f(g_f^{-1}[y_0]) \int_{H_\infty \backslash G_\infty} \beta_\infty(g_\infty^{-1}[y_0]) \mathscr{P}_{\bH}(R(g_f g_\infty)\overline{\phi_\lambda}) dg_\infty dg_f,
\ee
and consider the inner archimedean integral.  Since $\overline{\phi_\lambda}$ is right $K_\infty$-invariant, and $H_\infty = K_\infty$, $\mathscr{P}_\bH(R(g_f g_\infty) \overline{\phi_\lambda})$ is a bi-$K_\infty$-invariant function on $G_\infty$. Moreover, since $\overline{\phi_\lambda}$ was assumed to be an eigenfunction on $G_{v_0}$ with spectral parameter $\overline\lambda$, the same is true for $\mathscr{P}_\bH(R(g_f g_\infty) \overline{\phi_\lambda})$. From the uniqueness of spherical functions, it follows that
\[
\mathscr{P}_\bH(R(g_f g_\infty) \overline{\phi_\lambda})=\varphi_{\overline\lambda}(g_{v_0})\mathscr{P}_\bH(R(g_f) \overline{\phi_\lambda}).
\]
Inserting this into \eqref{period-expr} gives
\[
\mathscr{B}_\bR^{\, \psi_N}(\phi_\lambda')=\int_{\bH(\A_f)\backslash\bG(\A_f)} \beta_f(g_f^{-1}[y_0]) \mathscr{P}_\bH(R(g_f) \overline{\phi_\lambda}) dg_f \int_{H_\infty \backslash G_\infty} \beta_\infty(g_\infty^{-1}[y_0]) \varphi_{\overline\lambda}(g_{v_0}) dg_\infty.
\]
Recalling from \S\ref{sec:adelic-schwartz} that $\beta_\infty([y])=\alpha_\infty(y)$, the archimedean integral unfolds to
\[
\int_{G_\infty} \alpha_\infty( g_\infty^{-1} y_0) \varphi_{\overline\lambda}(g_{v_0}) dg_\infty = \widehat{\alpha}_{v_0}( -\overline\lambda) \prod_{v\mid\infty, v\neq v_0}\alpha_v(y_0).
\]
From the right $K_f$-invariance of $\beta_f( g_f^{-1}[y_0])$, and the description of its support in \eqref{eq:defn-G-int}, the integral over all finite places becomes
\[
\sum_{[b]\in\bH(\A_f)\backslash \bG(\A_f)^{\rm int}/K_f}  \textup{vol}(\bH(\A_f)bK_f)\beta_f( b^{-1} [y_0] ) \mathscr{P}_{\bH}(R(b) \overline{\phi_\lambda}).
\]
Combining these gives the proposition.
\end{proof}

\subsection{Distinction relation}\label{sec:lemma-schwartz-space}
Using Proposition \ref{period-relation}, as well as density statement in functional analysis that we prove in Proposition \ref{prop:density-in-Schwartz}, we deduce the following corollary.

\begin{cor}\label{distinction}
Let notations and assumptions be as in Proposition \ref{period-relation}. If
\[
\sum_{[b]\in\bH(\A_f)\backslash \bG(\A_f)^{\rm int}/K_f} \textup{vol}(\bH(\A_f)bK_f)\beta_f( b^{-1} [y_0]) \mathscr{P}_\bH (R(b)\phi_\lambda)\neq 0
\]
then $\Theta(\phi_\lambda, \alpha;W)$ is non-zero for some choice of $\alpha_\infty$.
\end{cor}
\begin{proof}

By Corollary \ref{alphahat-nonvanish} below, we may find  $\alpha\in \mathscr{S}_{\rm alg}(\mathcal{Y}(F_{v_0}))$, invariant under $\textup{O}(U_W)(F_{v_0}) \times K_{v_0}$, such that $\widehat{\alpha}_{v_0}(-\overline\lambda)\neq 0$. In particular, $\alpha_{v_0}$ satisfies the conditions of \S\ref{sec:adelic-schwartz}. We may furthermore assume that $\alpha_v(y_0)\neq 0$ for all archimedean places $v$ distinct from $v_0$. Proposition \ref{period-relation} then implies that $\phi'_\lambda = \Theta(\phi_\lambda, \alpha; W)$ is non-zero, because it satisfies $\mathscr{B}_\bR^{\, \psi_N}(\phi_\lambda') \neq 0$.
\end{proof}

We now prove that $\mathscr{S}_{\rm alg}(\mathcal{Y}(F_{v_0}))$ is dense in $\mathscr{S}(\mathcal{Y}(F_{v_0}))$ with the Schwartz topology, as promised in Remark \ref{rem:dense}. In order to simplify notation in the proof, and because the result may be of independent interest, we shall state it in a way that does not involve our previously established notation.  Instead, until the end of the proof of Proposition \ref{prop:density-in-Schwartz}, $d$ will denote an arbitrary positive integer.  We let $\mathscr{S}(\R^d)$ denote the space of Schwartz functions on $\R^d$, with the topology coming from the seminorms
\[
p_{ab}(f) = \| x^a f^{(b)} \|_\infty\qquad (a,b\in\N).
\]
We let the algebraic Schwartz space $\mathscr{S}_{\rm alg}(\R^d)$ be the space of all products of polynomials on $\R^d$ with the Gaussian $e^{-\| x \|^2/2}$.

\begin{prop}\label{prop:density-in-Schwartz}

$\mathscr{S}_{\rm alg}( \R^d)$ is dense in $\mathscr{S}( \R^d)$ with the Schwartz topology.

\end{prop}

\begin{proof}

We shall use the theory of Hermite functions, for which we refer to \cite[Ch. II, Section 9.4, p. 91]{CH} for basic definitions.  We let $H_0 = -\Delta + x^2$ be the Hermite operator on $\R$, and $H = -\Delta + \|x \|^2$ be the Hermite operator on $\R^d$.  The Hermite functions of one variable,
\be
\label{Hermite-formula}
h_n(x) = (2^n n! \sqrt{\pi} )^{-1/2} (-1)^n e^{x^2/2} \frac{d^n}{dx^n} e^{-x^2}
\ee
for $n \geqslant 0$, have the following properties:

\begin{enumerate}

\item $h_n(x) = p_n(x) e^{-x^2/2}$, where $p_n(x)$ is a polynomial of degree $n$.

\item\label{Hermite-eqn} $H_0 h_n = (2n+1) h_n$.

\item The functions $h_n$ form an orthonormal basis for $L^2(\R)$.

\end{enumerate}
Let $\alpha = (\alpha_1, \ldots, \alpha_d) \in \N^d$ be a multi-index, and let $|\alpha| = \alpha_1 + \ldots + \alpha_d$.  If we define $\phi_\alpha(x) = h_{\alpha_1}(x_1) \cdots h_{\alpha_d}(x_d)$ for all $\alpha$, then $\{ \phi_\alpha \}$ is an orthonormal basis for $L^2(\R^d)$ satisfying $H \phi_\alpha = (2|\alpha| + d)\phi_\alpha$.

If $f \in \mathscr{S}( \R^d)$, we may expand $f$ (viewed as a vector in $L^2(\R^d)$) in the orthonormal basis $\{ \phi_\alpha \}$ as $f = \sum_\alpha \langle f, \phi_\alpha \rangle \phi_\alpha$.  Because the partial sums of this expansion lie in $\mathscr{S}_{\rm alg}( \R^d)$, it suffices to show that this sum converges in $\mathscr{S}( \R^d)$.  We have
\[
\langle f, \phi_\alpha \rangle = (2|\alpha| + d)^{-k} \langle f, H^k \phi_\alpha \rangle = (2|\alpha| + d)^{-k} \langle H^k f, \phi_\alpha \rangle \ll_k (2|\alpha| + d)^{-k}
\]
for any $k$. As a result, the convergence of $\sum_\alpha \langle f, \phi_\alpha \rangle \phi_\alpha$ will follow if we can show that for any seminorm $p_{ab}$ as above, there is an $l(a,b) > 0$ such that
\begin{equation}
\label{seminormbd}
p_{ab}( \phi_\alpha) \ll_{a,b} |\alpha|^{l(a,b)}.
\end{equation}

We claim that it suffices to establish the bound \eqref{seminormbd} in the special case when $d = 1$ and $b = 0$.  To demonstrate this, suppose we know the functions $h_n$ satisfy $p_{a0}(h_n) \ll_{a} n^{l(a,0)}$ for all $a$.  When combined with the relation $h_n' = -x h_n + 2n h_{n-1}$, this implies that we also have $p_{a1}(h_n) \ll_{a} n^{l(a,1)}$.  We may use the differential equation \eqref{Hermite-eqn} to show that for any $b$ we have
\[
h_n^{(b)}(x) = P_b(x,n) h_n(x) + Q_b(x,n) h'_n(x),
\]
where $P_b$ and $Q_b$ are polynomials depending only on $b$, and this implies that $p_{ab}(h_n) \ll_{a,b} n^{l(a,b)}$ for any $a$ and $b$, and hence gives \eqref{seminormbd}.

Having reduced to proving that $p_{a0}( h_n) \ll_a n^{l(a)}$ for any $a$, we now address this bound.  One could probably derive this from classical asymptotic formulas for $h_n$ such as \cite{PR} or \cite[Theorem 8.22.9]{Sz}, but we have found it simpler to use a bound of Koch--Tataru in the case where $x \ll \sqrt{n}$, combined with an elementary bound when $x \gg \sqrt{n}$. Koch--Tataru \cite[Corollary 3.2]{KT} prove that $\| h_n \|_\infty \ll n^{-1/12}$ (in fact, any polynomial bound will suffice for our purposes), which implies that
\[
\underset{ |x| \leqslant 10 \sqrt{n} }{\sup} | x^a h_n(x) | \ll_a n^{a/2-1/12}.
\]
The remaining range is treated in the following lemma.
\end{proof}

\begin{lemma}

There is a constant $C$ such that for $|x| > 10 \sqrt{n}$, we have $|h_n(x)| \leqslant C e^{-x^2/4}$.

\end{lemma}

\begin{proof}

We use the formula \eqref{Hermite-formula} for $h_n$, together with the contour integral formula
\[
\frac{d^n}{dx^n} e^{-x^2} = \frac{n!}{2\pi i} \int_\cC e^{-(z+x)^2} \frac{dz}{z^{n+1}},
\]
where $\cC$ is an anticlockwise circle of radius $\sqrt{n}$ centered at 0.

We have
\[
\text{Re}(-(z+x)^2) = -x^2 + \text{Re}( -2zx - z^2) \leqslant -x^2 + 2|x| \sqrt{n} + n.
\]
If we assume that $10 \sqrt{n} < |x|$, this becomes
\[
\text{Re}(-(z+x)^2) \leqslant -x^2 + \frac{1}{5} x^2 + \frac{1}{100} x^2 < -3x^2/4.
\]
We therefore have
\begin{align*}
\left| \int_\cC e^{-(z+x)^2} \frac{dz}{z^{n+1}} \right| & \leqslant \int_\cC | e^{-(z+x)^2} z^{-n-1} | dz \\
& < (2\pi \sqrt{n}) n^{-(n+1)/2} e^{-3x^2/4} \\
& = 2 \pi n^{-n/2} e^{-3x^2/4},
\end{align*}
which in turn yields the bound
\[
\left| \frac{d^n}{dx^n} e^{-x^2} \right| < n! n^{-n/2} e^{-3x^2/4}.
\]
Substituting this into \eqref{Hermite-formula} gives
\begin{align*}
| h_n(x) | & < (2^n n! \sqrt{\pi} )^{-1/2} e^{x^2/2} \left[ n! n^{-n/2} e^{-3x^2/4} \right] \\
& = \pi^{-1/4} 2^{-n/2} (n!)^{1/2} n^{-n/2} e^{-x^2/4}.
\end{align*}
We must therefore prove that $2^{-n/2} (n!)^{1/2} n^{-n/2}$ is bounded independently of $n$.  Using Stirling's formula, the logarithm of this is given by
\begin{align*}
\log 2^{-n/2} (n!)^{1/2} n^{-n/2} & = - \frac{n \log 2}{2} + \frac{n \log n}{2}  - \frac{n}{2} - \frac{n \log n}{2} + O( \log n) \\
& =  - \frac{n \log 2}{2} - \frac{n}{2} + O( \log n),
\end{align*}
which is bounded above.  This completes the proof.
\end{proof}

\begin{cor}
\label{alphahat-nonvanish}

For any $\lambda \in \ga^*_\C$, there is a function $\alpha_{v_0} \in \mathscr{S}_{\rm alg}(\mathcal{Y}(F_{v_0}))$ that is invariant under $\textup{O}(U_W)(F_{v_0}) \times K_{v_0}$, and such that $\widehat{\alpha}_{v_0}(\lambda) \neq 0$.

\end{cor}

\begin{proof}

We first construct a (possibly non-invariant) function satisfying $\widehat{\alpha}_{v_0}(\lambda) \neq 0$.  The map $\alpha_{v_0} \to \widehat{\alpha}_{v_0}(\lambda)$ is a tempered distribution, and we proved in Proposition \ref{prop:density-in-Schwartz} that $\mathscr{S}_{\rm alg}(\mathcal{Y}(F_{v_0}))$ is dense in $\mathscr{S}(\mathcal{Y}(F_{v_0}))$.  It therefore suffices to find $\alpha_{v_0} \in \mathscr{S}(\mathcal{Y}(F_{v_0}))$ for which $\widehat{\alpha}_{v_0}(\lambda) \neq 0$.  To do this, we observe that the orbit map $g \to g^{-1} y_0$ is an embedding of $\Oo(n) \backslash \Oo(n,m)$ into $\mathcal{Y}(F_{v_0})$, because $\Oo(n)$ is the stabilizer of $y_0$. The distribution $\widehat{\alpha}_{v_0}(\lambda)$ is given by integrating $\alpha_{v_0}$ against the function $\varphi_{-\lambda}$ pushed forward to this orbit, and this is clearly nonzero for some $\alpha_{v_0} \in \mathscr{S}(\mathcal{Y}(F_{v_0}))$.

It may be checked that the distribution $\alpha_{v_0} \to \widehat{\alpha}_{v_0}(\lambda)$ is invariant under $\textup{O}(U_W)(F_{v_0}) \times K_{v_0}$, and so we may assume that $\alpha_{v_0}$ has the required invariance by averaging.
\end{proof}

\subsection{Classical reformulation}\label{adelic2classical}

We would now like to write the period relation in Proposition \ref{period-relation} more classically, by writing the double quotient space $\bG(F)\backslash\bG(\A)/K$ as a union of locally symmetric spaces, and the adelic periods $\mathscr{P}_\bH$ as a collection of classical periods on these spaces. We retain all notation from previous sections.

By a theorem of Borel \cite[Theorem 5.1]{Borel}, the set of genus classes for $V$, as parametrized by $\bG(F)\backslash\bG(\A_f)/K_f$, is finite. We fix a complete set of representatives $\{g_i\}_{i\in I}$ for this double coset space. As usual, we let $S=G_\infty/K_\infty$. The map
\[
\coprod_{i\in I}  S\longrightarrow\bG(F)\backslash (\bG(\A_f)/K_f\times G_\infty/K_\infty),
\]
sending $(i,g_\infty K_\infty)\in I\times S$ to the orbit $\bG(F).(g_i K_f,g_\infty K_\infty)$, is clearly surjective, and the fiber over $\bG(F).(g_iK_f, K_\infty)$ is the $\Gamma_i$-orbit of $(i, K_\infty)$, where $\Gamma_i=\bG(F)\cap g_iK_fg_i^{-1}$. We therefore obtain an isomorphism 
\[
\bG(F)\backslash\bG(\A)/K=\coprod_{i\in I} \Gamma_i\backslash  S.
\]
As in the statement of Theorem \ref{sup-thm}, we let $Y$ denote the latter disjoint union. For any continuous function $\phi\in L^2([\bG])$, right-invariant under $K$, we use the above decomposition to define a collection of functions $\phi^{(i)}\in L^2(\Gamma_i\backslash S)$ via the rule $\phi^{(i)}(g_\infty)=\phi(g_ig_\infty)$.

We would now like to write the $\bH$-period of $\phi$ classically. Let $K_f^{\bH}=K_f\cap \bH(\A_f)$, and ${\rm Gen}_{\bH}=\bH(F)\backslash\bH(\A_f)/K_f^\bH$. Since $K_\infty^\bH=\bH(F_\infty)$ is compact, we have a well-defined injective map
\[
{\rm Gen}_{\bH}=\bH(F)\backslash\bH(\A)/K^\bH\rightarrow \bG(F)\backslash\bG(\A)/K= Y.
\]
The image in $Y$ is a finite collection of points, which (abusing notation) we continue to denote by ${\rm Gen}_{\bH}$. Thus
\begin{equation}\label{classical-geod-period}
\mathscr{P}_\bH(\phi)=\sum_{x\in {\rm Gen}_{\bH}} w_x \phi(x),
\end{equation}
where $w_x$ is the weight given to $x$ by the natural probability measure on ${\rm Gen}_{\bH}$.

We have just written the $\bH$-period of $\phi$ classically, but in the period relation of Proposition \ref{period-relation}, we encountered a certain \textit{finite union} (cf. Lemma \ref{lemma:finite-double-quotient}) of weighted $\bH$-periods, of the form
\begin{equation}\label{eq:sum-of-periods}
\sum_{[b]\in\bH(\A_f)\backslash \bG(\A_f)^{\rm int}/K_f} \textup{vol}(\bH(\A_f)bK_f) \beta_f( b^{-1} [y_0] )\mathscr{P}_\bH(R(b)\overline{\phi_\lambda}).
\end{equation}
To write this sum classically, we begin by noting that
\[
\mathscr{P}_\bH(R(b)\overline{\phi_\lambda}) = \int_{[\bH] b} \overline{\phi_\lambda}(x) dx.
\]
We let $K_f^{\bH}(b) = b K_f b^{-1} \cap \bH(\A_f)$, and $K^{\bH}(b) = K_\infty^\bH K_f^{\bH}(b)$.  If we define ${\rm Gen}_{\bH}^b = \bH(F)\backslash\bH(\A_f)/ K^{\bH}(b)$, then as above we have an injection
\[
{\rm Gen}_{\bH}^b = \bH(F)\backslash\bH(\A)/ K^{\bH}(b) \rightarrow \bG(F)\backslash\bG(\A)/K= Y
\]
sending $x \in {\rm Gen}_{\bH}^b$ to $xb$.  It follows that
\[
\mathscr{P}_\bH(R(b)\overline{\phi_\lambda}) = \sum_{x\in {\rm Gen}_{\bH}^b} w_x^b \overline{\phi_\lambda}(xb)
\]
for positive weights $w_x^b$, and this implies that \eqref{eq:sum-of-periods} can be written as
\begin{equation}\label{eq:H-period-double-sum}
\sum_{[b]\in\bH(\A_f)\backslash \bG(\A_f)^{\rm int}/K_f}\beta(b) \sum_{x\in {\rm Gen}_{\bH}^b} w_x^b \overline{\phi_\lambda}(xb),
\end{equation}
where we have put $\beta(b)= \textup{vol}(\bH(\A_f)bK_f) \beta_f(b^{-1} [y_0])$.

Finally, let let $\mathscr{X}\subset Y$ denote the finite collection of points
\begin{equation}\label{eq:defn-shifted-pts}
\mathscr{X}=\left\{xb\in Y : \; [b]\in\bH(\A_f)\backslash\bG(\A_f)^{\rm int}/K_f, \; x\in {\rm Gen}_{\bH}^b \right\}
\end{equation}
appearing in \eqref{eq:H-period-double-sum}. For each $p=xb\in\mathscr{X}$ let $c_p= \textup{vol}(\bH(\A_f)bK_f) \beta_f(b^{-1}[y_0])w_x^b$  -- a positive real number -- denote the weight appearing in that same formula. Then \eqref{eq:sum-of-periods} becomes
\begin{equation}\label{eq:final-classical-period}
\sum_{p\in\mathscr{X}} c_p\overline{\phi_\lambda}(p),
\end{equation}
a weighted sum of point evaluations.

\begin{remark}
Recall from \S\ref{sec:affine} the set-theoretic quotient $Y_q(R)=\Oo(U_W)(R)\backslash \mathcal{Y}_q(R)$, where $\mathcal{Y}_q$ is the algebraic variety defined in \eqref{eq:defn-YU} and $R$ is any $F$-algebra. It seems plausible that the work of Ellenberg--Venkatesh \cite[\S 3]{EV} should provide for a parametrization of the collection of points $\mathscr{X}$ in \eqref{eq:defn-shifted-pts} in terms of $\Gamma$-equivalence classes of ``integral points on $Y_q$". Such an interpretation would recover the Rudnick--Sarnak example of Section \ref{Maass-period-relation}. As it is not strictly necessary for our purposes, we have not pursued this argument.
\end{remark}

\section{Proof of Theorem \ref{sup-thm}}\label{ProofsOfThms}

We now return to the proof of our main result, Theorem \ref{sup-thm}, putting all ingredients together. We retain the notation from Sections \ref{sec:periodnotation}, \ref{sec:2periods}, and \ref{sec:level-signature}.

\subsection{Archimedean representations.}

We begin by introducing the archimedean representations of $\bG$ and $\bG'$ that we shall work with.  For $\bG$, these are simply the spherical representations.  For $\lambda \in \ga^*_\C$, let $\pi_{\lambda, v_0}$ be the irreducible spherical representation of $G_{v_0}$ with spectral parameter $\lambda$ defined in Proposition \ref{O-spherical}, and let $\pi_\lambda$ be the representation of $G_\infty$ obtained as the tensor product of $\pi_{\lambda, v_0}$ with the trivial representation of $G_v$ at the other real places.

For $\bG'$, the representations we shall consider are those admitting a one-dimensional $K'_\infty$-type.  We define the one-dimensional representation $\tau$ of $K'_\infty$ by setting
\[
\tau^{-1} = \det{}^{(n-m)/2}\bigotimes_{v\mid\infty, v\neq v_0}\det{}^{d/2}.
\]
We choose $\tau$ this way so that we may consider forms that are $\tau^{-1}$-isotypic, and this allows us to write them in a way consistent with our earlier notation in e.g. Section \ref{sec:upper-bd-notation}.  For $\lambda \in \ga'^*_\C$, we let $\pi'_{\lambda, \tau, v_0}$ be the representation of $G'_{v_0}$ with $K'_{v_0}$-type $\tau_{v_0}^{-1} = \det{}^{(n-m)/2}$ and spectral parameter $\lambda$ defined in Lemma \ref{tau-spherical-class}. For other real places $v$, we likewise use Lemma \ref{tau-spherical-class} to define a representation $\pi'_{\tau,v}$ of $G'_v$ with $K'_v$-type $\tau_{v}^{-1} = \det{}^{d/2}$ and spectral parameter $(\tfrac{n}{2} - 1, \ldots, \tfrac{n}{2} - m') \in \ga'^*_\C$. We then define the representation $\pi'_{\lambda, \tau}$ of $G'_\infty$ to be the tensor product of $\pi'_{\lambda, \tau, v_0}$ and the $\pi'_{\tau,v}$ for the other real $v$.

The results of Section \ref{sec:arch-lift} show that the local factors of $\pi_\lambda$ and $\pi'_{\lambda, \tau}$ are matched by the theta correspondence, assuming they actually occur in it.

\subsection{Automorphic forms}

We introduce notation for sets of automorphic forms and representations on $\bG$ and $\bG'$.  Let $Q \geqslant 1$ be a constant, to be fixed in Lemma \ref{cmpt-period-asymp} below.  $\mathscr{A}^{\bG}(K,\nu,Q)$ is the set of isomorphism classes of automorphic representations $\pi$ of $\bG$ with the property that $\pi_f^{K_f} \neq 0$, and $\pi_\infty \simeq \pi_\lambda$ with $\|{\rm Im}\,\lambda-\nu\| \leqslant Q$.  We define $\widetilde{\mathscr{A}}^{\bG}(K,\nu,Q)$ to be the set of pairs $(\pi, E)$, where $\pi \in \mathscr{A}^{\bG}(K,\nu,Q)$ and $E$ is a homomorphism from $\pi$ into the space of automorphic forms on $\bG$.  Then there is a map $\widetilde{\mathscr{A}}^{\bG}(K,\nu,Q) \to \mathscr{A}^{\bG}(K,\nu,Q)$.  The fiber of this map over a representation $\pi$ is the multiplicity space of $\pi$, which we denote $m(\pi, \bG)$.

We introduce similar notation for $\tau$-spherical representations of $\bG'$.  We let $\mathscr{A}^{\bG'}_{\rm cusp}(K',\nu,Q, \tau)$ be the set of isomorphism classes of cuspidal automorphic representations $\pi'$ of $\bG'$ such that $(\pi_f')^{K_f} \neq 0$, and $\pi_\infty' \simeq \pi_{\lambda, \tau}$ with $\|{\rm Im}\,\lambda-\nu\| \leqslant Q$.  We define $\widetilde{\mathscr{A}}^{\bG'}_{\rm cusp}(K',\nu,Q, \tau)$, and the multiplicity spaces $m(\pi', \bG')$, in the same way as for $\bG$, noting that we now only consider homomorphisms from $\pi'$ into the space of cusp forms.

We next introduce spaces of automorphic forms $\mathcal{E}^{\bG}(K,\nu,Q)$ and $\mathcal{E}^{\bG'}_{\rm cusp}(K',\nu,Q, \tau)$ corresponding to $\mathscr{A}^{\bG}(K,\nu,Q)$ and $\mathscr{A}^{\bG'}_{\rm cusp}(K',\nu,Q, \tau)$. The space $\mathcal{E}^{\bG}(K,\nu,Q)$ is spanned by automorphic forms on $\bG$ that are right invariant under $K$, are eigenfunctions for the ring of differential operators $\mathbb{D}(G_{v_0}/K_{v_0})$ defined in Section \ref{sec:O-spherical}, and whose spectral parameter $\lambda$ satisfies $\|{\rm Im}\,\lambda-\nu\| \leqslant Q$.  Such automorphic forms can be viewed as functions on 
\[
Y=\bG(F)\backslash\bG(\A)/K.
\]
We then have
\[
\mathcal{E}^{\bG}(K,\nu,Q) = \bigoplus_{\pi \in \mathscr{A}^{\bG}(K,\nu,Q)} m(\pi, \bG) \otimes \pi_f^{K_f}.
\]
We define $\mathcal{E}^{\bG'}_{\rm cusp}(K',\nu,Q, \tau)$ in the analogous way, and likewise have
\[
\mathcal{E}^{\bG'}_{\rm cusp}(K',\nu,Q, \tau) = \bigoplus_{\pi' \in \mathscr{A}^{\bG'}_{\rm cusp}(K',\nu,Q, \tau) } m(\pi', \bG') \otimes (\pi_f')^{K_f'}.
\]

We may write these spaces classically as follows. We begin with the case of $\mathcal{E}^{\bG}(K,\nu,Q)$.  Recall from Section \ref{adelic2classical} the identification
\[
Y=\bigcup_{ i \in I}\Gamma_i\backslash\mathbb{H}^{n,m}.
\]
Using the notation from \eqref{eq:E-lambda-cusp-tau} we have
\be
\label{EGclassical}
\mathcal{E}^{\bG}(K,\nu,Q) = \sum_{ i \in I}\sum_{\substack{\mu\in\Lambda(\Gamma_i)\\ \|{\rm Im}\,\mu-\nu\|\leqslant Q}}\mathcal{E}_\mu(\Gamma_i\backslash\mathbb{H}^{n,m}),
\ee
where we have omitted the cuspidality condition from the notation as it is empty for anisotropic $\bG$.  To describe the case of $\mathcal{E}^{\bG'}_{\rm cusp}(K',\nu,Q, \tau)$, let $\mathsf{S}'=(\mathcal{H}^m)^{[F:\Q]}$, and let $\Gamma'=\bG'(\A_f)\cap K_f'$. Then, since $\bG'$ is split and simply connected, the genus of $\Gamma'$ is $1$, and we have
\be
\label{EG'classical}
\Gamma'\backslash \bG'(F_\infty)=\bG'(F)\backslash \bG'(\A)/K_f',\qquad \Gamma'\backslash \mathsf{S}'=\bG'(F)\backslash \bG'(\A)/K'.
\ee
We now adapt the notation from \eqref{eq:E-lambda-cusp-tau} and \eqref{eq:L2-cusp-tau} to write
\[
\mathcal{E}^{\bG'}_{\rm cusp}(K',\nu,Q, \tau) = \bigoplus_{\substack{\mu \in\Lambda_{\rm cusp}(\Gamma';\tau)\\ \|{\rm Im}\,\mu-\nu\|\leqslant Q}}\mathcal{E}^{\rm cusp}_\mu(\Gamma'\backslash\mathsf{S}',\tau).
\]

\subsection{Lifted forms}

For any $\pi \in \mathscr{A}^{\bG}(K,\nu,Q)$, we wish to define the subspace $m(\pi,\bG)^\Theta \subset m(\pi,\bG)$ of representations that arise as theta lifts of cusp forms on $\bG'$.  If any local factor of $\pi$ does not occur in the local theta correspondence, we set $m(\pi,\bG)^\Theta = 0$.  Otherwise, we know that $\pi' = \otimes \theta(\pi_v, W)$ satisfies the local conditions to lie in $\mathscr{A}^{\bG'}_{\rm cusp}(K',\nu,Q, \tau)$.  At archimedean places this follows from Section \ref{sec:arch-lift}, while at finite places this follows from results of Cossutta \cite[Proposition 2.11]{Cossutta} after choosing $K_f'$ small enough.  The span of the global lifts $\Theta(\widetilde{\pi}', V)$, where $\widetilde{\pi}'$ runs over occurrences of $\pi'$ in the space of cusp forms on $\bG'$, is a subspace of the $\pi$-isotypic space on $\bG$.  As such, it has the form $m(\pi,\bG)^\Theta \otimes \pi$ for a subspace $m(\pi,\bG)^\Theta \subset m(\pi, \bG)$ with $\dim m(\pi,\bG)^\Theta \leqslant \dim m(\pi', \bG')$.  In the same way, the orthogonal compliment of the space of lifts in the $\pi$-isotypic subspace has the form $m(\pi,\bG)^{\Theta^\perp} \otimes \pi$ for some $m(\pi,\bG)^{\Theta^\perp} \subset m(\pi, \bG)$.

We define $\widetilde{\mathscr{A}}^{\bG,\Theta}(K,\nu,Q)$ to be the subset of $\widetilde{\mathscr{A}}^{\bG}(K,\nu,Q)$ corresponding to the spaces $m(\pi,\bG)^\Theta$, and likewise for $\widetilde{\mathscr{A}}^{\bG,\Theta^\perp}(K,\nu,Q)$.  We define
\[
\mathcal{E}^{\bG,\Theta}(K,\nu,Q) = \bigoplus_{\pi \in \mathscr{A}^{\bG}(K,\nu,Q)} m(\pi, \bG)^\Theta \otimes \pi_f^{K_f}
\]
and
\[
\mathcal{E}^{\bG,\Theta^\perp}(K,\nu,Q) = \bigoplus_{\pi \in \mathscr{A}^{\bG}(K,\nu,Q)} m(\pi, \bG)^{\Theta^\perp} \otimes \pi_f^{K_f}
\]
to be the corresponding spaces of automorphic forms.

\subsection{The distinction argument}

For a real parameter $T>0$, recall the notion of $T$-regular from Section \ref{sec:gp-decomp}.
\begin{prop}

Suppose that $\nu$ is $Q$-regular.  If $\phi_\lambda \in \mathcal{E}^{\bG,\Theta^\perp}(K,\nu,Q)$, then 
\be
\label{periodvanish}
\sum_{[b]\in\bH(\A_f)\backslash \bG(\A_f)^{\rm int}/K_f} \textup{vol}(\bH(\A_f)bK_f) \beta_f( b^{-1} [y_0]) \mathscr{P}_\bH (R(b)\phi_\lambda) = 0.
\ee

\end{prop}

\begin{proof}

We may assume that $\phi_\lambda \in \widetilde{\pi}$ for some $\widetilde{\pi} \in \widetilde{\mathscr{A}}^{\bG,\Theta^\perp}(K,\nu,Q)$.  We prove the contrapositive, by assuming that the left-hand side of \eqref{periodvanish} is nonzero and deducing that $\phi_\lambda \notin \mathcal{E}^{\bG,\Theta^\perp}(K,\nu,Q)$.  Under this assumption, Corollary \ref{distinction} shows that $\phi'_\lambda = \Theta(\phi_\lambda, \alpha; W)$ is nonzero for some choice of $\alpha_\infty$, and hence that $\widetilde{\pi}' = \Theta( \widetilde{\pi}; W)$ is also nonzero.  Our assumption that $\nu$ is $Q$-regular implies that $\lambda \in i \ga^*$, and hence that $\widetilde{\pi}_{v_0}$ is tempered.  Lemma \ref{cuspidal-theta} then implies that $\widetilde{\pi}' \in \widetilde{\mathscr{A}}^{\bG'}_{\rm cusp}(K',\nu,Q, \tau)$.  We now have
\[
0 \neq \langle \phi'_\lambda, \phi'_\lambda \rangle = \langle \Theta(\phi_\lambda, \alpha; W), \phi'_\lambda \rangle = \langle \phi_\lambda, \Theta( \phi'_\lambda, \alpha; V) \rangle,
\]
which shows that $\phi_\lambda \notin \mathcal{E}^{\bG,\Theta^\perp}(K,\nu,Q)$, as required.
\end{proof}

\begin{lemma}
\label{lemma:Theta-map}

We have $\dim \mathcal{E}^{\bG,\Theta}(K,\nu,Q) \ll \left(\log (3+\|\nu\|)\right)^{[F:\Q]m}\beta_{\mathcal{H}^m}(\nu)$.

\end{lemma}

\begin{proof}

Bernstein's uniform admissibility theorem implies that $\dim \pi_f^{K_f} \ll 1$ for any $\pi \in \mathscr{A}^{\bG}(K,\nu,Q)$.  We therefore have
\begin{align*}
\dim \mathcal{E}^{\bG,\Theta}(K,\nu,Q) & = \sum_{\pi \in \mathscr{A}^{\bG}(K,\nu,Q)} \dim m(\pi, \bG)^\Theta \dim \pi_f^{K_f} \\
& \ll \sum_{\pi \in \mathscr{A}^{\bG}(K,\nu,Q)} \dim m(\pi, \bG)^\Theta \\
& \leqslant \sum_{\pi' \in \mathscr{A}^{\bG'}_{\rm cusp}(K',\nu,Q, \tau)} \dim m(\pi', \bG') \\
& \leqslant \mathcal{E}^{\bG'}_{\rm cusp}(K',\nu,Q, \tau).
\end{align*}
The bound now follows from the classical description \eqref{EG'classical} of $\mathcal{E}^{\bG'}_{\rm cusp}(K',\nu,Q, \tau)$ and Proposition \ref{Weyl-upper-bd}.
\end{proof}

We let $\mathcal{B}(\nu)$ be a choice of orthonormal basis for $\mathcal{E}^{\bG}(K,\nu,Q)$, and define
\[
\mathcal{M}_\bH(\nu) = \sum_{\phi_\mu\in\mathcal{B}(\nu)} \bigg|\sum_{[b]\in\bH(\A_f)\backslash\bG(\A_f)^{\rm int}/K_f} \textup{vol}(\bH(\A_f)bK_f) \beta_f( b^{-1} [y_0])\mathscr{P}_\bH(R(b)\phi_\mu)\bigg|^2.
\]
Note that $\mathcal{M}_\bH(\nu)$ is independent of the choice of basis $\mathcal{B}(\nu)$.

\begin{lemma}\label{cmpt-period-asymp}
If $Q$ is chosen large enough depending on $Y$, we have $\mathcal{M}_\bH(\nu)\gg\beta_{\mathbb{H}^{n,m}}(\nu)$.
\end{lemma}

\begin{proof}

We apply the classical description \eqref{EGclassical} of $\mathcal{E}^{\bG}(K,\nu,Q)$, and choose the basis $\mathcal{B}(\nu)$ to be a union of bases $\mathcal{B}_i(\nu)$ for the corresponding spaces
\[
\sum_{\substack{\mu\in\Lambda(\Gamma_i)\\ \|{\rm Im}\,\mu-\nu\|\leqslant Q}}\mathcal{E}_\mu(\Gamma_i\backslash\mathbb{H}^{n,m})
\]
for each classical quotient $\Gamma_i\backslash\mathbb{H}^{n,m}$. Thus, each member of $\mathcal{B}(\nu)$ is supported on a single connected component of $Y$.

Recall the finite set of points $\mathscr{X}$ in $Y$ described in \eqref{eq:defn-shifted-pts}. For each $i$ let $\mathscr{X}_i$ denote the subset of $\mathscr{X}$ consisting of those points lying in the $i$-th connected component $\Gamma_i\backslash\mathbb{H}^{n,m}$. We deduce from \eqref{eq:final-classical-period} that the contribution to $\mathcal{M}_\bH(\nu)$ coming from the lattice $\Gamma_i$ is
\[
\sum_{\phi_\mu\in\mathcal{B}_i(\nu)}\big|\sum_{p \in \mathscr{X}_i} c_p \phi_\mu(p) \big|^2,
\]
where $c_p$ is as in \eqref{eq:final-classical-period}.  After choosing $Q$ large enough, we may therefore apply Proposition \ref{local-Weyl} to show that the contribution of each $\mathcal{B}_i(\nu)$ to $\mathcal{M}_\bH(\nu)$ is $\gg\beta_{\mathbb{H}^{n,m}}(\nu)$, as required.
\end{proof}

We may now combine these ingredients to prove Theorem \ref{sup-thm}.  Let $Q$ be as in Lemma \ref{cmpt-period-asymp}, and assume that $\nu$ is $Q$-regular.  Choose a basis $\mathcal{B}(\nu)$ such that $\mathcal{B}(\nu) = \mathcal{B}(\nu)^\Theta \cup \mathcal{B}(\nu)^{\Theta^\perp}$, where $\mathcal{B}(\nu)^\Theta$ and $\mathcal{B}(\nu)^{\Theta^\perp}$ are bases for $\mathcal{E}^{\bG,\Theta}(K,\nu,Q)$ and $\mathcal{E}^{\bG,\Theta^\perp}(K,\nu,Q)$ respectively.  Because the forms in $\mathcal{B}(\nu)^{\Theta^\perp}$ make no contribution to $\mathcal{M}_\bH(\nu)$, Lemma \ref{cmpt-period-asymp} gives
\[
\sum_{\phi_\mu\in\mathcal{B}(\nu)^\Theta }\bigg|\sum_{[b]\in\bH(\A_f)\backslash\bG(\A_f)^{\rm int}/K_f} \textup{vol}(\bH(\A_f)bK_f) \beta_f( b^{-1} [y_0]) \mathscr{P}_\bH(R(b)\phi_\mu)\bigg|^2 \gg \beta_{\mathbb{H}^{n,m}}(\nu).
\]
Moreover, Lemma \ref{lemma:Theta-map} implies that there is some $\phi_\mu\in\mathcal{B}(\nu)^\Theta$ such that
\begin{multline*}
\bigg|\sum_{[b]\in\bH(\A_f)\backslash\bG(\A_f)^{\rm int}/K_f} \textup{vol}(\bH(\A_f)bK_f) \beta_f( b^{-1} [y_0]) \mathscr{P}_\bH(R(b)\phi_\mu)\bigg|^2 \\
\gg \left(\log (3+\|\nu\|)\right)^{-[F:\Q]m} \beta_{\mathbb{H}^{n,m}}(\nu) / \beta_{\mathcal{H}^m}(\nu).
\end{multline*}
This implies the same lower bound for $\| \phi_\mu \|_\infty^2$, which completes the proof.

\end{document}